\documentclass[12pt,english]{article}
\usepackage{hyperref}
\hypersetup{
    unicode = false,
    pdftoolbar = true,
    pdfmenubar = true,
    pdffitwindow = true,
    pdftitle = {StoMADS-PB},
    pdfauthor = {DZAHINI},
    pdfsubject = {StoMADS-PB},
    pdfnewwindow = true,
    pdfkeywords = {Stochastic, MADS, Progressive, Barrier},
    colorlinks = true,
    linkcolor = blue,
    citecolor = blue,
    filecolor = black,
    urlcolor = blue,
    breaklinks = true
}

\usepackage[english]{babel}
\usepackage[T1]{fontenc}
\usepackage[latin1]{inputenc}
\usepackage{amsfonts}
\usepackage{dsfont}
\usepackage{amsmath}
\usepackage{amsthm}
\usepackage{amssymb}
\usepackage{multirow}
\usepackage{epsfig}
\usepackage{soul}
\usepackage{colortbl}
\usepackage{geometry}
\usepackage{times}
\usepackage{graphicx}
\usepackage{mathrsfs}
\usepackage{frcursive}
\usepackage[ruled, vlined, linesnumbered]{algorithm2e}
\usepackage{color}
\usepackage{booktabs}
\usepackage[toc]{appendix}
\definecolor{gris3}{rgb}{0.3,0.3,0.3}

\definecolor{Green}{rgb}{0,.6,0}
\definecolor{Blue}{rgb}{0,0,1}
\definecolor{Red}{rgb}{1,0,0}
\definecolor{Gray}{rgb}{0.2,0.2,0.2}
\definecolor{Maroon}{rgb}{0.6,0.05,0.03}

\newcommand{\noir}{\color{black}}
\newcommand{\blc}{\color{black}}
\newcommand{\bl}{\color{black}}

\usepackage{tikz}
\usetikzlibrary{shapes}
\usetikzlibrary{fit, arrows, calc, positioning}
\usetikzlibrary{arrows, automata}
\usetikzlibrary{positioning}
\usetikzlibrary{shadows}

\tikzstyle{c} = [rectangle, rounded corners, draw]
\tikzset{box/.style={draw,rectangle,rounded  corners=0pt, fill=gray!15, minimum  width=3cm, minimum  height=2cm}}
\tikzset{loz/.style={draw,diamond, aspect=1.5, fill=gray!15}}

\tikzstyle{rnd}=[circle,fill=blue!25,minimum  width=5em]

\newtheorem{theorem}{Theorem}  %[subsection]
\newtheorem{proposition}{Proposition}

\newtheorem{corollary}{Corollary}
\newtheorem{lemma}{Lemma}
\newtheorem{remark}{Remark}

\newtheorem{assumption}{Assumption}

\newtheorem{definition}{Definition}

\newcommand{\U}{\mathcal{U}}

\newcommand{\Z}{\mathbb{Z}}
\newcommand{\R}{\mathbb{R}}
\newcommand{\Esp}{\mathbb{E}}
\newcommand{\rn}{\mathbb{R}^n}

\newcommand{\N}{\mathbb{N}}

\newcommand{\Vk}{\mathcal{V}^k}
\newcommand{\Pk}{\mathcal{P}^k}
\newcommand{\D}{\mathcal{D}}
\newcommand{\B}{\mathcal{B}}

\newcommand{\X}{\mathcal{X}}

\newcommand{\lp}{{\mathbb{L}}^p(\Omega,\mathcal{G},\mathbb{P})}

\newcommand{\pb}{\mathbb{P}}
\newcommand{\M}{\mathcal{M}}

\newcommand{\dpl}{\delta^k_p}
\newcommand{\Dp}{\Delta^k_p}
\newcommand{\Dm}{\Delta^k_m}
\newcommand{\dm}{\delta^k_m}

\newcommand{\fd}{f}

\newcommand{\Mk}{\mathcal{M}^k}
\newcommand{\hok}{h^k_0}
\newcommand{\hokx}{h^k_0(x^k)}
\newcommand{\xkf}{x^k_{\textnormal{feas}}}
\newcommand{\xhatf}{\hat{x}_{\textnormal{feas}}}
\newcommand{\Xhatf}{\hat{X}_{\textnormal{feas}}}
\newcommand{\Xkf}{X^k_{\textnormal{feas}}}
\newcommand{\Xkft}{X^{k\vee T}_{\textnormal{feas}}}
\newcommand{\xkft}{x^{k\vee t}_{\textnormal{feas}}}
\newcommand{\Xkfunt}{X^{(k+1)\vee T}_{\textnormal{feas}}}
\newcommand{\xki}{x^k_{\textnormal{inf}}}
\newcommand{\Xki}{X^k_{\textnormal{inf}}}

\newcommand{\xkfun}{x^{k+1}_{\textnormal{feas}}}
\newcommand{\Xkfun}{X^{k+1}_{\textnormal{feas}}}

\newcommand{\xzeroi}{x^0_{\textnormal{inf}}}
\newcommand{\xkiun}{x^{k+1}_{\textnormal{inf}}}
\newcommand{\Xkiun}{X^{k+1}_{\textnormal{inf}}}
\newcommand{\hsk}{h^k_s}
\newcommand{\domf}{\prec_{f;\varepsilon}}
\newcommand{\domh}{\prec_{h;\varepsilon}}
\newcommand{\hskx}{h^k_s(x^k+s^k)}
\newcommand{\fok}{f^k_0}
\newcommand{\fokx}{f^k_0(x^k)}
\newcommand{\fsk}{f^k_s}
\newcommand{\fskx}{f^k_s(x^k+s^k)}
\newcommand{\uok}{u^k_0}
\newcommand{\uokx}{u^k_0(x^k)}
\newcommand{\usk}{u^k_s}
\newcommand{\uskx}{u^k_s(x^k+s^k)}
\newcommand{\lok}{\ell^k_0}
\newcommand{\lokx}{\ell^k_0(x^k)}
\newcommand{\lskx}{\ell^k_s(x^k+s^k)}
\newcommand{\lsk}{\ell^k_s}
\newcommand{\Lok}{L^k_0}
\newcommand{\Lokx}{L^k_0(x^k)}
\newcommand{\Lsk}{L^k_s}
\newcommand{\Lskx}{L^k_s(x^k+s^k)}
\newcommand{\csk}{c^k_{j,s}}
\newcommand{\cskx}{c^k_{j,s}(x^k+s^k)}
\newcommand{\cok}{c^k_{j,0}}
\newcommand{\cokx}{c^k_{j,0}(x^k)}
\newcommand{\ds}{\displaystyle}
\newcommand{\hkm}{h^k_{\max}}

\newcommand{\Hok}{H^k_0}
\newcommand{\Hokx}{H^k_0(X^k)}
\newcommand{\Hsk}{H^k_s}
\newcommand{\Hskx}{H^k_s(X^k+S^k)}
\newcommand{\Uok}{U^k_0}
\newcommand{\Uokx}{U^k_0(x^k)}
\newcommand{\Usk}{U^k_s}
\newcommand{\Uskx}{U^k_s(x^k+s^k)}
\newcommand{\Csk}{C^k_{j,s}}
\newcommand{\Cskx}{C^k_{j,s}(X^k+S^k)}
\newcommand{\Cok}{C^k_{j,0}}
\newcommand{\Cokx}{C^k_{j,0}(X^k)}
\newcommand{\dmax}{\tau^{-\hat{z}}}
\newcommand{\Fok}{F^k_0}
\newcommand{\Fokx}{F^k_0(X^k)}
\newcommand{\Fsk}{F^k_s}
\newcommand{\Fskx}{F^k_s(X^k+S^k)}
\newcommand{\efPB}{\varepsilon}
\newcommand{\fcf}{\mathcal{F}^{C\cdot F}_{k-1}}

\newcommand{\fij}{\mathcal{F}^{I\cdot J}_{k-1}}

\newcommand{\ijk}{\mathds{1}_{J_k}}

\newcommand{\itFini}{\mathds{1}_{\{T<+\infty\}}}
\newcommand{\itkless}{\mathds{1}_{\{T\leq k\}}}
\newcommand{\itksup}{\mathds{1}_{\{T> k\}}}
\newcommand{\iikj}{\mathds{1}_{I_k^j}}
\newcommand{\iji}{\mathds{1}_{J_i}}
\newcommand{\ijkc}{\mathds{1}_{\bar{J_k}}}

\newcommand{\iik}{\mathds{1}_{I_k}}
\newcommand{\iii}{\mathds{1}_{I_i}}
\newcommand{\iikc}{\mathds{1}_{\bar{I_k}}}
\newcommand{\iimp}{\mathds{1}_{\mathcal{I}}}
\newcommand{\idf}{\mathds{1}_{\mathcal{D}_f}}
\newcommand{\idh}{\mathds{1}_{\mathcal{D}_h}}
\newcommand{\iuns}{\mathds{1}_{\mathcal{U}}}
\newcommand{\abs}[1]{\left\lvert#1\right\rvert}
\newcommand{\accolade}[1]{\left\lbrace#1\right\rbrace}
\newcommand{\E}[1]{\mathbb{E}\left(#1\right)}
\newcommand{\pr}[1]{\mathbb{P}\left(#1\right)}
\newcommand{\var}[1]{\mathbb{V}\left(#1\right)}

\newcommand{\normp}[1]{{\left\lVert#1\right\rVert}_{p}}

\newcommand{\norminf}[1]{{\left\lVert#1\right\rVert}_{\infty}}

\title{Constrained stochastic blackbox optimization using a progressive barrier and probabilistic estimates}

\author{
	\href{mailto:kwassi-joseph.dzahini@polymtl.ca}{Kwassi Joseph Dzahini}\thanks{GERAD and DÈpartement de MathÈmatiques et de GÈnie Industriel, Polytechnique MontrÈal, C.P. 6079, Succ. Centre-ville, MontrÈal, QuÈbec H3C 3A7, Canada (\href{https://www.gerad.ca/fr/people/kwassi-joseph-dzahini}{www.gerad.ca/fr/people/kwassi-joseph-dzahini}, \href{https://www.gerad.ca/Sebastien.Le.Digabel}{www.gerad.ca/Sebastien.Le.Digabel}).
	}
	\and
	\href{mailto:michael.kokkolaras@mcgill.ca}{Michael Kokkolaras}\thanks{GERAD and McGill University, Mechanical Engineering Department, 845 Rue Sherbrooke Ouest, MontrÈal, QuÈbec H3A 0G4, Canada (\href{https://www.mcgill.ca/mecheng/people/staff/michael-kokkolaras}{www.mcgill.ca/mecheng/people/staff/michael-kokkolaras}).
	}
	\and
	\href{mailto:Sebastien.Le.Digabel@gerad.ca}{S\'ebastien Le Digabel}\footnotemark[1]
}

\begin{document}
\maketitle
%------------------------------------------------%

\vspace*{-0.5cm}

\noindent
{\bf Abstract:} 
This work introduces the StoMADS-PB algorithm for constrained stochastic blackbox optimization, which is an extension of the mesh adaptive direct-search (MADS) method originally developed for deterministic blackbox optimization under general constraints. The values of the objective and constraint functions are provided by a noisy blackbox, i.e., they can only be computed with random noise whose distribution is unknown. As in MADS, constraint violations are aggregated into a single constraint violation function. Since all functions values are numerically unavailable, StoMADS-PB uses estimates and introduces so-called probabilistic bounds for the violation. Such estimates and bounds obtained from stochastic observations are required to be accurate and reliable with high but fixed probabilities. The proposed method, which allows intermediate infeasible iterates, accepts new points using sufficient decrease conditions and imposing a threshold on the probabilistic bounds. Using Clarke nonsmooth calculus and martingale theory, Clarke stationarity convergence results for the objective and the violation function are derived with probability one.

\clearpage
\newpage
%------------------------------------------------%
%\tableofcontents
%------------------------------------------------%
%------------------------------------------------%
\clearpage
\newpage

\section{Introduction} 
%Any process which returns an output when provided an input and whose inner workings are analytically unavailable is called a blackbox
Blackbox optimization (BBO) {\noir considers the development and analysis of algorithms designed for objectives and constraints functions that are given by a process called a blackbox} which returns an output when provided an input but whose inner workings are analytically unavailable~\cite{AuHa2017}. Mesh adaptive direct-search (MADS)~\cite{AuDe2006,AuDe09a} with progressive barrier (PB) is an algorithm for deterministic BBO. This work considers the following constrained stochastic BBO problem

\begin{equation}\label{problemPB}
%{\noirc \left. \left. \underset{x\in \R^n}{\min}\  \right\lbrace f(x)=\Esp_\Theta\left[f_\Theta(x)\right] \right\rbrace}
\underset{x\in \D}{\min}\   f(x) 
%\quad \text{where}\quad 
\end{equation}
where $\D=\{x\in \X: c(x)\leq 0\}\subset\rn$ is the feasible region, $c = (c_1,c_2,\dots, c_m)^{\top}$, $\X$ is a subset of $\rn$, $f(x)=\Esp_{\Theta_0}\left[f_{\Theta_0}(x)\right]$ with  $f\colon\X \mapsto \R$, and $c_j(x)=\Esp_{\Theta_j}\left[c_{\Theta_j}(x)\right]$ with  $c_j\colon\X \mapsto \R$ for all $j\in J:=\{{\noir 1},2,\dots,m\}$. $\Esp_{\Theta_{\noir j}}$ denotes the expectation with respect to the random variable $\Theta_j$ for all ${\noir j \in J\cup\{0\}}$, which are supposed to be independent with unknown {\noir possibly different} distributions. 
%$\Esp_{\Theta_i}$ denotes the expectation with respect to $\Theta_i$.
$f_{\Theta_0}(\cdot)$~denotes the noisy computable version of the numerically unavailable objective function $f(\cdot)$, while for all $j\in J$, $c_{\Theta_j}(\cdot)$ denotes the noisy computable version of the numerically unavailable {\noir constraint} $c_j(\cdot)$. {\noir Note that the noisy objective function $f_{\Theta_0}$ and the constraints $c_{\Theta_j}, j\in J,$ are typically the outputs of a blackbox}.
%More precisely, $\tilde{f}(\cdot;{\Theta_0})$ and constraints $\tilde{c}(\cdot;{\Theta_j})$ are provided by the same stochastically noisy blackbox.  For the sake of simplicity in the presentation, $\tilde{f}(\cdot;{\Theta_0})$ and $\tilde{c}(\cdot;{\Theta_j})$ will be respectively abbreviated as $f_{\Theta_0}(\cdot)$ and $c_{\Theta_j}(\cdot)$ for the remainder of the manuscript. 
By means of some useful terminology, constraints that must always be satisfied, such as those defining $\X$, are differentiated from those that need only to be satisfied at the solution, such as $c(x)\leq 0$. The former will be called {\it unrelaxable} {\noir non-quantifiable} constraints and the latter, {\it relaxable} {\noir quantifiable} constraints~\cite{LedWild2015}.

Solving stochastic blackbox optimization problems such as Problem~\eqref{problemPB}, which often arise in signal processing and machine learning~\cite{curtis2020adaptive}, has {\bl recently} been {\bl a topic} of intense research. {\bl Most} methods for solving such problems {\bl borrow ideas from the stochastic gradient} method~\cite{RoMo1951}.
%For the scope and motivation of much of these efforts, see, e.g.,~\cite{curtis2020adaptive} and references therein. 
Several works have also attempted to transfer ideas from deterministic DFO methods to the stochastic {\bl context}. {\noir However, most of such proposed methods are restricted to unconstrained {\bl optimization}. Indeed,} after~\cite{barton1996nelder} which {\noir is} among the first to propose a stochastic variant of the deterministic Nelder-Mead (NM) method~\cite{NeMe65a},~\cite{AnFe01a} also considered the optimization of functions whose evaluations are subject {\bl to random} noise and proposed an algorithm which is shown to have convergence properties, based on Markov chain theory~\cite{durrett2010probability}. Another stochastic variant of NM was recently proposed in~\cite{Ch2012} and was proved to have global convergence properties with probability one. Using elements from~\cite{BaScVi2014,LaBi2016},~\cite{chen2018stochastic} proposed STORM, a trust-region algorithm designed for stochastic optimization problems, with almost sure global convergence results.
%utilizing random models of an objective function and estimates of function values to gauge the progress made in the proposed method. Under the assumption that such models and estimates are sufficiently accurate with large but fixed probabilities conditioned to the past, an almost sure global convergence result to a first-order stationary point of STORM has been derived with the help of martingale theory. 
%Then, using an assumption of sufficiently accurate stochastic gradients, 
%STORM is shown in~\cite{blanchet2019convergence} to have first-order and second-order complexity bounds.
%which also derived a  complexity bounds for the algorithm under additional assumptions. 
Many other research{\noir es} that extend the traditional deterministic trust-region method to stochastic setting have been conducted in~\cite{curtis2017stochastic,shashaani2018astro}. In~\cite{paquette2018stochastic}, a classical backtracking Armijo line search method~\cite{Armi66a} has been adapted to {\noir the} stochastic optimization setting and was shown to have first-order complexity bounds.
%By assuming that the function values of the objective function and those of its derivatives can be estimated within a prescribed accuracy with sufficiently large but fixed probabilities, the proposed stochastic line search method was shown to have first-order complexity bounds in the nonconvex, convex and strongly convex cases. 
Robust-MADS, a kernel smoothing-based variant of MADS~\cite{AuDe2006}, {\bl was proposed} in~\cite{AudIhaLedTrib2016} to approach the minimizer of an objective function whose values can only be computed with a random noise. It was shown to possess {\bl zeroth-order}~\cite{AuDeLe07} convergence properties. Another stochastic variant of MADS was proposed in~\cite{AlAuBoLed2019} for BBO, where the noise corrupting the blackbox was supposed to be Gaussian. Convergence results of the proposed method have been derived, making use of statistical inference techniques. \cite{audet2019stomads}~proposed another stochastic optimization approach using an algorithmic framework similar to that of MADS. StoMADS uses estimates of function values obtained from stochastic observations. 
%instead of the exact deterministic computable version of the objective function which is numerically unavailable. 
By assuming that such estimates satisfy a variance condition and are sufficiently accurate with a large but fixed probability conditioned to the past, a Clarke~\cite{Clar83a} stationarity convergence result of StoMADS has been derived with probability one, using martingale theory. A general framework for stochastic {\bl directional direct-search}~\cite{CoScVibook} methods was introduced in~\cite{dzahini2020expected} with expected complexity analysis. 

{\bl All the} above stochastic optimization methods are restricted to unconstrained problems and most of them use estimated gradient information when seeking for an optimal solution. {\bl When} the gradient does not exist or is computationally expensive to estimate, heuristics such as simulated annealing methods, genetic algorithms~\cite{lacksonen2001empirical}, and tabu/scatter search~\cite{klassen2009improving}, are also used for problems with noisy constraints but do not present {\bl any convergence theory}. 
%DFO approaches seem to be the best to solve BBO problems~\cite{wang2018constrained}. For problems with noisy constraints, such DFO approaches include simulated annealing methods, genetic algorithms~\cite{lacksonen2001empirical} and tabu/scatter search~\cite{klassen2009improving}. 
%~\cite{lacksonen2001empirical} compared the performance of two direct-search methods, a simulated annealing and a genetic algorithm and realized that the latter method was the most robust for twenty five tested problems. 
%In~\cite{klassen2009improving}, a tabu/scatter search optimization approach has been developed in order to determine optimal rules for a stochastic appointment scheduling problem. 
{\bl Surrogate model} based methods for constrained stochastic BBO have also been a topic of intense research{\bl ,} {\bl including} the response surface methodology with stochastic constraints~\cite{angun2009response} developed for expensive simulation.
%, and an artificial neural network algorithm~\cite{mohammad2014artificial}. 
In~\cite{AuMa2017}, the capabilities of the deterministic constrained trust-region algorithm NOWPAC~\cite{AuMa2014} are generalized for the optimization of blackboxes with inherently noisy evaluations of {\bl the objective and constraint functions}. To mitigate the noise in the latter functions evaluations, the resulting gradient-free method SNOWPAC utilizes Gaussian process surrogate combined with local fully linear surrogate models. 
%Then, the efficiency of the proposed method and the accuracy of its optimal solutions was demonstrated by means of several benchmark results. 
Another surrogate-based approach that has gained in increasing popularity in various research fields {\bl is Kriging}, also known as {\noir Bayesian optimization}~\cite{mockus2012bayesian}. 
%This type of approach uses a stochastic process model which can be evaluated relatively cheaply to approximate a computationally expensive simulation and calculate some infill criterion in order to balance exploitation and exploration when seeking for a new sampling point~\cite{kleijnen2014simulation,vzilinskas2016stochastic}. The first stochastic process model based algorithms include the Bayesian algorithm~\cite{movckus1975bayesian} and were later made popular by the efficient global optimization method~\cite{JoScWe1998} whose infill sampling criterion is an expected improvement function.
Various Bayesian optimization methods for constrained stochastic {\bl BBO have} been demonstrated to be efficient in practice~\cite{letham2019constrained,wang2018constrained}. 

Developing direct-search methods for BBO has received renewed interest since such methods generally known to be reliable and robust in practice~\cite{Audet2014a}, appear to be the most promising approach in most of real applications where the gradient does not exist or is computationally expensive to estimate. %It is even highlighted in~\cite{AuMa2017} that pattern search methods~\cite{AuDe03a,LeTo99a,LeTo00a,Torc97a} may be used to optimize blackboxes that are corrupted by a noise ``{\it small enough}'', since ``{\it avoiding gradient estimations makes these methods less sensitive to noisy blackbox evaluations}''. 
However, there is relatively scarce research on developing direct-search methods for constrained stochastic BBO, especially when noise is present in the constraint functions. A pattern search and implicit filtering algorithm (PSIFA)~\cite{diniz2019pattern,diniz2020applying} was recently developed for linearly constrained problems with a noisy objective function, and was shown to have global convergence properties. A class of direct-search methods for solving smooth linearly constrained problems was also studied in~\cite{GraRoViZh2019} but even though using a {\bl probabilistic feasible descent} based approach, this work assumes the objective and constraints function values to be exactly computed without noise.

The present work introduces StoMADS-PB, a stochastic variant of {\noir the mesh adaptive direct-search with progressive barrier}~\cite{AuDe09a}, using elements from~\cite{AuDe2006,AuDe09a,audet2019stomads,BaScVi2014,chen2018stochastic,paquette2018stochastic} and is{\noir ,} to the best of our knowledge{\noir ,} the first to propose a {\bl directional direct-search}~\cite{CoScVibook} stochastic BBO algorithm, capable to handle general noisy constraints without requiring any feasible initial point. Its main contribution is the analysis of the resulting new framework with fully supported theoretical results.
%designed for deterministic constrained BBO where constraints are handled using a progressive barrier (PB) approach.
%It is analyzed using tools from~\cite{AuDe2006,AuDe09a,audet2019stomads,BaScVi2014,chen2018stochastic,paquette2018stochastic} and aims to solve the following constrained stochastic BBO problem
%This work introduces StoMADS-PB, a direct-search stochastic BBO algorithm designed for noisy objective functions and general noisy constraints which are treated using a PB approach adapted from~\cite{AuDe09a}.
StoMADS-PB uses no gradient information to find descent directions or improve feasibility compared to prior work. Rather, it uses so-called {\bl probabilistic estimates}~\cite{chen2018stochastic} of the objective and {\bl constraint} function values and also introduces {\bl probabilistic bounds} on a {\bl constraint violation} function values. The reliability of such bounds is assumed to hold with a high but fixed probability. Moreover, although no distributions are assumed for the estimates and no assumption is made about the way they are generated, they are required to be sufficiently accurate with large but fixed probabilities and satisfy some variance conditions.

%This research is to the best of our knowledge the first to propose a {\it directional direct-search}~\cite{CoScVibook} stochastic BBO algorithm, capable to handle general noisy constraints without requiring any feasible initial guess. 
%Moreover, the ``{\it probabilistically reliable bounds}'' introduced by means of existing so-called {\it probabilistic estimates}~\cite{chen2018stochastic} play a crucial role in the feasibility improvement of the proposed method and its convergence analysis.
%Using existing so-called {\it probabilistic estimates} introduced in~\cite{chen2018stochastic}, it is also important to emphasize that the ``{\it probabilistically reliable bounds}'' introduced in the present work plays a crucial role in the feasibility improvement of the proposed method and its convergence analysis. 

The manuscript is organized as follows. Section~\ref{sec2PB} presents the general framework of the proposed StoMADS-PB algorithm. Section~\ref{sec3} explains how the proposed method results in a stochastic process and discusses requirements on random estimates to guarantee convergence. It also shows how such estimates can be constructed in practice. Section~\ref{convAnalysisPB} presents the main convergence results. {\bl Computational} results are reported in Section~\ref{sec5} followed by a discussion and suggestions for future work. Additional results are provided as an annex.

%As highlighted in~\cite{wang2018constrained}, there is a relatively scarce research on developing Kriging-based methods for constrained stochastic BBO. ~\cite{gelbart2014bayesian}

% There exist several optimization methods to solve constrained stochastic BBO problems. 

% 

%$\Theta$ is a random variable obeying some unknown distribution, 
%{\blc In other words, the {\blc value} of $f(x)$ can only be computed numerically with some random noise}. 
%{\blc In the convergence analysis {\blc of Section~\ref{convAnalysisPB}}, the objective function} is assumed to be {\blc locally Lipschitz} continuous and bounded from below.  

%------------------------------------------------%
\section{The StoMADS-PB algorithm}\label{sec2PB}
StoMADS-PB is based on an algorithmic framework similar to that of MADS with PB~\cite{AuDe09a}.
%a ``{\it progressive barrier}'' based constraint-handling approach for general constraints designed for deterministic blackbox optimization. 
For the needs of the convergence analysis of Section~\ref{convAnalysisPB}, deterministic constraint violations are aggregated into a single function $h$ called {\bl the} {\bl constraint violation} function, defined using the $\ell_1$-norm for needs of  convergence studies as opposed to~\cite{AuDe09a} where an $\ell_2$-norm has been favored
\begin{equation}\label{violationFunct}
h(x) := \left\{
\begin{array}{ll}
\ds{\sum_{j=1}^m }\max\{c_j(x),0 \} & \mbox{if } x\in\X \\
+\infty & \mbox{otherwise.}
\end{array}
\right. \nonumber
\end{equation}
According to this definition, $h:\rn\rightarrow\R\cup\{+\infty\}$  and $x\in\D$, i.e., $x$ is feasible with respect to the relaxable constraints if and only if $h(x)=0$. Moreover, if $0<h(x)<+\infty$, then $x$ is called infeasible and satisfies the unrelaxable constraints but not the relaxable ones. 

In MADS with PB, feasibility improvement is achieved by decreasing $h$, specifically by comparing its function value at a current point $x^k$ to that of a trial point $x^k+s^k$, where $s^k$ denotes a direction around $x^k$.
%is sought by imposing a threshold on the allowed constraint violation. Then, any infeasible point that exceeds it is simply rejected out of hand and this threshold is progressively tightened adaptively as the iteration evolves. Note that such a tightening is achieved by  seeking for a decrease in the function $h$, specifically by comparing its function values at a current point $x^k$ of a given iteration $k$ and a trial neighboring point $x^k+s^k$, where $s^k$ denotes a direction around $x^k$ that will be defined later. 
Likewise, to decrease $f$, MADS with PB uses objective function values since they are available in {\bl the} deterministic setting.

The main challenge {\bl here} is to guarantee for StoMADS-PB such decreases as well in $f$ as in $h$ whereas their function values are unavailable numerically, using only {\bl information} provided by the noisy blackbox outputs $f_{\Theta_0}$ and $c_{\Theta_j}$, $j\in J$. This section {\bl shows} how {\bl this} can be achieved, making use of so called $\efPB$-accurate estimates introduced in~\cite{chen2018stochastic} and then {\bl presents} the general framework of the proposed method.

\subsection{Feasibility and objective function improvements}
At iteration $k$, let $x^k$ and $x^k+s^k$ be two points of $\X$. Since the constraint function values $c_j(x^k)$ and $c_j(x^k+s^k)$, $j\in J=\{1,2,\dots,m\}$, are numerically unavailable, their corresponding estimates are respectively constructed using evaluations of the noisy blackbox outputs $c_{\Theta_j}$, $j\in J$. In general for the remainder of the manuscript, unless otherwise stated, given a function $g:\X\rightarrow\R$, an estimate of $g(x^k)$ is denoted by $g_0^k(x^k)$ (or simply by $g_0^k$ if there is no ambiguity) while that  of $g(x^k+s^k)$ is denoted by $g_s^k(x^k+s^k)$ or $g_s^k$. In StoMADS-PB, the violations of the estimates $\cokx$ and $\cskx$ of $c_j(x^k)$ and $c_j(x^k+s^k)$, respectively, are aggregated in so-called {\it estimated violations} $\hokx$ and $\hskx$ defined as follows
\begin{eqnarray}\label{EstimViolation}  
\hok(x^k) &=& \left\{
\begin{array}{ll}
\ds{\sum_{j=1}^m }\max\left\lbrace \cokx,0  \right\rbrace & \mbox{if } x^k\in\X \\
+\infty & \mbox{otherwise}
\end{array}
\right.\\ \text{and}\quad  \hsk(x^k+s^k) &=& \left\{
\begin{array}{ll}
\ds{\sum_{j=1}^m }\max \left\lbrace \cskx,0 \right\rbrace & \mbox{if } x^k+s^k\in\X \\
+\infty & \mbox{otherwise.}
\end{array}
\right.
\end{eqnarray}
In order for such estimated constraint violations to be reliable enough to determine whether $h(x^k+s^k)<h(x^k)$ or not, the estimates $\cokx$ and $\cskx$ need to be sufficiently accurate. The following definition similar to that of~\cite{audet2019stomads} is adapted from~\cite{chen2018stochastic}.
\begin{definition}\label{epsilonAccDef}
	Let $\efPB>0$ be a fixed constant and $\{\dpl\}_{k\in\N}$ be a sequence of nonnegative real numbers. For a given function $g\colon\X \mapsto \R$ and $y^k\in \X$, let $g^k$ be an estimate of $g(y^k)$. Then $g^k$ is said to be an $\efPB$-accurate estimate of $g(y^k)$ for the given $\dpl$, if 
	\begin{equation}
	\abs{g^k-g(y^k)}\leq \efPB(\dpl)^2.\nonumber
	\end{equation}
\end{definition}
As in~\cite{audet2019stomads}, the role of $\dpl$ will be played by the so-called {\it poll size} parameter introduced  in Section~\ref{descriptionAlgo}. The following result provides bounds on $h(x^k)$ and $h(x^k+s^k)${\bl ,} respectively, which will allow{\bl ,} in Proposition~\ref{successh}{\bl ,} to guarantee a decrease in the constraint violation function $h$ by means of a sufficient decrease condition on the estimated violations $\hok$ and $\hsk$.
\begin{proposition}\label{boundh}
	Let $\cok$ and $\csk$ be $\efPB$-accurate estimates of $c_j(x^k)$ and $c_j(x^k+s^k)$, respectively, with $x^k$ and $x^k+s^k\in\X$. Then the followings hold:
	\begin{equation}
	\lok(x^k):=\sum_{j=1}^{m}\max\left\lbrace \cok-\efPB(\dpl)^2,0 \right\rbrace \leq h(x^k)\leq \sum_{j=1}^{m}\max\left\lbrace \cok+\efPB(\dpl)^2,0 \right\rbrace=:\uok(x^k)\label{epspr}
	\end{equation}
	and 
	\begin{equation}
	\lsk(x^k+s^k):=\sum_{j=1}^{m}\max\left\lbrace \csk-\efPB(\dpl)^2,0 \right\rbrace \leq h(x^k+s^k)\leq \sum_{j=1}^{m}\max\left\lbrace \csk+\efPB(\dpl)^2,0 \right\rbrace=:\usk(x^k+s^k)\nonumber
	\end{equation}
\end{proposition}
\begin{proof}
	The result is shown for $h(x^k)$ but the proof for $h(x^k+s^k)$ is the same. Since $\cok$ is an $\efPB$-accurate estimate of $c_j(x^k)$ for all {\bl $j\in J$}, then it follows from Definition~\ref{epsilonAccDef} that
	\begin{equation}
	\cok-\efPB(\dpl)^2\leq c_j(x^k)\leq \cok+\efPB(\dpl)^2, \quad\text{for all}\  {\bl j\in J},\nonumber
	\end{equation}
	which implies that 
	\begin{equation}\label{cjeps}
	\max\left\lbrace \cok-\efPB(\dpl)^2,0 \right\rbrace \leq \max\accolade{c_j(x^k),0}\leq \max\left\lbrace \cok+\efPB(\dpl)^2,0 \right\rbrace.
	\end{equation}
	Finally, summing each term of~\eqref{cjeps} from {\bl $j=1$ to $m$} leads to~\eqref{epspr}.
\end{proof}
%\begin{remark}\label{rem1}
%Since the unrelaxable constraints that define $\X$ are not assumed to be stochastic, then trial points satisfying $x^k\not\in\X$ are clearly identified by {\em StoMADS-PB} and are such that $h(x^k)=+\infty$ even though the constraints $c_j$ function values defining $h$ are unavailable numerically. Thus for those trial points, the upper bound $\uokx$ is assumed to equal~$+\infty$. This will allow in particular the initialization in~Algorithm~\ref{algoPB} of the barrier threshold $\hkm$ that will be introduced later.
%\end{remark}
\begin{definition}
	The estimates $\lokx$ and $\uokx$ of {\em Proposition~\ref{boundh}}{\bl ,} satisfying $\lok(x^k)\leq h(x^k)\leq \uok(x^k)${\bl ,} are said to be $\efPB$-reliable bounds for $h(x^k)$. Similarly, the estimates $\lskx$ and $\uskx$ satisfying $\lskx \leq h(x^k+s^k)\leq\uskx$ are said to be $\efPB$-reliable bounds for $h(x^k+s^k)$.
\end{definition}

%\begin{definition}
%For all $j=1,2,\dots,m$, assume that $\cok$ and $\csk$ are $\efPB$-accurate estimates of $c_j(x^k)$ and $c_j(x^k+s^k)$, respectively, with $x^k$ and $x^k+s^k\in\X$. Then the corresponding estimated violations $\hokx$ and $\hskx$ are called $\efPB$-reliable.
%\end{definition}
The following result {\bl provides} sufficient information to identify a decrease in $h$ {\bl and} will {\bl be also} useful to determine an iteration type in Section~\ref{descriptionAlgo}.  
\begin{proposition}\label{successh}
	Let $\lokx$ and $\uokx$ be $\efPB$-reliable bounds for $h(x^k)$, and let $\lskx$ and $\uskx$ be $\efPB$-reliable bounds for $h(x^k+s^k)$. Let $\hok$ and $\hsk$ be the estimated constraint violations at $x^k$ and $x^k+s^k\in\X${\bl ,} respectively. Let $\gamma>2$ be {\bl a constant}. Then the following holds:
	\begin{equation}\label{five}
	\text{if }\ \hsk -\hok \leq -\gamma m\efPB(\dpl)^2,\ \ \text{then}\ \ h(x^k+s^k)-h(x^k)\leq -(\gamma-2) m\efPB(\dpl)^2<0.
	\end{equation}
\end{proposition}
\begin{proof}
	%Since $h(x^k)$ and $h(x^k+s^k)$ are $\efPB$-reliable bounded respectively, then 
	It follows from Proposition~\ref{boundh} that 
	\begin{equation}\label{hdif}
	h(x^k+s^k)-h(x^k)\leq \sum_{j=1}^{m}\max\left\lbrace \csk+\efPB(\dpl)^2,0 \right\rbrace-\sum_{j=1}^{m}\max\left\lbrace \cok-\efPB(\dpl)^2,0 \right\rbrace.
	%\\
	%\leq \left(\sum_{j=1}^{m}\max\left\lbrace \csk,0 \right\rbrace +m\efPB(\dpl)^2 \right)-
	\end{equation}
	By noticing that \[\sum_{j=1}^{m}\max\left\lbrace \csk+\efPB(\dpl)^2,0 \right\rbrace\leq  \sum_{j=1}^{m}\max\left\lbrace \csk,0 \right\rbrace +m\efPB(\dpl)^2=\hsk+ m\efPB(\dpl)^2 \]
	%and that
	\[\sum_{j=1}^{m}\max\left\lbrace \cok-\efPB(\dpl)^2,0 \right\rbrace\geq  \sum_{j=1}^{m}\max\left\lbrace \cok,0 \right\rbrace -m\efPB(\dpl)^2=\hok- m\efPB(\dpl)^2,\]
	then it follows from~\eqref{hdif} that
	\begin{equation}
	h(x^k+s^k)-h(x^k)\leq \hsk-\hok + 2m\efPB(\dpl)^2\leq -(\gamma-2) m\efPB(\dpl)^2,\nonumber
	\end{equation}
	where the last inequality follows from the assumption that $\hsk -\hok \leq -\gamma m\efPB(\dpl)^2$. The proof is complete by noticing that $\gamma>2$.
\end{proof}

As in~\cite{AuDe09a}, the present research also introduces a nonnegative barrier threshold $\hkm=\uok(\xki)$, where $\xki$ is a so-called $\efPB$-infeasible solution. Definition~\ref{feasInfeaspoints} presents $\efPB$-infeasible points and the updating rules of $\xki$ is presented in Section~\ref{descriptionAlgo}. While $\xki$ is updated at the end of each iteration of StoMADS-PB, $\hkm$ is rather computed at the beginning of iterations in order to avoid keeping its possibly inaccurate values from one iteration to another. In fact, estimates in StoMADS-PB are always computed at the beginning of the iterations and their accuracy is improved compared to previous iterations {\bl as seen} in Section~\ref{computationPB}. Consequently, even though the sequence $\{\hkm\}_{k\in\N}$ has a globally decreasing tendency, it is not nonincreasing as in MADS with PB, but can possibly increase between successive iterations. The goal of StoMADS-PB {\bl is} to accept only the trial points satisfying $h(x^k)\leq\hkm$, {\bl and} any trial point $x^k$ for which the inequality $\uokx\leq\hkm$ does not hold is discarded from consideration since such an inequality implies that $h(x^k)\leq\hkm$ due to~\eqref{epspr}. {\bl However, this is} a sufficient acceptance condition since $\uokx>\hkm$ does not necessarily imply that $h(x^k)\leq\hkm$ does not hold, but rather leads to a situation of uncertainty which is not explicitly distinguished in the present manuscript {\bl for the sake of simplicity.}

The $\efPB$-reliable upper bound $\uokx$ previously obtained for $h(x^k)$ {\bl also allows} to determine the feasibility with respect to the relaxable constraints of a given trial point $x^k\in\X$. Indeed,
%\begin{equation}
%h(x^k)\leq \uokx=\sum_{j=1}^{m}\max\left\lbrace \cok(x^k)+\efPB(\dpl)^2,0 \right\rbrace. \nonumber
%\end{equation}
it obviously follows from~\eqref{epspr} that $h(x^k)=0$ if $\uokx=0$, which is satisfied provided that $\cokx\leq -\efPB(\dpl)^2$, for all {\bl $j\in J$}. This means that {\bl in order for $h(x^k)=0$} to hold, all the estimates of constraint function values must be sufficiently negative and not simply zero. By means of the following definition, {\bl StoMADS-PB partitions} the trial points into so-called $\efPB$-{\it feasible} and $\efPB$-{\it infeasible} points.
%and similarly, $h(x^k+s^k)=0$ if $\uskx=0$.  \hkm
\begin{definition}\label{feasInfeaspoints}
	Let $x^k\in\X$ be any trial point and $\uokx$ be an $\efPB$-reliable upper bound for $h(x^k)$. Then $x^k$ is called $\efPB$-feasible if $\uokx=0$, and it is called  $\efPB$-infeasible if $\ 0<\uokx\leq \hkm$. Similarly, $x^k+s^k\in\X$ is called $\efPB$-feasible if $\uskx=0$, and it is called  $\efPB$-infeasible if $\ 0<\uskx\leq \hkm$. 
	%Iterates that are $\efPB$-feasible are denoted by $\xkf$ while $\efPB$-infeasible ones are denoted by $\xki$.
\end{definition}
%\begin{remark}
%Note that with regard to the inequalities~\eqref{epspr}, it would have been more reasonable to state that trial points $x^k\in\X$ that are $\efPB$-infeasible are those for which $\lokx>0$ while those satisfying $\lokx=0$ and $\uokx>0$ simultaneously are of uncertain feasibility. However, this would have led to a more complex analysis and therefore has not been considered in order to make the analysis simpler. 
%\end{remark}
StoMADS-PB does not require that the starting {\bl point is} $\efPB$-feasible. The algorithm can be applied to any problem satisfying only the following {\bl assumption} adapted from~\cite{AuDe09a}.
\begin{assumption}\label{assumptA1}
	There exists some point $x^0\in\mathcal{X}$ such that $f^0_0(x^0)$ and $u^0_0(x^0)$ are both finite, and $u^0_0(x^0)\leq h^0_{\max}$.
\end{assumption}
The next result similar to that in~\cite{audet2019stomads} provides a sufficient information to identify a decrease in~$f$ {\bl and also allows} to determine an iteration type in Section~\ref{descriptionAlgo}.
\begin{proposition}\label{fdcrease}
	Let $f_0^k$ and $f_s^k$ be $\efPB$-accurate estimates of $f(x^k)$ and $f(x^k+s^k)$, respectively, for $x^k$ and $x^k+s^k\in\X$. Let $\gamma>2$ be {\bl a constant}. Then the following holds:
	\begin{equation}
	\text{if }\ \fsk -\fok \leq -\gamma\efPB(\dpl)^2,\ \ \text{then}\ \ \fd(x^k+s^k)-\fd(x^k)\leq-(\gamma-2)\efPB(\dpl)^2 <0.\label{successf}
	\end{equation}
\end{proposition}
\begin{proof}
	The proof {\bl follows from} Definition~\ref{epsilonAccDef} and the next equality
	\[\fd(x^k+s^k)-\fd(x^k)=\fd(x^k+s^k)-\fsk+\left(\fsk-\fok\right)+\fok-\fd(x^k).\] 
\end{proof}

\subsection{The StoMADS-PB algorithm and parameter update}\label{descriptionAlgo}
Recall first that MADS with PB is an iterative algorithm where every iteration comprises two main steps: an optional step called the SEARCH, {\bl and the POLL}. The SEARCH which typically consists of a global exploration may use a plethora of strategies like those based on interpolatory models, heuristics and surrogate functions or simplified physics models~\cite{AuDe09a} to explore the variables space. Each iteration of StoMADS-PB can also allow a SEARCH step, but it is not shown here for simplicity. Similarly to MADS with PB, the POLL step of StoMADS-PB is more rigidly defined unlike the freedom of the SEARCH and consists of a local exploration. During each of these two steps, a finite number of trial points is generated on an underlying {\it mesh} $\Mk$. The mesh is a discretization of the variables space, whose coarseness or fineness is controlled by a {\bl mesh size parameter} $\dm$ thus deviating from the notation $\Delta^m_k$ from~\cite{AuDe09a}, since uppercase letters will be used to denote random variables. For the remainder of the manuscript, $s^k=\dm d^k$ where $d^k$ is a nonzero direction around $x^k\in\Mk$. The POLL step is governed by the {\bl poll size parameter} $\dpl$ which is linked to $\dm$ by $\dm=\min \{\dpl,(\dpl)^2\}$~\cite{AuHa2017}. As specified earlier, $\{\dpl\}_{k\in\N}$ will play the role of the sequence of nonnegative real numbers introduced in Definition~\ref{epsilonAccDef}. Let $\hat{z}\in\N$ be a large fixed integer and $\tau\in (0,1)\cap\mathbb{Q}$ be a fixed rational constant. {\bl For the needs} of Section~\ref{convAnalysisPB}, note also that as in~\cite{audet2019stomads}, $\dpl$ is supposed to be bounded above by the positive and fixed constant $\dmax$ in order for the random poll size parameter $\Dp$ introduced in Section~\ref{subRandom} to be integrable. The definitions of the mesh $\Mk$ and the POLL set $\mathcal{P}^k$ inspired from~\cite{AuDe09a} are given next.
\begin{definition}\label{meshpollset}
	Let $\mathbf{D}\in\R^{n\times p}$ be a matrix, with columns denoted by the set $\mathbb{D}$ which form a positive spanning set. At the beginning of iteration~$k$, let $\xki$ and $\xkf$ denote respectively the $\efPB$-infeasible and the $\efPB$-feasible incumbent solutions (there might be only one), and let $\Vk:=\{\xki,\xkf\}$ be the set of such incumbents. The mesh $\M^k$ and the POLL set $\mathcal{P}^k$ are respectively
	\[\mathcal{M}^k:= \{x^k+\dm d: x^k\in\Vk,\ d=\mathbf{D}y,\ y\in \Z^p\}\quad\text{and}\quad \mathcal{P}^k := \mathcal{P}^k(\xki)\cup\mathcal{P}^k(\xkf), \]
	where $\forall x^k\in\mathcal{M}^k\cap\X$, $\mathcal{P}^k(x^k)=\{x^k+\dm d^k\in \mathcal{M}^k\cap\X: \dm\norminf{d^k}\leq \dpl b,\ d^k\in\mathbb{D}^k_p(x^k)\}$ is called a frame around $x^k$, with $b=\max\{\norminf{d'}, d'\in \mathbb{D}\}$. $\mathbb{D}^k_p(x^k)$ is a positive spanning set which is said to be a set of {\bl frame} directions around~$x^k$. 
	%The set $\mathbb{D}^k_p(x^k)$ is said to be a set of {\bl frame} directions around $x^k$. 
	The set $\mathbb{D}^k_p$ of all polling directions at iteration $k$ is defined by $\mathbb{D}^k_p:=\mathbb{D}^k_p(\xki)\cup\mathbb{D}^k_p(\xkf)$. When there is no incumbent $\efPB$-feasible solution $\xkf$, then the set $\Vk$ is reduced to $\{\xki\}$, in which case $\mathcal{P}^k = \mathcal{P}^k(\xki)$ and $\mathbb{D}^k_p=\mathbb{D}^k_p(\xki)$.
\end{definition}

After the POLL step is completed, {\bl StoMADS-PB computes} not only estimates $\fok$, $\fsk$, $\hok$ and $\hsk$ of $f(x^k)$, $f(x^k+s^k)$, $h(x^k)$ and $h(x^k+s^k)$, respectively at trial points $x^k\in \Vk$ and $x^k+s^k\in\Pk$, but also upper bounds $\uskx$ and $\uok(\xki)$, respectively for $h(x^k+s^k)$ and $h(\xki)$. The values of such estimates and bounds determine the iteration type of the algorithm
%, which can be $f$-Dominating, $h$-Dominating , Improving, or Unsuccessful. Moreover, they 
and govern also the way $\dpl$ is updated. Recall Definition~\ref{feasInfeaspoints} of $\efPB$-feasible and $\efPB$-infeasible points 
%and note that $\Vk$ denotes the set of the best $\efPB$-feasible and $\efPB$-infeasible incumbent solutions denoted respectively by $\xkf$ and $\xki$, 
at the beginning of iteration~$k$. The incumbent solutions $\xki$ and $\xkf$ are constructed by ranking trial mesh points of $\X$, making use of {\bl the dominance} notion inspired from~\cite{AuDe09a}.

\begin{definition}
	The $\efPB$-feasible point $x^k+s^k$ is said to dominate the $\efPB$-feasible point $x^k$, denoted $x^k+s^k \domf x^k$, when $\fsk-\fok\leq-\gamma\efPB(\dpl)^2$, with $\uskx=0$.\\
	The $\efPB$-infeasible point $x^k+s^k$ is said to dominate the $\efPB$-infeasible point $x^k$, denoted $x^k+s^k \domh x^k$, when $\fsk-\fok\leq-\gamma\efPB(\dpl)^2$ and $\hsk-\hok\leq-\gamma m\efPB(\dpl)^2$, with $\ 0<\uskx\leq \hkm$.
	%A point $x$ in some set $\mathcal{S}\subset\X$ is said to be undominated if there are no $x^k+s^k\in \mathcal{S}$ that dominate $x^k$.
\end{definition}

Adapting the terminologies {\bl from}~\cite{AuDe09a} and depending on the values of the aforementioned estimates and bounds, there are four StoMADS-PB iterations types: an iteration can be either $f$-Dominating, $h$-Dominating (the former and the latter are referred to as {\bl dominating} iterations), Improving, or Unsuccessful. During a dominating iteration, either the algorithm has found a first $\efPB$-feasible iterate or a trial point that dominates an incumbent is generated. An iteration which is Improving is not dominating but it aims to improve the feasibility of the $\efPB$-infeasible incumbent. Unsuccessful iterations are neither dominating nor improving. 
%More details about these iterations are given below. 
\begin{itemize}
	\item At the beginning of iteration $k$, if there is no available $\efPB$-feasible solution, then the iteration is called $f$-Dominating if for $x^k\in\Vk$, a first trial point $x^k+s^k\in\Pk$ satisfying $\uskx=0$ is found, in which case  $h(x^k+s^k)=0$ due to Proposition~\ref{boundh}, meaning that $x^k+s^k$ is $\efPB$-feasible. Otherwise, if an $\efPB$-feasible point that dominates the incumbent is generated, i.e., $x^k+s^k \domf \xkf$ for some $x^k\in\Vk$, then the inequality $\fsk(x^k+s^k)-\fok(\xkf)\leq-\gamma\efPB(\dpl)^2$ leads to a decrease in $f$ due to Proposition~\ref{fdcrease}. In either case, $\xkfun:=x^k+s^k$ and $\delta_p^{k+1}= \min\{\tau^{-1}\dpl,\dmax\}$. The $\efPB$-infeasible incumbent $\xki$ is not updated since there is no feasibility improvement.
	\item Iteration $k$ is said to be $h$-Dominating whenever an $\efPB$-infeasible point that dominates the incumbent is generated, i.e., $\xki+s^k \domh \xki$, which means that both inequalities $\fsk(\xki+s^k)-\fok(\xki)\leq-\gamma\efPB(\dpl)^2$ and $\hsk(\xki+s^k)-\hok(\xki)\leq-\gamma m\efPB(\dpl)^2$ hold. Consequently, it follows from Propositions~\ref{successh} and~\ref{fdcrease} that  decreases occur both in $f$ and $h$. In this case, $\xkfun=\xkf$ and since feasibility is improved, $\xkiun$ is set to equal $\xki+s^k$ while the poll size parameter is updated as at $f$-Dominating iterations.
	\item Iteration $k$ is said to be Improving if it is not dominating but there is at least one $\efPB$-infeasible point $\xki+s^k$ satisfying $\hsk(\xki+s^k)-\hok(\xki)\leq-\gamma m\efPB(\dpl)^2$. Indeed, this means that $\xki+s^k$ improves the feasibility of the $\efPB$-infeasible incumbent $\xki$ since the previous inequality leads to a decrease in~$h$ due to Proposition~\ref{successh}. In this case, $\dpl$ is updated as in dominating iterations, $\xkfun=\xkf$ while the $\efPB$-infeasible incumbent is updated according to
	% so that $\xkiun$ belongs to the set of trial points that minimize the $\efPB$-reliable upper bounds $\usk(\xki+s^k)$ for which the previous inequality holds. Specifically,
	\begin{equation}
	\xkiun\in\underset{{\xki+s^k}}{\arg\!\min} \accolade{\usk(\xki+s^k): \hsk(\xki+s^k)-h_0^k(\xki)\leq -\gamma m\efPB(\dpl)^2}.\nonumber
	\end{equation}
	%Then, the barrier point $\xkmxun$ is set to equal $\xkiun$ while 
	
	\item Finally, an iteration is called Unsuccessful if it is neither dominating nor Improving. In this case, $\delta_p^{k+1}=\tau\dpl$ while neither $\xki$ nor $\xkf$ are updated.
\end{itemize}

\begin{remark}\label{remarkFCF}
	Denote by $t>0$ the number of the first $f$-Dominating iteration of Algorithm~\ref{algoPB} and assume that $t<+\infty$. Then it is easy to notice that $\xkf=\xzeroi$ for all $k=0,1,\dots,t$ while $x^{t+1}_{\textnormal{feas}}\neq \xzeroi$. Moreover, even though estimates $\fok(\xkf)$, $\fsk(\xkf+s^k)$, $\hok(\xkf)$ and $\hsk(\xkf+s^k)$ are computed at $\xkf$ and $\xkf+s^k\in\Pk$ respectively for all $k\leq t$, they are not used by the algorithm until the end of iteration $t$ and it can also be noticed that no point in $\Pk$ that is generated using $\mathbb{D}^k_p(\xkf)$ is evaluated until the end of iteration $t$. In fact, setting the initial $\efPB$-feasible guess to equal $\xzeroi$ as it is in Algorithm~\ref{algoPB} and then computing the latter estimates are not necessary in practice. However, doing so allows simply the aforementioned estimates to be defined for all $k\geq 0$ for theoretical needs, specifically the construction of the $\sigma$-algebra $\fcf$ in Section~\ref{sec3}. 
\end{remark}

\subsection{Frame center selection rule}\label{frameSelection}
Before describing the frame center selection rule, recall the set $\Vk$ of incumbent solutions introduced in Definition~\ref{meshpollset} and the fact that POLL trial points are generated inside frames around such incumbents At a given iteration, there are either one or two frame centers in $\Vk$. When $\Vk$ contains only one point, then using terminologies from~\cite{AuDe09a}, that point is called the {\bl primary} frame center. In the event that there are two incumbent solutions $\xki$ and $\xkf$, one of them is chosen as {\bl the} {\bl primary} frame center while the other one is the {\bl secondary} frame center. The primary frame center in~\cite{AuDe09a} is chosen to be the infeasible incumbent solution while the secondary frame center is the feasible incumbent whenever $f_k^F-\rho > f_k^I$, where the positive scalar $\rho$ is the so called {\bl frame center trigger}, $f_k^F$ and $f_k^I$ are respectively the incumbent feasible and infeasible $f$-values at iteration $k$. Otherwise if the previous inequality does not hold, the primary and secondary frame centers {\bl are the} feasible and infeasible incumbent solutions. Because of the unavailability of $f$ function values for StoMADS-PB, a specific frame center selection strategy using estimates of such function values {\bl is proposed and relies} on the following result.
\begin{proposition}\label{trigger}
	Let $\fok(\xkf)$ and $\fok(\xki)$ be $\efPB$-accurate estimates of $f(\xkf)$ and $f(\xki)$ respectively. Let $\rho>0$ be a scalar. 
	\begin{equation}
	\text{If}\ \ \fok(\xkf)-\rho>\fok(\xki)+2\efPB(\dpl)^2,\ \ \text{then}\ \ f(\xkf)-\rho>f(\xki).
	\end{equation} 
\end{proposition}
\begin{proof}
	Assume that $\fok(\xkf)-\rho>\fok(\xki)+2\efPB(\dpl)^2$. Then, it follows from the $\efPB$-accuracy of $\fok(\xkf)$ and $\fok(\xki)$ that 
	\begin{eqnarray}
	f(\xki)-f(\xkf)&=& \left[f(\xki)-\fok(\xki)\right]+\left[\fok(\xki)-\fok(\xkf)\right]+\left[\fok(\xkf)-f(\xkf)\right]\nonumber \\
	&<& 2\efPB(\dpl)^2-(\rho+2\efPB(\dpl)^2)=-\rho.
	\end{eqnarray} 
\end{proof}
Thus according to Proposition~\ref{trigger}, $\xkf$ is always chosen as {\bl the} StoMADS-PB primary frame center unless the estimates $\fok(\xkf)$ and $\fok(\xki)$ satisfy a sufficient decrease condition leading to the inequality $f(\xkf)-\rho>f(\xki)$, which as in~\cite{AuDe09a} allows  the choice of the infeasible incumbent solution as primary frame center.

As in~\cite{AuDe09a}, StoMADS-PB as implemented for the computational study in Section~\ref{sec5} places less effort in polling around the secondary frame center than the primary one. Specifically, the default strategy is to use a {\bl maximal positive basis}~\cite{AuHa2017} for the primary frame center and only two directions with one being the negative of the first for the secondary frame center.

%------------------------------------------------%
\begin{figure*}[ht!]
	\begin{algorithm}[H]
		\caption{StoMADS-PB}
		\label{algoPB} 
		\textbf{[0] Initialization}\\
		\hspace*{10mm}choose $\xzeroi\in\X$, $\delta_p^0>0$, $\tau\in(0,1)\cap\mathbb{Q}$, $\efPB>0$, $\gamma>2$ and  $\hat{z}\in\N^*$\\ %and $\epsilon_{stop}\geq 0$. \\		
		\hspace*{10mm}set the feasibility success ${flag}$ = {\small FALSE}, $\mathcal{V}^0 \gets\{\xzeroi \}$ and $x^0_{\textnormal{feas}}\gets \xzeroi$ \\
		\hspace*{10mm}set the iteration counter $k \gets 0$\\
		\textbf{[1] Parameter Update}\\
		\hspace*{10mm}set $\delta^k_m\gets\min \{\dpl,(\dpl)^2\}$\\
		\textbf{[2] Poll}\\
		\hspace*{10mm}generate a finite list $\mathcal{P}^k$ of candidates using the polling directions $\mathbb{D}^k_p(\xki)\cup\mathbb{D}^k_p(\xkf)$ \\  
		\hspace*{10mm}obtain estimates $\fok,\fsk,\hok$ and $\hsk$ of $f(x^k),f(x^k+s^k), h(x^k)$ and $h(x^k+s^k)$\\
		\hspace*{10mm}respectively, at $x^k\in \mathcal{V}^k\cup\{\xkf\}$, $\ x^k+s^k\in\mathcal{P}^k$, then compute bounds $u^k_s(x^k+s^k)$ \\
		\hspace*{10mm}and $\uok(\xki)$, using blackbox evaluations\\
		\hspace*{10mm}set the barrier threshold $\hkm\gets u_0^k(\xki)$\\
		\hspace*{10mm}\textbf{$f$-Dominating}\\
		\hspace*{10mm}if ${flag}$ = {\small FALSE} and $u^k_s(x^k+s^k)=0$ or ${flag}$ = {\small TRUE} and $x^k+s^k\domf\xkf$\\
		\hspace*{10mm}for some $x^k\in \mathcal{V}^k$ and $s^k\in \{\dm d^k: d^k\in \mathbb{D}^k_p(x^k)\}$\\
		%\hspace*{10mm}if ${flag}$ = {\small TRUE} and $x^k+s^k\domf\xkf$ or ${flag}$ = {\small FALSE} and $u^k_s(x^k+s^k)=0$ for some \\
		%\hspace*{10mm}$x^k\in \mathcal{V}^k$ and $s^k\in \{\dm d^k: d^k\in \mathbb{D}^k_p(x^k)\}$, \\
		\hspace*{10mm}set $\xkiun\gets\xki$, $\xkfun\gets x^k+s^k$ and $\delta_p^{k+1}\gets{\blc \min\{\tau^{-1}\dpl,\dmax\}}$  \\
		\hspace*{10mm}reset the feasibility success ${flag}$ = {\small TRUE}, set $\mathcal{V}^{k+1} \gets\{\xkiun, \xkfun\}$ and go to \textbf{[4]}  \\
		\hspace*{10mm}\textbf{$h$-Dominating}\\
		\hspace*{10mm}else if $\xki+s^k\domh\xki$ for some $s^k\in \{\dm d^k: d^k\in \mathbb{D}^k_p(\xki)\}$ \\
		\hspace*{10mm}set $\xkiun\gets \xki+s^k$, $\xkfun\gets\xkf$ and  $\delta_p^{k+1}\gets{\blc \min\{\tau^{-1}\dpl,\dmax\}}$\\%, and go to \textbf{[3]}.\\
		\hspace*{10mm}\textbf{Improving}\\
		\hspace*{10mm}else if $\hsk(\xki+s^k)-h_0^k(\xki)\leq -\gamma m\efPB(\dpl)^2$ for some previously evaluated $\xki+s^k$\\
		\hspace*{10mm}set $\xkiun\in\arg\!\min_{\xki+s^k}\{\usk(\xki+s^k): \hsk(\xki+s^k)-h_0^k(\xki)\leq -\gamma m\efPB(\dpl)^2\}$ \\
		%\hspace*{10mm}set $\xkmxun\in\arg\!\max_{x^k+s^k}\{\uskx: \hskx-h_0^k(\xki)\leq -\gamma m\efPB(\dpl)^2\}$, \\
		%\hspace*{10mm}set $\xkmxun\gets\xkiun $ \\
		\hspace*{10mm}$\xkfun\gets\xkf$ and $\delta_p^{k+1}\gets{\blc \min\{\tau^{-1}\dpl,\dmax\}}$\\% and go to \textbf{[3]}. \\
		\hspace*{10mm}\textbf{Unsuccessful}\\ 
		\hspace*{10mm}otherwise, set $\xkiun\gets\xki$, $\xkfun\gets\xkf$ and $\delta_p^{k+1}\gets\tau\dpl$ \\
		\textbf{[3] Feasibility update}\\
		\hspace*{10mm}if ${flag}$ = {\small TRUE}\\
		\hspace*{10mm}set $\mathcal{V}^{k+1} \gets\{\xkiun, \xkfun\}$\\
		\hspace*{10mm}otherwise, $\mathcal{V}^{k+1} \gets\{\xkiun\}$ \\
		\textbf{[4] Termination}\\
		\hspace*{10mm}if no termination criterion is met \\
		\hspace*{10mm}set $k\gets k+1$ and go to \textbf{[1]}\\
		\hspace*{10mm}otherwise stop
	\end{algorithm}
	\caption{\small{StoMADS-PB algorithm for constrained stochastic optimization.}}
\end{figure*}
%------------------------------------------------------------------------

\section{Stochastic process generated by StoMADS-PB}\label{sec3}

The stochastic quantities in the present work are all defined on the same probability space $(\Omega,\mathcal{G},\pb)$. The nonempty set $\Omega$ is referred to as the sample space and its subsets are called events. The collection $\mathcal{G}$ of such events is called a $\sigma$-algebra or $\sigma$-field and $\pb$ is a finite measure satisfying $\pb(\Omega)=1$, referred to as probability measure and defined on the measurable space $(\Omega,\mathcal{G})$. Each element $\omega\in\Omega$ is referred to as a sample point or a possible outcome. Let $\B(\rn)$ be the Borel $\sigma$-algebra of $\rn$, i.e., the one generated by its open sets. A random variable $X$ is a measurable map defined on $(\Omega,\mathcal{G},\pb)$ into the measurable space $(\rn,\B(\rn))$, where measurability means that each event $\{X\in B \}:=X^{-1}(B)$ belongs to $\mathcal{G}$ for all $B\in\B(\rn)$~\cite{bhattacharya2007basic,dzahini2020expected}. 

The estimates $\fokx$, $\fskx$, $\cokx$ and $\cskx$, for $j=1,2,\dots,m$, $x^k\in\{\xki,\xkf\}$ and $x^k+s^k\in\Pk$, of function values are computed at every iteration of Algorithm~\ref{algoPB} using the noisy blackbox evaluations. Because of the randomness of the blackbox outputs, such estimates can respectively be considered as realizations of random estimates $\Fokx$, $\Fskx$, $\Cokx$ and $\Cskx$, for $j=1,2,\dots,m$. Since each iteration $k$ of Algorithm~\ref{algoPB} is influenced by the randomness stemming from such random estimates, Algorithm~\ref{algoPB} results in a stochastic process. For the remainder of the manuscript, uppercase letters will be used to denote random quantities while their realizations will be denoted by lowercase letters. Thus, $x^k=X^k(\omega)$, $\xki=\Xki(\omega)$, $\xkf=\Xkf(\omega)$, $s^k=S^k(\omega)$, $\dpl=\Dp(\omega)$ and $\dm=\Dm(\omega)$ denote respectively realizations of $X^k$, $\Xki$, $\Xkf$, $S^k$, $\Dp$ and $\Dm$. Similarly, $\fokx=\Fokx(\omega)$, $\fskx=\Fskx(\omega)$, $\cokx=\Cokx(\omega)$, $\cskx=\Cskx(\omega)$, $\hokx=\Hokx(\omega)$, $\hskx=\Hskx(\omega)$, $\lokx=\Lok(X^k)(\omega)$, $\lskx=\Lsk(X^k+S^k)(\omega)$, $\uokx=\Uok(X^k)(\omega)$ and $\uskx=\Usk(X^k+S^k)(\omega)$. 
When there is no ambiguity, $\Fok$ will be used instead of $\Fokx$, etc. In general, following the notations in~\cite{audet2019stomads,blanchet2019convergence,chen2018stochastic,dzahini2020expected,paquette2018stochastic}, $\Fok$, $\Fsk$, $\Hok$ and $\Hsk$ are respectively the estimates of $f(X^k)$, $f(X^k+S^k)$, $h(X^k)$ and $h(X^k+S^k)$. Moreover, as highlighted in~\cite{audet2019stomads}, the notation ``$f(X^k)$'' is used to denote the random variable with realizations $f(X^k(\omega))$. 

The present research aims to show that the stochastic process $\left\lbrace\Xki,\Xkf,\Dp,\Dm,\Fok,\Fsk,\Hok,\Hsk,\right.$ $\left.\Lok,\Uok,\Lsk,\Usk\right\rbrace$ resulting from Algorithm~\ref{algoPB} converges with probability one under some assumptions on the estimates $\Fok,\Fsk,\Cok,\Csk,\Hok,\Hsk$ and on the bounds $\Lok,\Uok,\Lsk,\Usk$. In particular, the estimates $\Fok,\Fsk,\Cok$ and $\Csk$ will be assumed to be accurate while the bounds will be assumed to be reliable, with sufficiently high but fixed probabilities, {\bl conditioned on the past}.

\subsection{Probabilistic bounds and probabilistic estimates}\label{subRandom}

The previously mentioned notion of conditioning on the past is formalized following~\cite{audet2019stomads,blanchet2019convergence,chen2018stochastic,dzahini2020expected,paquette2018stochastic}. Denote by $\fcf$ the $\sigma$-algebra generated by $F_0^{\ell}(X^{\ell})$, $F_s^{\ell}(X^{\ell}+S^{\ell})$, $C_{j,0}^{\ell}(X^{\ell})$ and $C_{j,s}^{\ell}(X^{\ell}+S^{\ell})$, for $j=1,2,\dots,m$, for $X^{\ell}\in\accolade{X^{\ell}_{\inf},X^{\ell}_{\textnormal{feas}}}$ and for $\ell=0,1,\dots,k-1$. For completeness, $\mathcal{F}^{C\cdot F}_{-1}$ is set to equal $\sigma(x^0)=\sigma(\xzeroi)$. Thus, $\{\mathcal{F}^{C\cdot F}_k\}_{k\geq -1}$ is a filtration, i.e., a subsequence of increasing $\sigma$-algebras of $\mathcal{G}$.

Sufficient accuracy of functions estimates is measured using the poll size parameter and is formalized, following~\cite{audet2019stomads,blanchet2019convergence,chen2018stochastic,dzahini2020expected,paquette2018stochastic} by means of the definitions bellow.

\begin{definition}\label{jkDef}
	A sequence of random estimates $\{F^k_0,F^k_s\}$ is said to be $\beta$-probabilistically $\efPB$-accurate with respect to the corresponding sequence $\{X^k, S^k, \Dp\}$ if the events 
	\begin{equation}\label{event1PB}
	J_k=\{F^k_0, F^k_s,\ \text{are}\ \efPB\text{-accurate estimates of}\ f(x^k)\ \text{and}\ f(x^k+s^k),\ \text{respectively for}\ \Dp\}\nonumber
	\end{equation}
	satisfy the following submartingale-like condition
	\begin{equation}\label{beta1PB}
	\pr{J_k\ |\ \fcf}= \E{\ijk\ |\ \fcf}\geq \beta,\nonumber
	\end{equation}
	where $\ijk$ denotes the indicator function of the event $J_k$, i.e.,  $\ijk=1$ if $\omega\in J_k$ and $\ijk=0$  otherwise. The estimates are called ``good'' if $\ijk=1$. Otherwise they are called ``bad''.
\end{definition}

\begin{definition}\label{probConstr}
	A sequence of random estimates $\{\Cok,\Csk\}$ is said to be $\alpha^{1/m}$-probabilistically $\efPB$-accurate for some $j=1,2,\dots,m$ with respect to the corresponding sequence $\{X^k, S^k, \Dp\}$ if the events 
	\begin{equation}\label{eventJi}
	I_k^j=\{\Cok,\Csk,\ \text{are}\ \efPB\text{-accurate estimates of}\ c_j(x^k)\ \text{and}\ c_j(x^k+s^k),\ \text{respectively for}\ \Dp\}\nonumber
	\end{equation}
	satisfy the following submartingale-like condition
	\begin{equation}\label{alpha2m}
	\pr{I_k^j\ |\ \fcf}= \E{\iikj\ |\ \fcf}\geq \alpha^{1/m}.\nonumber
	\end{equation}
\end{definition}

To formalize the sufficient reliability of random bounds in the present work, the following definition is introduced.
\begin{definition}\label{ikDef}
	A sequence of random bounds $\{\Lok, \Uok,\Lsk, \Usk \}$ is said to be $\alpha$-probabilistically $\efPB$-accurate with respect to the corresponding sequence $\{X^k, S^k, \Dp\}$ if the events 
	\begin{eqnarray}\label{event2}
	I_k &=& \left\lbrace ``\Lok\ \text{and}\ \Uok\ \text{are}\ \efPB\text{-reliable bounds for}\ h(x^k) \text{'', and}\ `` \Lsk\ \text{and}\ \Usk\ \text{are}\ \efPB\text{-reliable bounds}\right.\nonumber\\
	& & \left. \text{for}\ h(x^k+s^k)\text{'', respectively for}\ \Dp\right\rbrace
	\end{eqnarray}
	satisfy the following submartingale-like condition
	\begin{equation}\label{alpha1}
	\pr{I_k\ |\ \fcf}=\E{\iik\ |\ \fcf}\geq\pr{\overset{m}{\underset{j=1}{\bigcap}}I_k^j\ |\ \fcf} \geq \alpha,\nonumber
	\end{equation}
	The bounds are called ``good'' if $\ \iik=1$. Otherwise, $\ \iik=0$ and they are called ``bad''.
\end{definition}
The $p$-integrability of random variables~\cite{audet2019stomads,bhattacharya2007basic} is defined below and will be useful for the analysis of Algorithm~\ref{algoPB}.
\begin{definition}\label{pIntegrableDefPB}
	Let $(\Omega,{\blc \mathcal{G}},\mathbb{P})$ be a probability space and $p\in [1,+\infty)$ be an integer. Then the Space $\lp$ of so-called $p$-integrable random variables is the set of all real-valued random variables $X$ such that
	\begin{equation}\label{pIntegrableIneqPB}
	\normp{X}:=\left(\int_{\Omega} \abs{X(\omega)}^p\pr{d\omega} \right)^{\frac{1}{p}}=\left(\E{\abs{X}^p} \right)^{\frac{1}{p}}<+\infty.\nonumber
	\end{equation}	
\end{definition}
As in~\cite{audet2019stomads}, the following is assumed in order for the random variables $f(X^k)$, $h(X^k)$ and $c_j(X^k)$, $j\in J$, to be integrable so that the conditional expectations $\E{f(X^k)|\fcf}$, $\E{c_j(X^k)|\fcf}$, $j\in J$ and $\E{h(X^k)|\fcf}$ can be well defined~\cite{bhattacharya2007basic}.

\begin{assumption}\label{lipschitzAssumption}
	The objective function $f$ and the constraints violation function $h$ are locally Lipschitz with constants $\lambda^f>0$ and $\lambda^h>0$, respectively. The constraint functions $c_j$, $j\in J$, are continuous on~$\X$. The set $\mathcal{U}\subset\X$ containing all iterates realizations is compact.
\end{assumption}
Local Lipschitz in the above assumption means, Lipschitz with a finite constant in some nonempty neighborhood intersected with $\X$~\cite{AuDe09a}.

\begin{proposition}\label{kappafmax}
	Under Assumption~\ref{lipschitzAssumption}, there exists a finite constant $\kappa^f_{\max}$ satisfying $\abs{f(x^k)}\leq \kappa^f_{\max}\ $ for all $x^k\in\U$. Moreover, the random variables $f(X^k)$, $h(X^k)$, $c_j(X^k)$ and $\Dp$ belong to $\mathbb{L}^1(\Omega,\mathcal{G},\pb)$, for all $j\in J$ and for all $k\geq 0$.
\end{proposition}
\begin{proof}
	The proof is inspired from~\cite{audet2019stomads}. Since $f$ is locally Lipschitz on the compact set $\U$, the it is bounded on $\U$. Consequently, there exists a finite constant $\kappa^f_{\max}$ such that $\abs{f(x^k)}\leq \kappa^f_{\max}\ $ for all $x^k\in\U$. Similarly, there exist $\kappa^h_{\max}$ satisfying $\abs{h(x^k)}\leq \kappa^h_{\max} $ and $\kappa^c_{\max}$ such that $\abs{c_j(x^k)}\leq \kappa^c_{\max}\ $ for all $j\in J$ and all $x^k\in\U$, since $h$ is locally Lipschitz and $c_j$ is continuous on $\U$. Thus, $\E{\abs{f(X^k)}}:=\int_{\Omega}\abs{f(X^k(\omega))} \mathbb{P}(d\omega)\leq \kappa^f_{\max}<+\infty$. Similarly, $\E{\abs{h(X^k)}}\leq \kappa^h_{\max}\leq+\infty$ and for all $j\in J$, $\E{\abs{c_j(X^k)}}\leq \kappa^c_{\max}\leq+\infty$. Finally, the integrability of $\Dp$ follows from the fact that $\Dp(\omega)\leq \dmax$ for all $\omega\in\Omega$, which implies that $\E{\abs{\Dp}}:=\int_{\Omega}\abs{\Dp(\omega)} \mathbb{P}(d\omega)\leq \dmax<+\infty$.
\end{proof}

Next are stated some key assumptions on the nature of the stochastic information in Algorithm~\ref{algoPB}, some of which are made in~\cite{audet2019stomads} and which will be useful for the convergence analysis of Section~\ref{convAnalysisPB}.

\begin{assumption}\label{randomEstim}
	For fixed $\alpha$ and $\beta\in (0,1)$, the followings hold for the random quantities generated by Algorithm~\ref{algoPB}. 
	\begin{itemize}
		\item[(i)] The sequence of estimates $\{\Fok,\Fsk\}$ generated by Algorithm~\ref{algoPB} is $\beta$-probabilistically $\efPB$-accurate.
		\item[(ii)] The sequence of estimates $\{\Fok,\Fsk\}$ generated by Algorithm~\ref{algoPB} satisfies the following variance condition for all $k\geq 0$,
		\begin{equation}\label{varcond1PB}
		\begin{split}
		& \E{\abs{\Fsk-\fd(X^k+S^k)}^2 |\ \fcf}\leq \efPB^2(1-\sqrt{\beta})(\Dp)^4\\
		\text{and}\quad & \E{\abs{\Fok-\fd(X^k)}^2 |\ \fcf}\leq \efPB^2(1-\sqrt{\beta})(\Dp)^4.
		\end{split}
		\end{equation}
		\item[(iii)] For all $j=1,2,\dots,m$, the sequence of estimates $\{\Cok,\Csk\}$ is $\alpha^{1/m}$-probabilistically $\efPB$-accurate.
		\item[(iv)] For all $j=1,2,\dots,m$, the sequence of estimates $\{\Cok,\Csk\}$ satisfies the following variance condition for all $k\geq 0$,
		\begin{equation}\label{varcond2PB}
		\begin{split}
		& \E{\abs{\Csk-c_j(X^k+S^k)}^2 |\ \fcf}\leq \efPB^2\left(1-\alpha^{1/2m}\right)(\Dp)^4\\
		\text{and}\quad & \E{\abs{\Cok-c_j(X^k)}^2 |\ \fcf}\leq \efPB^2\left(1-\alpha^{1/2m}\right)(\Dp)^4.
		\end{split}
		\end{equation}
		\item[(v)]  The sequence of random bounds $\{\Lok, \Uok,\Lsk, \Usk \}$ is  $\alpha$-probabilistically $\efPB$-reliable. 
		\item[(vi)] The sequence of random estimated violations $\{\Hok,\Hsk\}$ satisfies 
		\begin{equation}\label{Hcond1}
		\begin{split}
		& \E{\abs{\Hsk-h(X^k+S^k)} |\ \fcf}\leq m\efPB(1-\alpha)^{1/2}(\Dp)^2\\
		\text{and}\quad & \E{\abs{\Hok-h(X^k)} |\ \fcf}\leq m\efPB(1-\alpha)^{1/2}(\Dp)^2.
		\end{split}
		\end{equation}
	\end{itemize}
\end{assumption}
An iteration $k$ for which $\iik\ijk=1$, i.e., for which the events $I_k$ and $J_k$ both occur, will be called ``{\it true}''. Otherwise, it will be called ``{\it false}''. Even though the present algorithmic framework does not allow to determine which iterations are true or false, Theorem~\ref{trueIterations} shows that true iterations occur infinitely often for convergence to hold, provided that estimates and bounds are sufficiently accurate. Theorem~\ref{trueIterations} will also be useful for the convergence analysis of Algorithm~\ref{algoPB}, more precisely in Subsection~\ref{forf}. 
%Its proof makes use of the following auxiliary result~\cite{BaScVi2014,chen2018stochastic}, taken from the martingale literature~\cite{durrett2010probability}.
%\begin{theorem}\label{durett}
%	Let $\{G_k\}_{k\in\N}$ be a submartingale, i.e., a sequence of random variables satisfying
%	\[\E{G_k|\mathcal{F}^G_{k-1}}\geq G_{k-1}\quad\text{for all}\ k\geq 0,\] 
%	where $\mathcal{F}^G_{k-1}=\sigma(G_0,G_1,\dots,G_{k-1})$ is the $\sigma$-algebra generated by $G_0,G_1,\dots,G_{k-1}$, and $\Esp(G_k|\mathcal{F}^G_{k-1})$ denotes the conditional expectation of $G_k$, given the past history of events $\mathcal{F}^G_{k-1}$.
%	
%	Assume further that $G_k-G_{k-1}\leq M < +\infty$, for every $k$. Then, 
%	\begin{equation}\label{submartingale}
%	\pr{\left\lbrace \underset{k\to\infty}{\lim} G_k < \infty \right\rbrace \cup \left\lbrace \underset{k\to\infty}{\limsup}\ G_k = \infty \right\rbrace} = 1.\nonumber
%	\end{equation}
%\end{theorem}
\begin{theorem}\label{trueIterations}
	Assume that Assumption~\ref{randomEstim} holds for $\alpha\beta\in (1/2,1)$. Then true iterations of Algorithm~\ref{algoPB} occur infinitely often. 
\end{theorem}
\begin{proof} 
	%The result easily follows using the following random walk
	%\begin{equation}\label{Wk}
	%W_k= \sum_{i=0}^{k}(2\cdot\iii\iji-1),
	%\end{equation}
	%a similar of which is proved in~\cite{chen2018stochastic} to satisfy $\accolade{\underset{k\to+\infty}{\limsup}\ W_k=+\infty}$ almost surely.
	%%where the indicator random variables $\iii$ and $\iji$ are such that $\iii=1$ if $I_i$ occurs, $\iii=0$ otherwise, and similarly, $\iji=1$ if $J_i$ occurs while $\iji=0$ otherwise. 
	%Indeed, notice that $\{W_k\}$ is a $\fcf$-submartingale since,
	%\begin{eqnarray*}
	%	\E{W_k|\fcf} &=& \E{W_{k-1}|\fcf} + \E{2 \cdot \iik\ijk - 1|\fcf}\\
	%	&=& W_{k-1}+2 \E{\iik\ijk |\fcf}- 1\\
	%%	&=& W_{k-1}+2 \pr{I_k\cap J_k|\fcf}- 1\\
	%	&\geq& W_{k-1},
	%\end{eqnarray*}
	%where the last inequality follows from the fact that $\E{\iik\ijk|\fcf}\geq\alpha\beta>1/2$. Moreover, it is easy to notice that $\{W_k\}$ has no limit since $W_k$ has only $\pm 1$ increments. Thus, it follows from Theorem~\ref{durett} that $\accolade{\underset{k\to+\infty}{\limsup}\ W_k=+\infty}$ almost surely, which means that 
	%\[\pr{ \accolade{\omega\in\Omega:\exists K(\omega)\subset\N\ \text{such that}\ \lim_{k\in K(\omega)}W_k(\omega)=+\infty} }=1. \]
	%This implies that $\iii\iji=1$ infinitely often, which achieves the proof.
	Consider the following random walk 
	\begin{equation}\label{Wk}
	W_k= \sum_{i=0}^{k}(2\cdot\iii\iji-1).
	\end{equation}
	Then, the result easily follows from the fact that $\accolade{\underset{k\to+\infty}{\limsup}\ W_k=+\infty}$ almost surely, the proof of which can be derived from that of Theorem~4.16 in~\cite{chen2018stochastic} (using $\fcf$ instead of $\fij$), where a similar random walk was studied.
	Indeed, the latter result means that 
	\[\pr{ \accolade{\omega\in\Omega:\exists K(\omega)\subset\N\ \text{such that}\ \lim_{k\in K(\omega)}W_k(\omega)=+\infty} }=1, \]
	which implies that $\iii\iji=1$ infinitely often.
\end{proof}

%\clearpage
%------------------------------------------------%

%\subsection{Probabilistic estimates}

\subsection{ Computation of probabilistically accurate estimates and reliable bounds}\label{computationPB}
%------------------------------------------------%
This section discusses approaches for computing accurate random estimates and reliable bounds satisfying Assumption~\ref{randomEstim} in a simple random noise framework, and hence how corresponding deterministic estimates can be obtained using evaluations of the stochastic blackbox. Such approaches strongly rely on the computation of $\alpha^{1/m}$-probabilistically $\efPB$-accurate estimates $\{\Cok,\Csk\}$, using techniques derived in~\cite{chen2018stochastic}.

%First, recall that $f_{\Theta_0}$ denotes the noisy computable version of the numerically unavailable objective function $f$, while for all $j\in J=\{1,2,\dots,m\}$, $\ c_{\Theta_j}$ denotes the noisy computable version of the numerically unavailable constraints $c_j$. The random variables $\Theta_0,\Theta_1,\dots,\Theta_m,$ are supposed to be independent. Moreover, 
Consider the following typical noise assumption often made in stochastic optimization literature:
%, i.e., suppose that
\begin{eqnarray*}
	\Esp_{\Theta_0}\left[f_{\Theta_0}(x)\right]&=&f(x)\quad\text{and}\quad {\mathbb{V}}_{\Theta_0}\left[f_{\Theta_0}(x)\right]\leq V_0<+\infty\ \ \text{for all}\ x\in \X \\
	\Esp_{\Theta_j}\left[c_{\Theta_j}(x)\right]&=& c_j(x)\! \quad\text{and}\quad {\mathbb{V}}_{\Theta_j}\left[c_{\Theta_j}(x)\right]\leq V_j<+\infty\ \ \text{for all}\ x\in \X\ \text{and for all}\ j\in J,
\end{eqnarray*}
where $V_i>0$ is a constant for all $i=0,1,\dots,m$. Let $V=\max\{V_0,V_1,\dots,V_m\}$.  

For some fixed $j\in J$, let $\Theta_j^0$ and $\Theta_j^s$ be two independent random variables following the same distribution as $\Theta_j$. Let $\Theta^0_{j,\ell},\  \ell=1,2,\dots,p^k_j$ and $\Theta^s_{j,\ell},\  \ell=1,2,\dots,p^k_j$ be independent random samples of $\Theta_j^0$ and $\Theta_j^s$ respectively, where $p_j^k\geq 1$ is an integer denoting the sample size. In order to satisfy Assumption~\ref{randomEstim}-$(iii)$, define $\Cok$ and $\Csk$ respectively by  

\[\Cok=\frac{1}{p^k_j}\sum_{\ell=1}^{p^k_j}c_{\Theta^0_{j,\ell}}(x^k)\quad \text{and}\quad  \Csk=\frac{1}{p^k_j}\sum_{\ell=1}^{p^k_j}c_{\Theta^s_{j,\ell}}(x^k+s^k).\]
By noticing that $\E{\Cok}=c_j(x^k)$ and that $\var{\Cok}\leq\frac{V}{p^k_j}$ for all $j$, then it follows from the Chebyshev inequality that
\begin{equation}\label{chebichev}
\pr{\abs{\Cok-c_j(x^k)}>\efPB(\dpl)^2}=\pr{\abs{\Cok-\E{\Cok}}>\efPB(\dpl)^2}\leq\frac{V}{p^k_j \efPB^2(\dpl)^4}.
\end{equation}
Thus, choosing $p^k_j$ such that 
\begin{equation}\label{pkj}
p_j^k\geq \frac{V}{\efPB^2\left(1-\alpha^{1/2m} \right)(\dpl)^4}
\end{equation}
ensures that $\frac{V}{p^k_j \efPB^2(\dpl)^4}\leq 1-\alpha^{1/2m}$. Then, combining~\eqref{chebichev} and~\eqref{pkj} yields for all $j\in J$,
\begin{equation}\label{coalpha2m}
\pr{\abs{\Cok-c_j(x^k)}\leq \efPB(\dpl)^2}\geq \alpha^{1/2m}
\end{equation}
and similarly, $\pr{\abs{\Csk-c_j(x^k+s^k)}\leq \efPB(\dpl)^2}\geq \alpha^{1/2m}$.
%\begin{equation}\label{csalpha2m}
%\pr{\abs{\Csk-c_j(x^k+s^k)}\leq \efPB(\dpl)^2}\geq \alpha^{1/2m}.
%\end{equation}
It follows from the independence of the random variables $\Theta_j^0$ and $\Theta_j^s$ and both previous inequalities that 
\begin{equation}\label{alpham}
\pr{\accolade{\abs{\Cok-c_j(x^k)}\leq \efPB(\dpl)^2}\cap \accolade{\abs{\Csk-c_j(x^k+s^k)}\leq \efPB(\dpl)^2}}\geq \alpha^{1/m},
\end{equation}
which means that Assumption~\ref{randomEstim}-{\it (iii)} holds. Estimates $\cok=\Cok(\omega)$ and $\csk=\Csk(\omega)$, obtained by averaging $\, p^k_j$ realizations of $c_{\Theta_j}$, resulting from the evaluations of the stochastic blackbox, respectively at $x^k$ and $x^k+s^k$, are obviously $\efPB$-accurate.

In order to satisfy Assumption~\ref{randomEstim}-{\it (v)}, notice that the independence of the random variables $\Theta_j, j\in J$ combined with~\eqref{coalpha2m} implies
\begin{equation}\label{coalphademi}
\pr{\overset{m}{\underset{j=1}{\bigcap}}\left\lbrace \abs{\Cok-c_j(x^k)}\leq\efPB(\dpl)^2 \right\rbrace}=\prod_{j=1}^{m} \pr{ \abs{\Cok-c_j(x^k)}\leq\efPB(\dpl)^2}\geq \alpha^{1/2}
\end{equation}
%and similarly, using~\eqref{csalpha2m} yields
\begin{equation}\label{csalphademi}
\text{and similarly}, \quad  \pr{\overset{m}{\underset{j=1}{\bigcap}}\left\lbrace \abs{\Csk-c_j(x^k+s^k)}\leq\efPB(\dpl)^2 \right\rbrace}\geq \alpha^{1/2}.
\end{equation}
Define the random bounds $\Lokx$, $\ \Lskx$, $\ \Uokx\ $ and $\ \Uskx,\ $ respectively by
\begin{eqnarray*}\label{LokxUokx}
	\Lokx&=&\sum_{j=1}^m\max\left\lbrace \Cok-\efPB(\dpl)^2,0 \right\rbrace,\quad\quad \quad \quad      \Uokx=\sum_{j=1}^m\max\left\lbrace \Cok+\efPB(\dpl)^2,0 \right\rbrace\\
	\Lskx&=&\sum_{j=1}^m\max\left\lbrace \Csk-\efPB(\dpl)^2,0 \right\rbrace\ \ \text{and}\ \     \Uskx=\sum_{j=1}^m\max\left\lbrace \Csk+\efPB(\dpl)^2,0 \right\rbrace.
\end{eqnarray*}
Define the events $E_0^k$ and $E_s^k$ respectively by
\begin{equation}\label{Eok}
E^k_0=\accolade{\Lokx\leq h(x^k)\leq \Uokx}\ \text{and}\ \ E^k_s=\accolade{\Lskx\leq h(x^k+s^k)\leq \Uskx}
\end{equation}
By noticing that 
\begin{equation}\label{inclusion1}
\overset{m}{\underset{j=1}{\bigcap}}\left\lbrace \abs{\Cok-c_j(x^k)}\leq\efPB(\dpl)^2 \right\rbrace=\overset{m}{\underset{j=1}{\bigcap}}\accolade{\Cok-\efPB(\dpl)^2\leq c_j(x^k)\leq \Cok+\efPB(\dpl)^2}\subseteq E_0^k
\end{equation}
%and that 
\begin{equation}\label{inclusion2}
\overset{m}{\underset{j=1}{\bigcap}}\left\lbrace \abs{\Csk-c_j(x^k+s^k)}\leq\efPB(\dpl)^2 \right\rbrace \subseteq E_s^k,
\end{equation}
then combining respectively~\eqref{coalphademi} and~\eqref{inclusion1}, and~\eqref{csalphademi} and~\eqref{inclusion2}, yields 
\begin{equation}\label{Eosalphademi}
\pr{E_0^k}\geq\pr{\overset{m}{\underset{j=1}{\bigcap}}\left\lbrace \abs{\Cok-c_j(x^k)}\leq\efPB(\dpl)^2 \right\rbrace} \geq \alpha^{1/2}
\end{equation}
%and
\begin{equation}\label{Eskalphademi}
\pr{E_s^k}\geq\pr{\overset{m}{\underset{j=1}{\bigcap}}\left\lbrace \abs{\Csk-c_j(x^k+s^k)}\leq\efPB(\dpl)^2 \right\rbrace} \geq  \alpha^{1/2}.
\end{equation}
It follows from the independence of the random variables $\Theta^0_{j,\ell}$ and $\Theta^s_{j,\ell}$, for all $j\in J$ and for all $\ell=1,2,\dots,p^k_j$, that the events $E_0^k$ and $E_s^k$ are also independent. Hence, both inequalities~\eqref{Eosalphademi} and~\eqref{Eskalphademi} imply that 
\begin{eqnarray*}
	%\pr{E_0^k\cap E_s^k}=\pr{E_0^k}\times\pr{E_s^k}\geq\alpha,  \nonumber
	\alpha&\leq& \pr{\overset{m}{\underset{j=1}{\bigcap}}\left\lbrace \abs{\Cok-c_j(x^k)}\leq\efPB(\dpl)^2 \right\rbrace}\times \pr{\overset{m}{\underset{j=1}{\bigcap}}\left\lbrace \abs{\Csk-c_j(x^k+s^k)}\leq\efPB(\dpl)^2 \right\rbrace}\\
	&=& \pr{\overset{m}{\underset{j=1}{\bigcap}}\left\lbrace \abs{\Cok-c_j(x^k)}\leq\efPB(\dpl)^2 \right\rbrace \cap \left\lbrace \abs{\Csk-c_j(x^k+s^k)}\leq\efPB(\dpl)^2 \right\rbrace}\\
	&\leq&\pr{E_0^k}\times\pr{E_s^k}=\pr{E_0^k\cap E_s^k},
\end{eqnarray*}
which shows that Assumption~\ref{randomEstim}-{\it (v)} holds. 

In order to show that Assumption~\ref{randomEstim}-{\it (iv)} holds, notice that $\E{\Cok-c_j(x^k)}=0$ for all $j\in J$, which implies that for all $j\in J$,
\begin{equation}\label{espCarre1}
\E{\abs{\Cok-c_j(x^k)}^2}=\var{\Cok-c_j(x^k)}=\var{\Cok}\leq\frac{V}{p_j^k}\leq \efPB^2\left(1-\alpha^{1/2m} \right)(\dpl)^4,
\end{equation}
where the last inequality in~\eqref{espCarre1} follows from~\eqref{pkj}. Similarly, since $\E{\Csk-c_j(x^k+s^k)}=0$ for all $j\in J$, then 
\begin{equation}\label{espCarre2}
\E{\abs{\Csk-c_j(x^k+s^k)}^2}\leq  \efPB^2\left(1-\alpha^{1/2m} \right)(\dpl)^4,
\end{equation}
which shows that Assumption~\ref{randomEstim}-{\it (iv)} holds. 

Before showing Assumption~\ref{randomEstim}-{\it (vi)}, let first notice that
\begin{eqnarray}
\abs{\Hok-h(x^k)}&=&\abs{\sum_{j=1}^m\max\{\Cok,0\} - \sum_{j=1}^m\max\{c_j(x^k),0\}}\nonumber \\
&\leq&\sum_{j=1}^m \abs{\max\{\Cok,0\} - \max\{c_j(x^k),0\}}\leq \sum_{j=1}^m \abs{\Cok-c_j(x^k)},\label{ineq}
\end{eqnarray}
where the last inequality in~\eqref{ineq} follows from the inequality $\abs{\max\{x,0\} - \max\{y,0\}}\leq \abs{x-y}$, for all $x,y\in\R$. Moreover, it follows from the Cauchy-Schwarz inequality~\cite{bhattacharya2007basic} that for all $j\in J$,
\begin{equation}\label{esp3}
\E{\abs{\Cok-c_j(x^k)}}\leq\left[\E{\abs{\Cok-c_j(x^k)}^2}\right]^{1/2}\leq\efPB\left(1-\alpha\right)^{1/2}(\dpl)^2,
\end{equation}
where the last inequality in~\eqref{esp3} follows from~\eqref{espCarre1}.  Thus, taking the expectation in~\eqref{ineq}, combined with~\eqref{esp3} yields
\begin{equation}
\E{\abs{\Hok-h(x^k)}}\leq \sum_{j=1}^m\E{\abs{\Cok-c_j(x^k)}}\leq m\efPB\left(1-\alpha\right)^{1/2}(\dpl)^2, \nonumber
\end{equation}
\begin{equation}
\text{and similarly}\quad \E{\abs{\Hsk-h(x^k+s^k)}}\leq m\efPB\left(1-\alpha\right)^{1/2}(\dpl)^2, \nonumber
\end{equation}
which shows that Assumption~\ref{randomEstim}-{\it (vi)} holds. 

Finally, let compute estimates $\Fok$ and $\Fsk$ that satisfy Assumption~\ref{randomEstim}-{\it (i)} and {\it (ii)}. For that purpose, let $\Theta_0^0$ and $\Theta_0^s$ be two independent random variables following the same distribution as $\Theta_0$. Let $\Theta^0_{0,\ell},\  \ell=1,2,\dots,p^k_0$ and $\Theta^s_{0,\ell},\  \ell=1,2,\dots,p^k_0$ be independent random samples of $\Theta_0^0$ and $\Theta_0^s$ respectively, where $p_0^k\geq 1$ denotes the sample size. Define $\Fok$ and $\Fsk$ respectively by  

\[\Fok=\frac{1}{p^k_0}\sum_{\ell=1}^{p^k_0}f_{\Theta^0_{0,\ell}}(x^k)\quad \text{and}\quad  \Fsk=\frac{1}{p^k_0}\sum_{\ell=1}^{p^k_0}f_{\Theta^s_{0,\ell}}(x^k+s^k).\]
Then $\E{\Fok}=f(x^k)$, which implies that $\var{\Fok}\leq\frac{V}{p^k_0}$. Thus, it is easy to notice that the proof of Assumption~\ref{randomEstim}-{\it (i)} follows that of Assumption~\ref{randomEstim}-{\it (iii)}. More precisely, the following inequality holds: 
\begin{equation}\label{fosbetaPB}
\pr{\accolade{\abs{\Fok-f(x^k)}\leq \efPB(\dpl)^2}\cap \accolade{\abs{\Fsk-f(x^k+s^k)}\leq \efPB(\dpl)^2}}\geq \beta,
\end{equation}
provided that 
\begin{equation}\label{pk0}
p_0^k\geq \frac{V}{\efPB^2\left(1-\sqrt{\beta} \right)(\dpl)^4}
\end{equation}
Estimates $\fok=\Fok(\omega)$ and $\fsk=\Fsk(\omega)$, obtained by averaging $\, p^k_0$ realizations of $f_{\Theta_0}$, resulting from the evaluations of the stochastic blackbox, respectively at $x^k$ and $x^k+s^k$, are obviously $\efPB$-accurate. It is also easy to notice that the proof of Assumption~\ref{randomEstim}-{\it (ii)} follows that of Assumption~\ref{randomEstim}-{\it (iv)}. Specifically, 
\begin{equation}\label{espfCarre2}
\E{\abs{\Fok-f(x^k)}^2}\leq  \efPB^2(1-\sqrt{\beta})(\dpl)^4\quad\text{and}\quad \E{\abs{\Fsk-f(x^k+s^k)}^2}\leq  \efPB^2(1-\sqrt{\beta})(\dpl)^4,\nonumber
\end{equation}
provided that $p^k_0$ is chosen according to~\eqref{pk0}.
%\begin{equation}\label{fsbeta}
%\pr{\abs{\Csk-c_j(x^k+s^k)}\leq \efPB(\dpl)^2}\geq \alpha^{1/2m}
%\end{equation}

\section{Convergence analysis}\label{convAnalysisPB} 
Using ideas inspired by~\cite{AuDe09a,audet2019stomads,chen2018stochastic,LaBi2016,paquette2018stochastic} this section presents convergence results of StoMADS-PB, most of which are stochastic variants of those in~\cite{AuDe09a}. It introduces the random time $T$ at which Algorithm~\ref{algoPB} generates a first $\efPB$-feasible solution. Then assuming that $T$ is either almost surely finite or almost surely infinite, a so-called {\bl zeroth-order} result~\cite{AuDeLe08,audet2019stomads} is derived showing that there exists a subsequence of Algorithm~\ref{algoPB}-generated
random iterates with mesh realizations becoming infinitely fine and which converges with probability one to a limit.
This is achieved by showing by means of Theorem~\ref{zerothorderTheorem} that the sequence of random poll size parameters converges to zero with probability one. Section~\ref{forh} analyzes the function $h$ and the random $\efPB$-infeasible iterates generated by Algorithm~\ref{algoPB}. In particular, it gives conditions under which an  almost sure limit of a subsequence of such iterates is shown in Theorem~\ref{resulth} to satisfy a first-order necessary optimality condition via the Clarke generalized derivative of $h$ with probability one. Then, a similar result for $f$ and the sequence of $\efPB$-feasible iterates is derived in Theorem~\ref{resultf} of Section~\ref{forf}. {\bl Note finally that the proofs of the main results of this section are presented in the Appendix.}

%------------------------------------------------%

\subsection{Zeroth-order convergence}\label{subZeroOrder}

Recall Remark~\ref{remarkFCF} and denote by $\mathscr{S}^k_X=\{X^{\ell}_{\text{feas}}:X^{\ell}_{\text{feas}}\neq\xzeroi,\ \ell\leq k\}$ the set of all random $\efPB$-feasible  iterates generated by Algorithm~\ref{algoPB} until the beginning of iteration~$k$. Consider the following random time $T$ defined by 
\begin{equation}\label{randomTimeT}
T:=\inf\{k\geq 0:\mathscr{S}^k_X\neq\emptyset\}.
\end{equation}
Then it is easy to notice that $T\geq 1$ and that for all $k\geq 1$, the occurrence of the event $\{T\leq k\}$ is determined by observing the random quantities generated by Algorithm~\ref{algoPB} until the iteration $k-1$, which means that $T$ is a {\bl stopping time}~\cite{durrett2010probability} for the stochastic process generated by Algorithm~\ref{algoPB}. The following is assumed for the remainder of the analysis.

\begin{assumption}\label{stoppingAssumption}
	The stopping time $T$ associated to the stochastic process generated by Algorithm~\ref{algoPB} is either almost surely finite or almost surely infinite.
\end{assumption}

The next result implies that the sequence $\{\Dp\}_{k\in\N}$ of random poll size parameters converges to zero with probability one and will be useful for the Clarke stationarity results of Sections~\ref{forh} and~\ref{forf}. It holds under the assumption below.

\begin{assumption}\label{fisbounded}
	The objective function $f$ is bounded from below, i.e., there exists $\kappa^f_{\min}\in\R$ such that $-\infty<\kappa^f_{\min}\leq f(x)$, for all $x\in\rn$.
\end{assumption}

\begin{theorem}\label{zerothorderTheorem}
	Let Assumptions~\ref{lipschitzAssumption}, \ref{stoppingAssumption} and~\ref{fisbounded} be satisfied. Let $\gamma>2$ and $\tau\in(0,1)\cap\mathbb{Q}$. Let $\nu\in(0,1)$ be chosen such that
	\begin{equation}\label{nuChoice}
	\frac{\nu}{1-\nu}\geq \frac{2(\tau^{-2}-1)}{\gamma-2}
	\end{equation}
	and assume that Assumption~\ref{randomEstim} holds for $\alpha$ and $\beta$ chosen such that 
	\begin{equation}\label{betaChoice}
	\alpha\beta\geq \frac{4\nu}{(1-\nu)(1-\tau^2)}\left[(1-\alpha)^{1/2}+2(1-\beta)^{1/2} \right].
	\end{equation}
	Then, the sequence $\{\Dp\}_{k\in\N}$ of frame size parameters generated by Algorithm~\ref{algoPB} satisfies
	\begin{equation}\label{deltaSeries}
	\sum_{k=0}^{+\infty}(\Dp)^2<+\infty\quad\text{almost surely}.
	\end{equation}
\end{theorem}

The following result is a simple consequence of  Theorem~\ref{zerothorderTheorem}. It shows that the sequences $\{\Dm\}_{k\in\N}$ and $\{\Dp\}_{k\in\N}$ converge to zero almost surely respectively.
\begin{corollary}\label{meshframeConv}
	The followings hold under all the assumptions made in Theorem~\ref{zerothorderTheorem}
	%Let all assumptions that were made in Theorem~\ref{zerothorderTheorem} hold. Let $\{\Dm\}_{k\in\N}$ be the sequence of random mesh size parameters and $\{\Dp\}_{k\in\N}$ be the sequence of random frame size parameters generated by Algorithm~\ref{algoPB}. Then the followings hold. 
	\begin{equation}
	\underset{k\to+\infty}{\lim}\Dm=0\ \text{almost surely}\quad\text{and}\quad \underset{k\to+\infty}{\lim}\Dp=0\ \text{almost surely}.\nonumber
	\end{equation}
\end{corollary}
The next result shows that with probability one, the difference between the estimates and their corresponding true function values converge to zero. This means that Algorithm~\ref{algoPB} behaves like an exact deterministic method asymptotically. This result will be also useful in Subsection~\ref{forf} for the proof of Theorem~\ref{xhatD}.
\begin{corollary}\label{corHFok}
	Let all assumptions that were made in Theorem~\ref{zerothorderTheorem} hold. Then, 
	\begin{equation}\label{hfzero}
	\underset{k\to+\infty}{\lim}\abs{\Hok-h(X^k)}=0\ \text{almost surely}\quad\text{and}\quad \underset{k\to+\infty}{\lim}\abs{\Fok-f(X^k)}=0\ \text{almost surely},
	\end{equation}
	and the same result holds for $\abs{\Hsk-h(X^k+S^k)}$ and $\abs{\Fsk-f(X^k+S^k)}$ respectively.
\end{corollary}

\begin{definition}
	A convergent subsequence $\{x^k\}_{k\in \mathcal{K}}$ of Algorithm~\ref{algoPB} iterates, for some subset of indices~$\mathcal{K}$, is called a refining subsequence if and only if the corresponding subsequence $\{\dm\}_{k\in \mathcal{K}}$ converges to zero. The limit $\hat{x}$ is called a refined point.
\end{definition}
Combining the results of Corollary~\ref{meshframeConv} and the compactness hypothesis of Assumption~\ref{lipschitzAssumption} was shown in~\cite{audet2019stomads} to be enough to ensure the existence of refining subsequences. Specifically the following holds.
\begin{theorem}
	Let the assumptions that were made in Corollary~\ref{meshframeConv} hold. Then there exists at least one refining subsequence $\{X^k\}_{k\in K}$ (where $K$ is a sequence of random variables) which converges almost surely to a refined point $\hat{X}$.
\end{theorem}

%$\hat{X}_{\textnormal feas}$
\subsection{Nonsmooth optimality conditions: {\bl Results} for $h$}\label{forh}
This subsection aims to show with probability one that Algorithm~\ref{algoPB} generates a refining subsequence $\{\Xki\}_{k\in K}$ with refined point $\hat{X}_{\inf}$ which satisfies a first-order necessary optimality condition via the Clarke generalized derivative of $h$. As in~\cite{audet2019stomads}, this optimality result strongly relies on the requirement that the polling directions $d^k\in\mathbb{D}^k_p(\xki)$ of Algorithm~\ref{algoPB} are such that $\dpl\norminf{d^k}$ never approaches zero for all $k$. The way such an expectation can be met is discussed in~\cite{audet2019stomads}. Indeed, by choosing the columns of the matrix $\mathbf{D}$ used in the definition of the mesh $\mathcal{M}^k$ to be the $2n$ positive and negative coordinate directions, $\delta_p^0=1$ and $\tau=1/2$, the directions $\dpl d^k$ were shown in~\cite{audet2019stomads} to satisfy $\dpl\norminf{d^k}\geq 1$ whenever $d^k$ is constructed by means of the so-called Householder matrix~\cite{AuHa2017}. Thus, the following assumption is made for the remainder of the analysis.
\begin{assumption}\label{dmin}
	Let $d^k\in\mathbb{D}^k_p$ be any polling direction used by Algorithm~\ref{algoPB} at iteration $k$. Then there exists a constant $d_{\min}>0$ such that $\dpl\norminf{d^k}\geq d_{\min}\ $ for all $k\geq 0$.
\end{assumption}

%\begin{definition}
%
%\end{definition}
%The following auxiliary result~\cite{BaScVi2014,chen2018stochastic} is taken from the martingale literature~\cite{durrett2010probability} and will be useful latter in the analysis.
%\begin{theorem}\label{durett}
%	Let $\{G_k\}_{k\in\N}$ be a submartingale, i.e, a sequence of random variables which, for every $k\in\N$, satisfy 
%	\[\E{G_k|\mathcal{F}^G_{k-1}}\geq G_{k-1},\] 
%	where $\mathcal{F}^G_{k-1}=\sigma(G_0,G_1,\dots,G_{k-1})$ is the $\sigma$-algebra generated by $G_0,G_1,\dots,G_{k-1}$, and $\Esp(G_k|\mathcal{F}^G_{k-1})$ denotes the conditional expectation of $G_k$, given the past history of events $\mathcal{F}^G_{k-1}$.
%	
%	Assume further that $G_k-G_{k-1}\leq M < +\infty$, for every $k$. Then, 
%	\begin{equation}\label{submartingale}
%	\pr{\left\lbrace \underset{k\to\infty}{\lim} G_k < \infty \right\rbrace \cup \left\lbrace \underset{k\to\infty}{\limsup}\ G_k = \infty \right\rbrace} = 1.\nonumber
%	\end{equation}
%\end{theorem}

The main result of this subsection relies on the properties of the random function $\Psi_k^h$ introduced next, a similar of which was used in~\cite{audet2019stomads}. 
\begin{lemma}\label{Psikh}
	Let the same assumptions that were made in Theorem~\ref{zerothorderTheorem} hold and assume in addition to~\eqref{betaChoice} that $\alpha\beta\in(1/2,1)$. Consider the random function $\Psi_k^h$ with realizations $\psi_k^h$ defined by
	\begin{equation}%\label{psikh}
	\psi_k^h:=\frac{h(\xki)-h(\xki+\dm d^k)}{\dpl}\quad\text{for all}\ k\geq 0,  \nonumber
	\end{equation}
	%(\xkf)
	where $d^k\in\mathbb{D}^k_p(\xki)$ denotes any available polling direction around $\xki$ at iteration $k$. Then the following holds, 
	\begin{equation}\label{liminfPsikh}
	\underset{k\to+\infty}{\liminf}\ \Psi_k^h\leq 0 \ \text{almost surely.}
	\end{equation}
\end{lemma}

The following definition of refining directions~\cite{AuDe2006,AuHa2017} will be useful in the analysis.
\begin{definition}\label{refiningDirPB}
	Let $\hat{x}$ be the refined point associated to a convergent refining subsequence $\{x^k\}_{k\in \mathcal{K}}$. A~direction $v$ is said to be a refining direction for $\hat{x}$ if and only if there exists an infinite subset $\mathcal{L}\subseteq \mathcal{K}$ with polling directions $d^k\in\mathbb{D}^k_p(x^k)$ such that $v=\underset{k\in \mathcal{L}}{\lim}\frac{d^k}{\norminf{d^k}}$.
\end{definition}
The analysis in this subsection also relies on the following definitions~\cite{AuDe09a}. The Clarke generalized derivative $h^{\circ}(\hat{x};v)$ of $h$ at $\hat{x}\in\X$ in the direction $v\in\rn$ is defined by
\begin{equation}\label{ClarkeDef}
h^{\circ}(\hat{x};v):=\underset{t\searrow 0,\ y+tv\in\X}{\underset{y\to \hat{x},\ y\in\X}{\limsup}} \frac{h(y+tv)-h(y)}{t}.
\end{equation}
As highlighted in~\cite{AuDe09a}, this definition from~\cite{Jahn94a} is a generalization of the original one by Clarke~\cite{Clar83a} to the case where the constraints violation function $h$ is not defined outside $\X$. 

The analysis involves a specific cone $T^H_{\X}(\hat{x}_{\inf})$ called the hypertangent cone~\cite{Rock80a} to $\X$ at $\hat{x}_{\inf}$. The hypertangent cone to a subset $\mathcal{O}\subseteq\X$ at~$\hat{x}$ is defined by
\begin{equation}
T^H_{\mathcal{O}}(\hat{x}):=\{v\in\rn:\exists\bar{\epsilon}>0\ \text{such that}\  y+tw\in\mathcal{O}\ \forall y\in\mathcal{O}\cap\mathcal{B}_{\bar{\epsilon}}(\hat{x}), w\in\mathcal{B}_{\bar{\epsilon}}(v)\ \text{and}\ 0<t<\bar{\epsilon} \}. \nonumber
\end{equation}
%As in~\cite{AuDe2006,AuDe09a}, in the present manuscript this cone is the most important tangent cone for the analysis of Algorithm~\ref{algoPB}. 
Next is stated a lemma~\cite{AuDe09a} from elementary analysis, that will be useful latter in the present analysis.
\begin{lemma}\label{akbk}
	If $\{a_k\}$ is a bounded real sequence and $\{b_k\}$ is a convergent real sequence, then 
	\begin{equation}
	\limsup_k(a_k+b_k)=\limsup_k a_k+\lim_k b_k. \nonumber
	\end{equation}
\end{lemma}
The next result is a stochastic variant of Theorem~3.5 in~\cite{AuDe09a}. Since the inequality $h(\xki+\dm d^k)-h(\xki)\geq 0$ on which relies the latter theorem does not hold in the present stochastic setting, then the proof of the result below is based on the random function $\Psi_k^h$ $\liminf$-type result of Lemma~\ref{Psikh}.
\begin{theorem}\label{resulth}
	Let Assumptions~\ref{assumptA1}, \ref{dmin} and all the assumptions made in Theorem~\ref{zerothorderTheorem} and Lemma~\ref{Psikh} hold. Then Algorithm~\ref{algoPB} generates a convergent $\efPB$-infeasible refining subsequence $\{\Xki\}_{k\in K}$, for some sequence $K\subseteq K'$ of random variables satisfying $\lim_{K'} \Psi_k^h\leq 0$ almost surely, such that if $\hat{x}_{\inf}\in\X$ is a refined point for a realization $\{\xki\}_{k\in \mathcal{K}}$ of $\{\Xki\}_{k\in K}$ for which the events $\Dp\to 0$ and $\lim_{K'} \Psi_k^h\leq 0$ both occur,
	% and assume that $h$ is Lipschitz near $\hat{x}_{\inf}$. 
	and if $v\in T^H_{\X}(\hat{x}_{\inf})$ is a refining direction for $\hat{x}_{\inf}$, then $h^{\circ}(\hat{x}_{\inf};v)\geq 0$. In particular, this means that
	%\vspace{-0.35cm}
	\begin{equation}\label{pvtrois}
	\begin{split}
	\pb\left( \left\lbrace \omega\in\Omega:\exists K(\omega)\subseteq\N\ \text{and}\ \exists\hat{X}_{\inf}(\omega)\right.\right.&= \lim_{k\in K(\omega)}\Xki(\omega), \hat{X}_{\inf}(\omega)\in\X, \ \text{such that} \\
	\forall V(\omega)\in& \left.\left. T^H_{\mathcal{X}}(\hat{X}_{\inf}(\omega)),\   h^{\circ}(\hat{X}_{\inf}(\omega);V(\omega))\geq 0\right\rbrace\right)=1. 
	\end{split} 
	\end{equation}
\end{theorem}

Next is stated a stochastic variant of a result in~\cite{AuDe09a}, showing that Clarke stationarity is ensured when the set of refining directions is dense in a nonempty hypertangent cone to $\X$.
\begin{corollary}
	%Let Assumptions~\ref{assumptA1}, \ref{dmin} and all assumptions that were made in Theorems~\ref{zerothorderTheorem} and~\ref{Psikh} hold. 
	Let all assumptions that were made in Theorem~\ref{resulth} hold.
	Let $\{\Xki\}_{k\in K}$ be the $\efPB$-infeasible refining subsequence of Theorem~\ref{resulth}, with realizations $\{\xki\}_{k\in \mathcal{K}}$ which converges to a refined point $\hat{x}_{\inf}\in\X$. 
	%Then Algorithm~\ref{algoPB} generates a convergent $\efPB$-infeasible refining subsequence $\{\Xki\}_{k\in K}$, for some sequence $K\subseteq K'$ of random variables satisfying $\lim_{K'} \Psi_k^h\leq 0$ almost surely, such that if $\hat{x}_{\inf}\in\X$ is a refined point for a realization $\{\xki\}_{k\in \mathcal{K}}$ of $\{\Xki\}_{k\in K}$ for which both events $\Dp\to 0$ and $\lim_{K'} \Psi_k^h\leq 0$ occur,
	% and assume that $h$ is Lipschitz near $\hat{x}_{\inf}$. 
	%and 
	If the set of refining directions for $\hat{x}_{\inf}$ is dense in $T^H_{\X}(\hat{x}_{\inf})\neq\emptyset$, then $\hat{x}_{\inf}$ is a Clarke stationary point for the problem $\ \ds{\min_{x\in\X}h(x)}$.
	%\begin{equation}
	%\min_{x\in\X}h(x). \nonumber
	%\end{equation}
\end{corollary}

\begin{proof}
	The proof of this result is almost identical to the proof of a similar result (Corollary 3.6) in~\cite{AuDe09a} and hence will not be presented here again. 
	%Recall the almost sure event $E_1\cap E_2$ of the proof of Theorem~\ref{resulth}. 
\end{proof}
%The main goal of this section is to show with probability one that any re?ned point X? derived in Theorem 2 satis?es a stochastic variant of the ?rst-order necessary optimality condition based on the Clarke derivative stated as Theorem 6.9 in
\subsection{Nonsmooth optimality conditions: {\bl Results} for $f$}\label{forf}
The analysis presented in this subsection assumes that Algorithm~\ref{algoPB} generates infinitely many $\efPB$-feasible points. It aims to show with probability one that StoMADS-PB generates a refining subsequence $\{\Xkf\}_{k\in K}$ with refined point $\Xhatf$, which satisfies a first-order necessary optimality condition based on the Clarke derivative of $f$. The following lemma will be useful latter in the analysis.
\begin{lemma}\label{Psikf}
	Let the same assumptions that were made in Theorem~\ref{zerothorderTheorem} hold and assume in addition to~\eqref{betaChoice} that $\alpha\beta\in(1/2,1)$. Assume that the random time $T$ with realizations $t$ is finite almost surely. Consider the random function $\Psi_k^{f,T}$ with realizations $\psi_k^{f,t}$  defined by
	\begin{equation}\label{psikf}
	\psi_k^{f,t}:=\frac{f(\xkft)-f(\xkft+\dm d^k)}{\dpl}\quad\text{for all}\ k\geq 0,  \nonumber
	\end{equation}
	where $k\vee t:=\max\{k,t\}$ and $d^k$ denotes any available polling direction around $\xkft$ at iteration~$k$. Then the following holds, 
	\begin{equation}\label{liminfPsikf}
	\underset{k\to+\infty}{\liminf}\ \Psi_k^{f,T}\leq 0 \ \text{almost surely.} 
	\end{equation}
\end{lemma}

Now let prove that the almost sure limit $\Xhatf$ of any convergent refining subsequence of $\efPB$-feasible iterates which drives the random estimated violations $\Hok(\Xkf)$ to zero almost surely, satisfies $\pr{\Xhatf\in\D}=1$. First, notice that the existence of such a refining subsequence can be assumed. Indeed, it is known from Theorem~\ref{trueIterations} that true iterations occur infinitely often provided that estimates and bounds are sufficiently accurate. In addition, every $\efPB$-feasible point $\xkf$ newly accepted by Algorithm~\ref{algoPB} satisfies $\uok(\xkf)=0$, which implies that $\hok(\xkf)=0$, thus leading to the overall conclusion that $\underset{k\to+\infty}{\liminf}\ \Hok(\Xkf)=0$ almost surely, which is implicitly assumed next.
\begin{theorem}\label{xhatD}
	Let all the assumptions of Lemma~\ref{Psikf} hold. Let $\Xhatf$ be the almost sure limit of a convergent $\efPB$-feasible refining subsequence $\{\Xkft\}_{k\in K}$ for which $\underset{k\in K}{\lim} \Hok(\Xkft)=0$ almost surely. Then
	\begin{equation}\label{XinD} 
	\pr{\Xhatf\in\D}=1.
	\end{equation} 
\end{theorem}

The following result is a stochastic variant of Theorem~3.3 in~\cite{AuDe09a}.
\begin{theorem}\label{resultf}
	Let Assumptions~\ref{assumptA1}, \ref{dmin} and all assumptions that were made in Theorem~\ref{zerothorderTheorem} and Lemma~\ref{Psikf} hold. 
	%Assume that Algorithm~\ref{algoPB} generates infinitely many $\efPB$-feasible points. 
	Let $\{\Xkft\}_{k\in K}$ be an almost surely convergent $\efPB$-feasible refining subsequence, for some sequence $K$ of random variables satisfying $\lim_{K} \Psi_k^{f,T}\leq 0$ and $\lim_K \Hok(\Xkft)=0$ almost surely. Then, if $\xhatf\in\D$ is a refined point for a realization $\{\xkft\}_{k\in \mathcal{K}}$ of $\{\Xkft\}_{k\in K}$ for which the events $\Dp\to 0$, $\lim_{K} \Psi_k^{f,T}\leq 0$ and $\lim_K \Hok(\Xkft)=0$ occur, and if $v\in T^H_{\D}(\xhatf)$ is a refining direction for $\xhatf$, then $f^{\circ}(\xhatf;v)\geq 0$. In particular, this means that
	%\vspace{-0.35cm}
	\begin{equation}\label{pvquatre}
	\begin{split}
	\pb\left( \left\lbrace \omega\in\Omega:\exists K(\omega)\subseteq\N\ \text{and}\ \exists\Xhatf(\omega)\right.\right.&= \lim_{k\in K(\omega)}\Xkft(\omega),  \Xhatf(\omega)\in\D,\ \text{such that} \\
	\forall V(\omega)\in& \left.\left. T^H_{\D}(\Xhatf(\omega)),\   f^{\circ}(\Xhatf(\omega);V(\omega))\geq 0\right\rbrace\right)=1. 
	\end{split} 
	\end{equation}
\end{theorem}

\begin{corollary}
	Let all assumptions that were made in Theorem~\ref{resultf} hold.
	Let $\{\Xkft\}_{k\in K}$ be the $\efPB$-feasible refining subsequence of Theorem~\ref{resultf}, with realizations $\{\xkft\}_{k\in \mathcal{K}}$ which converges to a refined point $\xhatf\in\D$. 
	If the set of refining directions for $\xhatf$ is dense in $T^H_{\D}(\xhatf)\neq\emptyset$, then $\xhatf$ is a Clarke stationary point for~\eqref{problemPB}.
	%	\begin{equation}
	%	\min_{x\in\D}f(x). \nonumber
	%	\end{equation}
\end{corollary}

\begin{proof}
	The proof of this result is almost identical to the proof of a similar result (Corollary 3.4) in~\cite{AuDe09a} and hence will not be presented here again. 
\end{proof}

%\subsection{\rd{Nonsmooth optimality conditions: results on $f$ and $h$}}

\section{Computational study}\label{sec5}

This section {\bl illustrates} the performance and the efficiency of StoMADS-PB using noisy variants of $42$ continuous analytical computational constrained problems from the optimization literature. The sources and characteristics of these problems are summarized in Table~\ref{tabprob}. The number of variables ranges from $n=2$ to $n=20$, where every problem has at least one constraint ($m>0$) other than bound constraints. In order to show the capability of StoMADS-PB to cope with noisy constrained problems {\bl compared to} MADS with PB~\cite{AuDe09a} referred to as MADS-PB, the latter algorithm is compared to several variants of StoMADS-PB. 
%since there is no noisy constrained blackbox optimization algorithm in the {\sf NOMAD}~\cite{Le09b} software package.
For all numerical investigations of both algorithms, only the POLL step is used, i.e., no SEARCH step is involved. The OrthoMADS-$2n$ directions~\cite{AbAuDeLe09} are used for the POLL which is ordered by means of an opportunistic strategy~\cite{AuHa2017}. MADS-PB and all the proposed variants of StoMADS-PB are implemented in {\sf MATLAB}.

The stochastic variants of the $42$ abovementioned deterministic constrained optimization problems are solved using three different infeasible initial points for a total of $126$ problem instances. Inspired from~\cite{audet2019stomads}, such stochastic variants are constructed by additively perturbing the objective $f$ by a random variable $\Theta_0$ and each constraint $c_j, j=1,2,\dots,m$ by a random variable $\Theta_j$ as follows
\begin{equation}
f_{\Theta_0}(x)=f(x)+\Theta_0\quad\text{and}\quad c_{\Theta_j}(x)=c_j(x)+\Theta_j, \ \text{for all}\ j\in J,
\end{equation}
where $\Theta_0$ is uniformly generated in the interval $I(\sigma,x^0,f)=\left[-\sigma\abs{f(x^0)-f^*}, \sigma\abs{f(x^0)-f^*}\right]$ and $\Theta_j$ is uniformly generated in $I(\sigma,x^0,c_j)=\left[-\sigma\abs{c_j(x^0)}, \sigma\abs{c_j(x^0)}\right]$. The scalar $\sigma>0$ is used to define different noise levels, $x^0$ denotes an initial point and $f^*$ is the best known feasible minimum value of~$f$. The random variables $\Theta_0, \Theta_1,\dots,\Theta_m$ are independent. For the remainder of the study, the process which returns the vector $\left[f_{\Theta_0}(x), c_{\Theta_1}(x), c_{\Theta_2}(x),\dots,c_{\Theta_m}(x) \right]$ when provided the input $x$ will be referred to as noisy blackbox.

The MADS-PB algorithm~\cite{AuDe09a} of which StoMADS-PB is a stochastic variant and to which the latter is compared is an iterative direct-search method originally developed for deterministic constrained blackbox optimization. In MADS-PB, feasibility is sought by progressively decreasing in an adaptive manner a threshold imposed on a constraint violation function into which all the constraint violations are aggregated. Any trial point with a constraint violation value greater than that threshold  is rejected out of hand. 
%Note also that during an iteration $k$, the MADS-PB algorithm explores around two incumbent solutions: a feasible one with the best objective function value and an {\it undominated}~\cite{AuHa2017} infeasible one with the best objective function value and its constraint violation function value less than the aforementioned threshold. Such iteration can be declared to be either {\it dominating} or {\it improving}, or {\it unsuccessful}.
Full description of MADS-PB iterations and useful information for better understanding of the algorithm behavior can also be found in~\cite{AuHa2017}. 

The relative performance and efficiency of algorithms are assessed by performance profiles~\cite{DoMo02,MoWi2009} and data profiles~\cite{MoWi2009}, which require to define for a given computational problem a convergence test. For each of the $126$ problems, denote by $x^N$ the best feasible iterate found after $N$ evaluations of the noisy blackbox and let $x^*$ be the best feasible point obtained by all tested algorithms on all run instances. Then, the convergence test from~\cite{AuLeDTr2018} used for the experiments is defined as follows:
\begin{equation}\label{convtest}
f(x^N)\leq f(x^*)+\tau(\bar{f}_{\textnormal feas}-f(x^*)), 
\end{equation}
where, $\tau\in [0,1]$ is the convergence tolerance and $\bar{f}_{\textnormal feas}$ is a reference value obtained by taking the average of the first feasible $f$ function values over all run instances of a given computational problem for all algorithms. If no feasible point is found, then the convergence test fails. Otherwise, a problem is said to be successfully solved within the tolerance $\tau$ if~\eqref{convtest} holds. As highlighted in~\cite{AuLeDTr2018}, $\bar{f}_{\textnormal feas}=f(x^0)$ for unconstrained computational problems, where $x^0$ denotes the initial point. 

The horizontal axis of the performance profiles shows the ratio of the number of noisy objective function evaluations while the fraction of computational problems solved within the convergence tolerance $\tau$ is shown on the vertical axis. On the horizontal axis of the data profiles is shown the number of function calls to the noisy blackbox divided by $(n+1)$\footnote{$n+1$ is the number of evaluations required to construct a linear interpolant or a {\bl simplex gradient}~\cite{AuHa2017} in $\rn$~\cite{AuLeDTr2018,MoWi2009}.} while the vertical axis shows the proportion of computational problems solved by all run instances of a given algorithm within a tolerance~$\tau$. As emphasized in~\cite{AuHa2017}, performance profiles capture information on speed of convergence (i.e., the quality of a given algorithm's output in terms of the objective function evaluations) and robustness (i.e., the fraction of computational problems solved) in a compact graphical format, while data profiles also examine the robustness and efficiency from a different perspective. 

Now recall that in StoMADS-PB, according to Section~\ref{computationPB}, the noisy blackbox needs to be evaluated many times at a given point in order to compute function estimates unlike the MADS-PB method where it is evaluated only once at each point. But since a limited budget of $1000(n+1)$ noisy blackbox evaluations is set in all the experiments, that is, since MADS-PB and all variants of StoMADS-PB stop as soon as the number of noisy blackbox evaluations reaches $1000(n+1)$, only few calls to the blackbox need to be used when computing StoMADS-PB function estimates. However, given that such estimates are required to be sufficiently accurate in order for the solutions to be satisfactory, a procedure inspired from~\cite{audet2019stomads} aiming at improving the estimates accuracy by making use of available samples at a given current point is proposed. Note in passing that the proposed computation procedure is very efficient in practice as highlighted in~\cite{audet2019stomads} even though it is {\bl inherently biased}. The following computation scheme is described only for $\fokx$ but is the same for $\fskx$, $\cokx$ and $\cskx$, for all $j\in J$. First, let mention that during the optimization, all trial points $x^k$ used by StoMADS-PB and all corresponding values $f_{\Theta_0}(x^k)$ are stored in a cache. When constructing an estimate of $f(x^k)$ at the iteration $k\geq 1$, denote by $a^k(x^k)$\footnote{It is implicitly assumed without any loss of generality that $a^k(x^k)\geq 1$.} the number of sample values of $f_{\Theta_0}(x^k)$ available in the cache from previous blackbox evaluations until iteration $k-1$. Since all the values of the noisy objective function $f_{\Theta_0}$ are always computed independently of each other, the aforementioned sample values can be considered as independent realizations $f_{\theta_{0,1}}(x^k), f_{\theta_{0,2}}(x^k),\dots,f_{\theta_{0,a^k(x^k)}}(x^k)$ of $f_{\Theta_0}(x^k)$, where for all $\ell=1,2,\dots,a^k(x^k)$, $\theta_{0,\ell}$ is a realization of the random variable $\Theta_{0,\ell}$ following the same distribution as $\Theta_0$. Now let $n^k\geq 1$ be the number of blackbox evaluations at $x^k$ and consider the following independent realizations $\theta_{0,a^k(x^k)+1}, \theta_{0,a^k(x^k)+2},\dots, \theta_{0,a^k(x^k)+n^k}$ of $\Theta_0$. Then, an estimate $\fokx$ of $f(x^k)$ is computed according to,
\begin{equation}\label{scheme}
\fokx=\frac{1}{p^k}\sum_{\ell=1}^{p^k}f_{\theta_{0,\ell}}(x^k),
\end{equation}
where $p^k=n^k+a^k(x^k)$ is the sample size.

Same values are used to initialize most of the common parameters to StoMADS-PB and MADS-PB. Specifically, the mesh refining parameter $\tau=1/2$, the frame center trigger $\rho=0.1$ and $\delta_m^0=\delta_p^0=1$. Nevertheless in MADS-PB, the initial barrier threshold is set equal its default value, i.e., $h^0_{\max}=+\infty$~\cite{AuDe09a} while in StoMADS-PB it equals $u^0_0(\xzeroi)$, with $\uokx$ defined in~\eqref{epspr} for all $k\in\N$. The default values of Algorithm~\ref{algoPB} parameters $\gamma>2$ and $\efPB>0$\footnote{The use of $\varepsilon_f$ instead of $\efPB$ is favored in~\cite{audet2019stomads}.} are borrowed from~\cite{audet2019stomads} in which StoMADS, an unconstrained stochastic variant of MADS~\cite{AuDe2006} is introduced. Specifically, $\gamma=17$ and $\efPB=0.01$.

%In order to improve the 

\begin{table}[ht!]
	
	\caption{Description of the set of $42$ analytical problems.}
	\label{tabprob}
	\centering
	\begin{tabular}{llllllllllll}
		\hline
		No   & Name         & Source                              & $n$    & $m$   & Bnds & No   & Name         & Source                              & $n$    & $m$    & Bnds \\ \hline
		$1$  & ANGUN       & \cite{wang2018constrained} & $2$  & $1$ & Yes  & $22$ & MAD1         & \cite{LuVl00}    & $2$  & $1$  & No\\
		$2$  & BARNES       & \cite{RodRenWat98} & $2$  & $3$ & Yes  & $23$ & MAD2         & \cite{LuVl00}    & $2$  & $1$  & No\\
		$3$  & BERTSIMAS       & \cite{bertsimas2010nonconvex} & $2$  & $2$ & No  & $24$ & MAD6         & \cite{LuVl00}    & $7$  & $7$  & Yes\\
		$4$  & CHENWANG\_F2       & \cite{ChWa2010} & $8$  & $6$ & Yes  & $25$ & MEZMONTES         & \cite{MezCoe05}    & $2$  & $2$  & Yes\\
		$5$  & CHENWANG\_F3       & \cite{ChWa2010} & $10$  & $8$ & Yes  & $26$ & NEW-BRANIN         & \cite{wang2018constrained}    & $2$  & $1$  & Yes\\
		$6$  & CONSTR-BRANIN       & \cite{wang2018constrained} & $2$  & $1$ & Yes  & $27$ & OPTENG-BENCH4         & \cite{KiArYa2011}    & $2$  & $1$  & Yes	\\
		$7$  & CRESCENT       & \cite{AuDe09a} & $10$  & $2$ & No  & $28$ & OPTENG-BENCH5         & \cite{KiArYa2011}    & $2$  & $3$  & Yes\\
		$8$  & DEMBO5       & \cite{LuVl00} & $8$  & $3$ & Yes  & $29$ & OPTENG-RBF         & \cite{KiArYa2011}    & $3$  & $4$  & Yes\\
		$9$  & DISK       & \cite{AuDe09a} & $10$  & $1$ & No  & $30$ & PENTAGON         & \cite{LuVl00}   & $6$  & $15$  & No\\
		$10$  & G23       & \cite{AuDeLe07}    & $3$  & $2$ & Yes  & $31$ & PRESSURE-VESSEL         & \cite{MezCoe05}   & $4$  & $4$  & Yes\\
		$11$  & G210       & \cite{AuDeLe07}    & $10$  & $2$ & Yes  & $32$ & SASENA         & \cite{wang2018constrained}   & $2$  & $1$  & Yes\\
		$12$  & G220       & \cite{AuDeLe07}    & $20$  & $2$ & Yes  & $33$ & SNAKE        & \cite{AuDe09a}     & $2$  & $2$  & No\\
		$13$  & GOMEZ       & \cite{wang2018constrained}    & $2$  & $1$ & Yes  & $34$ & SPEED-REDUCER        & \cite{MezCoe05}     & $7$  & $11$  & Yes\\
		$14$  & HS15         & \cite{HoSc1981}    & $2$  & $2$ & Yes  & $35$ & SPRING       & \cite{RodRenWat98} & $3$  & $4$  & Yes	\\
		$15$  & HS19         & \cite{HoSc1981}    & $2$  & $2$ & Yes  & $36$ & TAOWANG\_F1  & \cite{TaoWan08} & $2$  & $2$  & Yes\\
		$16$  & HS22         & \cite{HoSc1981}    & $2$  & $2$ & No  & $37$ & TAOWANG\_F2  & \cite{TaoWan08} & $7$  & $4$  & Yes\\
		$17$  & HS23         & \cite{HoSc1981}    & $2$  & $5$ & Yes  & $38$ & WELDED-BEAM  & \cite{MezCoe05} & $4$  & $7$  & Yes	\\
		$18$  & HS29         & \cite{HoSc1981}    & $3$  & $1$ & No  & $39$ & WONG2        & \cite{LuVl00}      & $10$ & $3$  & No\\
		$19$  & HS43         & \cite{HoSc1981}    & $4$  & $3$ & No  & $40$ & ZHAOWANG\_F5 & \cite{ChWa2010b}   & $13$ & $9$  & Yes\\
		$20$  & HS108         & \cite{HoSc1981}    & $9$  & $13$ & Yes  & $41$ & ZILONG\_G4 & \cite{wang2018constrained}   & $5$ & $1$  & Yes\\
		$21$  & HS114         & \cite{HoSc1981}    & $10$  & $5$ & Yes  & $42$ & ZILONG\_G24 & \cite{wang2018constrained}   & $2$ & $1$  & Yes
		\\ \hline
	\end{tabular}
\end{table}

%\begin{table}[ht!]
%	\label{tabtau1}
%	\centering
%	\begin{tabular}{@{}lccc@{}}
%		\toprule
%		Algorithm          & $\sigma=0.01$ & $\sigma=0.03$ & $\sigma=0.05$ \\ \midrule
%		StoMADS-PB $n^k=1$ & $74.6\%$      & $78.57\%$     & $73.02\%$     \\
%		StoMADS-PB $n^k=2$ & $74.6\%$      & $76.98\%$     & $76.19\%$     \\
%		StoMADS-PB $n^k=3$ & $76.19\%$     & 65.08\%       & $66.67\%$     \\
%		MADS-PB            & $69.5\%$      & $64.29\%$     & $54.76\%$     \\ \bottomrule
%	\end{tabular}
%\end{table}
\begin{table}[ht!]
		\centering
	\caption{Percentage of problems solved for each noise level $\sigma$ within a convergence tolerance $\tau$.}
	\label{pbtab}
%	\centering
	\begin{tabular}{@{}lccclccc@{}}
		\toprule
		& \multicolumn{3}{c}{$\tau=10^{-1}$}                        &  & \multicolumn{3}{c}{$\tau=10^{-3}$}                                            \\ \midrule
		Algorithm          & $\ \sigma=0.01 $ & $ \sigma=0.03 $ & $ \sigma=0.05\ \ \ $ &  & $\ \sigma=0.01 $ & $ \sigma=0.03 $ & \multicolumn{1}{l}{$ \sigma=0.05\ $} \\ \midrule
		StoMADS-PB $n^k=1$ & $74.6\%$          & $78.57\%$         & $73.02\%$         &  & $44.44\%$         & $45.24\%$         & $45.24\%$                             \\
		StoMADS-PB $n^k=2$ & $74.6\%$          & $76.98\%$         & $76.19\%$         &  & $47.62\%$         & $47.62\%$         & $50.79\%$                             \\
		StoMADS-PB $n^k=3$ & $\ 76.19\%$         & 65.08\%           & $66.67\%$         &  & $48.41\%$         & $41.27\%$         & $38.10\%$                             \\
		MADS-PB            & $69.5\%$          & $64.29\%$         & $54.76\%$         &  & $41.27\%$         & $36.51\%$         & $29.37\%$                             \\ \bottomrule
	\end{tabular}
\end{table}

Three variants of StoMADS-PB corresponding to $n^k=1, n^k=2$ and $n^k=3$ are compared to MADS-PB. The data and performance profiles used for the comparisons are depicted on Figures~\ref{dataprofPB1pc},~\ref{dataprofPB3pc} and~\ref{dataprofPB5pc} and Figures~\ref{perfprofPB1pc},~\ref{perfprofPB3pc} and~\ref{perfprofPB5pc}. Three levels of noise are used during the experiments, which correspond to $\sigma=0.01$, $\sigma=0.03$ and $\sigma=0.05$. For a given algorithm, the estimated percentages of problems solved after $1000(n+1)$ noisy blackbox evaluations for each noise level within a convergence tolerance $\tau$ are reported in Table~\ref{pbtab}. They are obtained based on the profiles graphs using {\sf MATLAB} tools. 

The data and performance profiles show that when given the time, StoMADS-PB {\bl eventually} outperforms MADS-PB in general. 
%%However, the profiles corresponding to $\sigma=0.01$ suggest that the latter algorithm could have a similar behavior as StoMADS-PB for $n^k=1$ when the noise level is very low.  
Moreover as in~\cite{audet2019stomads}, varying the value of the convergence tolerance~$\tau$ in the data profiles does not significantly alter the conclusions drawn from the performance profiles. Indeed as expected, it can be easily observed from Table~\ref{pbtab} that the higher the tolerance parameter $\tau$, the larger the percentage of problems solved by all algorithms for a fixed noise level $\sigma$. Now notice that while for a given $\tau$, the fraction of problems solved by MADS-PB decreases when the noise level increases from $\sigma=0.01$ to $\sigma=0.05$, this seems not to be the case for StoMADS-PB variants. Before giving an insight as to why, recall that in the present constrained framework, the success or failure of the convergence test~\eqref{convtest} does not depend only on the values of the objective function $f$ but also on whether a feasible point is found or not, unlike the framework of~\cite{audet2019stomads} where no constraints are involved. In fact, as highlighted in~\cite{audet2019stomads} from which is inspired the computation scheme~\eqref{scheme}, even though the robustness and efficiency of each StoMADS-PB variants depends on the number $n^k$ of noisy blackbox evaluations which is constant for all $k$, the quality of the solutions is influenced by the sample size $p^k=n^k+a^k(x^k)$ which is not constant. On one hand, this is the reason why for $n^k=1$, StoMADS-PB does not have the same behavior as MADS-PB. On the other hand, such computation scheme naturally favors StoMADS-PB by improving the accuracy of the estimates of its constraints function values, thus allowing it to find more feasible solutions than MADS-PB and consequently possibly solve larger fraction of problems when the noise level increases for a fixed tolerance parameter $\tau$.

Finally, based on Table~\ref{pbtab}, it can be noticed that for a given convergence tolerance $\tau$, varying $\sigma$ seems not to have significant influences on the fractions of problems solved by StoMADS-PB variants corresponding to $n^k=1$ and $n^k=2$. Moreover, even though for the lowest noise level studied $\sigma=0.01$, StoMADS-PB with $n^k=3$ solved the most problems, the corresponding percentage is not significantly larger than that of  StoMADS-PB with $n^k=2$. For all these reasons, the latter variant seems preferable for constrained stochastic blackbox optimization problems.

\clearpage
\twocolumn 
\begin{figure}
	\centering
	\includegraphics[scale=0.6]{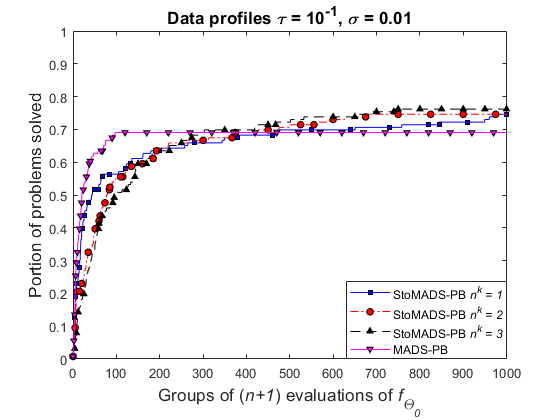}
	\includegraphics[scale=0.6]{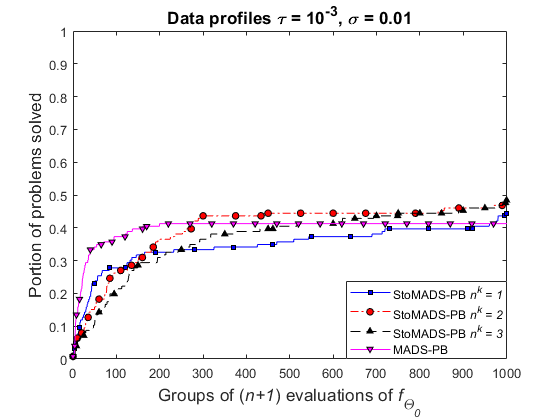}
	\caption{\small{Data profiles for convergence tolerances $\tau=10^{-1}$ and $\tau=10^{-3}$, and noise level $\sigma=0.01$ on $126$ analytical constrained test problems additively perturbed in the intervals $I(\sigma,x^0,f)$ and $I(\sigma,x^0,c_j)$.}}
	\label{dataprofPB1pc}
\end{figure}
\begin{figure}
	\centering
	\includegraphics[scale=0.6]{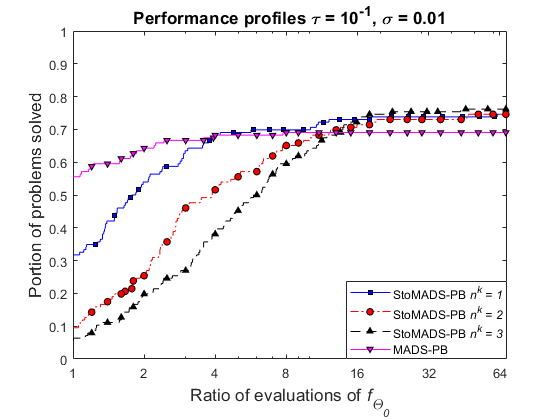} 
	\includegraphics[scale=0.6]{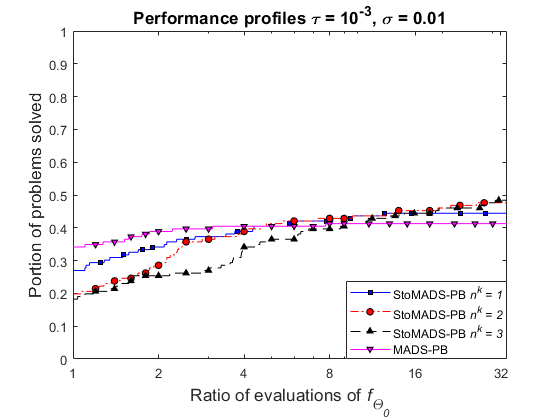}
	\caption{\small{Performance profiles for convergence tolerances $\tau=10^{-1}$ and $\tau=10^{-3}$, and noise level $\sigma=0.01$ on $126$ analytical constrained test problems additively perturbed in the intervals $I(\sigma,x^0,f)$ and $I(\sigma,x^0,c_j)$.}}
	\label{perfprofPB1pc}
\end{figure}

\clearpage

\begin{figure}
	\centering
	\includegraphics[scale=0.6]{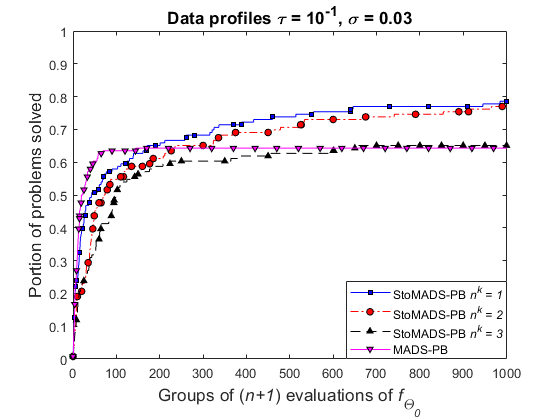}
	\includegraphics[scale=0.6]{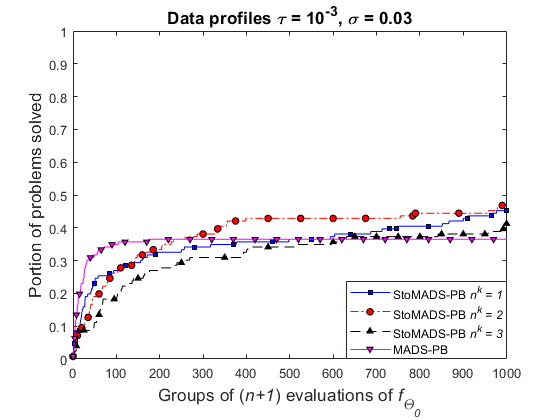}
	\caption[]{\small{Data profiles for convergence tolerances $\tau=10^{-1}$ and $\tau=10^{-3}$, and noise level $\sigma=0.03$ on $126$ analytical constrained test problems additively perturbed in the intervals $I(\sigma,x^0,f)$ and $I(\sigma,x^0,c_j)$.}}
	\label{dataprofPB3pc}
\end{figure}

\begin{figure}
	\centering
	\includegraphics[scale=0.6]{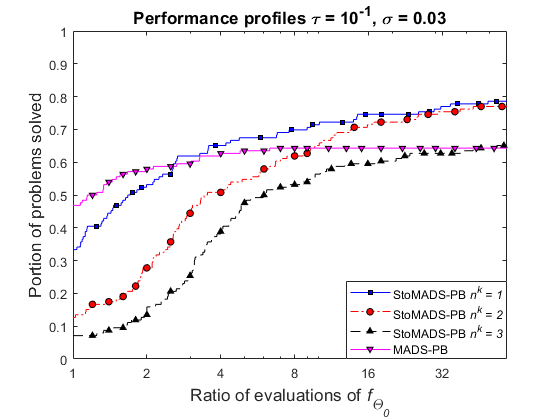}
	\includegraphics[scale=0.6]{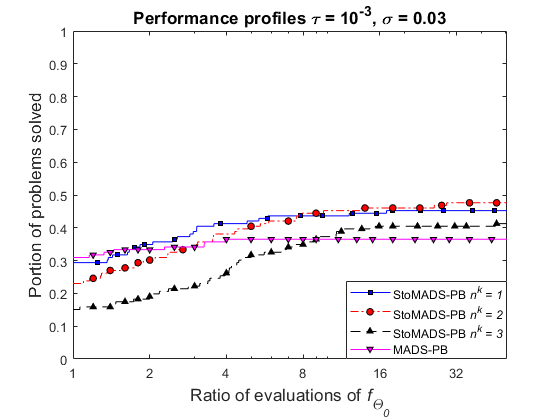}
	\caption[]{\small{Performance profiles for convergence tolerances $\tau=10^{-1}$ and $\tau=10^{-3}$, and noise level $\sigma=0.03$ on $126$ analytical constrained test problems additively perturbed in the intervals $I(\sigma,x^0,f)$ and $I(\sigma,x^0,c_j)$.}}
	\label{perfprofPB3pc} 
\end{figure}

\clearpage

\begin{figure}
	\centering
	\includegraphics[scale=0.6]{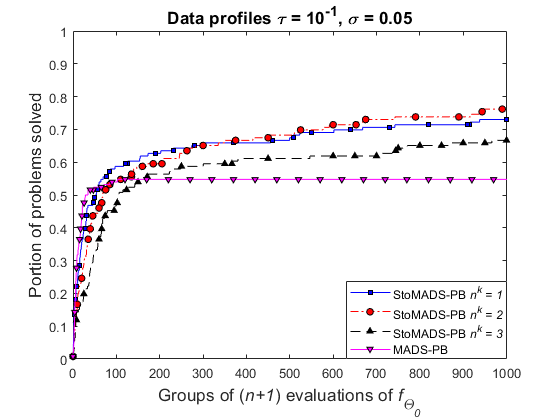}
	\includegraphics[scale=0.6]{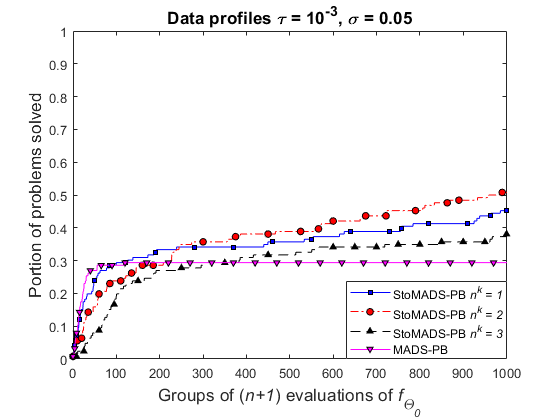}
	\caption[]{\small{Data profiles for convergence tolerances $\tau=10^{-1}$ and $\tau=10^{-3}$, and noise level $\sigma=0.05$ on $126$ analytical constrained test problems additively perturbed in the intervals $I(\sigma,x^0,f)$ and $I(\sigma,x^0,c_j)$.}}
	\label{dataprofPB5pc} 
\end{figure}

\begin{figure}
	\centering
	\includegraphics[scale=0.6]{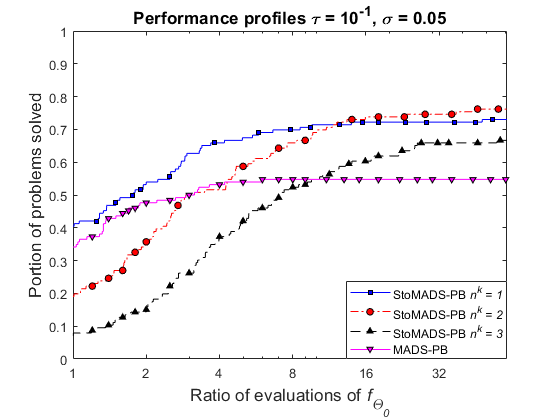}
	\includegraphics[scale=0.6]{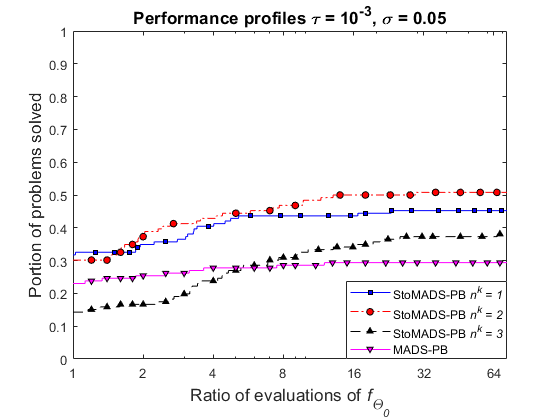}
	\caption[]{\small{Performance profiles for convergence tolerances $\tau=10^{-1}$ and $\tau=10^{-3}$, and noise level $\sigma=0.05$ on $126$ analytical constrained test problems additively perturbed in the intervals $I(\sigma,x^0,f)$ and $I(\sigma,x^0,c_j)$.}}
	\label{perfprofPB5pc} 
\end{figure}

\clearpage
\onecolumn

%------------------------------------------------%
\section*{Concluding remarks}
%------------------------------------------------%
This research proposes the StoMADS-PB algorithm for constrained stochastic blackbox optimization. The proposed method which uses an algorithmic framework similar to that of MADS considers the optimization of objective and constraints functions whose values can only be accessed through a stochastically noisy blackbox. It treats constraints using a progressive barrier approach, by aggregating their violations into a single function. It does not use any model or gradient information to find descent directions or improve feasibility unlike prior works, but instead, uses function estimates and introduces probabilistic bounds on which sufficient decrease conditions are imposed. By requiring the accuracy of such estimates and bounds to hold with sufficiently high but fixed probabilities, convergence results of StoMADS-PB are derived, most of which are stochastic variants of those of MADS. 

Computational experiments conducted on several variants of StoMADS-PB on a collection of constrained stochastically noisy problems showed the proposed method to eventually outperform MADS, and also showed some of its variants to be almost robust to random noise despite the use of very inaccurate estimates.

This research is to the best of our knowledge the first to propose a stochastic directional direct-search algorithm for BBO, developed to cope with a noisy objective and constraints that are also stochastically noisy. 

%Given that the proposed method can often converge to local minima since having presented any SEARCH step, 
%future research could focus on developing sophisticated SEARCH strategies, making use of existing so-called probabilistically accurate random models for example.
future research could focus on improving the proposed method to handle large-scale machine learning problems, making use for example of parallel space decomposition.
%------------------------------------------------%
\section*{Acknowledgments}
%------------------------------------------------%
The authors are grateful to Charles Audet from Polytechnique Montr\'eal for valuable discussions and constructive suggestions. This work is supported by the NSERC CRD RDCPJ 490744-15 grant and by an Innov\'E\'E grant, both in collaboration with {\bl Hydro-Qu\'ebec and Rio Tinto, and by a FRQNT fellowship.}

%\clearpage
%\appendix
\section*{Appendix}
Now we prove a sequence of convergence results of Section~\ref{convAnalysisPB}.
\subsubsection*{Proof of Theorem~\ref{zerothorderTheorem}}
\begin{proof}
	This theorem is proved using ideas from~\cite{audet2019stomads,blanchet2019convergence,chen2018stochastic,dzahini2020expected,LaBi2016,paquette2018stochastic}. 
	%Recall the sets $\mathscr{S}^k_X=\{X^{\ell}_{\text{feas}}:\ell\leq k\}$ of random $\efPB$-feasible iterates generated by Algorithm~\ref{algoPB} by the start of iteration~$k$. Consider the following random time $T$ defined by 
	%\begin{equation}\label{randomTimeT}
	%T:=\inf\{k\geq 0:\mathscr{S}^k_X\neq\emptyset\}.
	%\end{equation}
	According to Assumptions~\ref{stoppingAssumption}, the proof considers two different parts: Part~1 assumes that $T=+\infty$ almost surely, i.e., no $\efPB$-feasible iterate is found by Algorithm~\ref{algoPB}, while Part~2 considers that $T<+\infty$ almost surely.
	%which means that the sets $\mathscr{S}^k_X$ of $\efPB$-feasible iterates are nonempty for all $k\geq T$. 
	Part~1 considers two separate cases: ``good bounds'' and ``bad bounds'', each of which is broken into whether an iteration is $h$-Dominating, Improving or Unsuccessful. Part~2 considers three separates cases: ``good estimates and good bounds'', ``bad estimates and good bounds'' and ``bad bounds'', each of which is broken into whether an iteration is $f$-Dominating, $h$-Dominating, Improving or Unsuccessful.
	
	In order to show~\eqref{deltaSeries}, the goal of Part~1 is to show that there exists a constant $\eta>0$ such that conditioned on the almost sure event $\{T=+\infty\}$, the following holds for all $k\in\N$
	\begin{equation}\label{phiPart1}
	\E{\Phi_{k+1}-\Phi_k|\fcf}\leq -\eta(\Dp)^2, %\quad\text{for all}\ k\in\N,
	\end{equation}
	where $\Phi_k$ is the random function defined by
	\begin{equation}\label{phik}
	\Phi_k:=\frac{\nu}{m\efPB}h(\Xki)+(1-\nu)(\Dp)^2, \quad\text{for all}\ k\in\N.
	\end{equation}
	Indeed, assume that~\eqref{phiPart1} holds. Since $\Phi_k>0$ for all $k\in\N$, then summing~\eqref{phiPart1} over $k\in\N$ and taking expectations on both sides lead to
	\begin{equation}\label{deltaPart1}
	\Esp\left[\sum_{k=0}^{+\infty}(\Dp)^2 \right]\leq\frac{\E{\Phi_0}}{\eta} =\frac{\Phi_0}{\eta},  %\E{\Phi_0|\fcfmUn}\itplusInf.
	\end{equation}
	That is, \eqref{deltaSeries} holds. Then, making use of the following random function 
	\begin{equation}\label{phikT}
	\Phi_k^T:=\frac{\nu}{\efPB}(f(\Xkft)-\kappa^f_{\min})+\frac{\nu}{m\efPB}h(\Xki)+(1-\nu)(\Dp)^2, \quad\text{for all}\ k\in\N,
	\end{equation}
	where $k\vee T:=\max\{k,T\}$, Part~2 aims to show that for the same previous constant $\eta>0$, then conditioned on the almost sure event $\{T<+\infty\}$, the following holds for all $k\in\N$
	\begin{equation}\label{phiPart2}
	\E{\Phi_{k+1}^T-\Phi_k^T|\fcf}\leq -\eta(\Dp)^2.%\quad\text{for all}\ k\geq T.
	\end{equation}
	Indeed, assume that~\eqref{phiPart2} holds. Since $\Phi_k^T>0$ for all $k\geq 0$, then summing~\eqref{phiPart2} over $k\in\N$ and taking expectations on both sides, yield 
	\begin{equation}\label{deltaPartT}
	%\itFini\sum_{k=T}^{+\infty}(\Dp)^2\leq\frac{1}{\eta}\E{\Phi_T^*|\fcfT}\itFini.
	\begin{split}
	\Esp\left[\sum_{k=0}^{+\infty}(\Dp)^2 \right]&\leq\frac{\E{\Phi_0^T}}{\eta}=\frac{1}{\eta}\left[\frac{\nu}{\efPB}\left(\Esp\left[f(X^T_{\textnormal{feas}})\right]-\kappa^f_{\min}\right)+\frac{\nu}{m\efPB}h(\xzeroi)+(1-\nu)(\delta^0_p)^2 \right] \\
	&\leq \frac{1}{\eta}\left[\frac{\nu}{\efPB}\left(\kappa^f_{\max}-\kappa^f_{\min}\right)+\frac{\nu}{m\efPB}h(\xzeroi)+(1-\nu)(\delta^0_p)^2 \right]=:\mu,
	\end{split}
	\end{equation}
	where the last inequality in~\eqref{deltaPartT} follows from the inequality $f(\Xkf)\leq \kappa^f_{\max}$ for all $k\geq 0$, due to Proposition~\ref{kappafmax}, and the fact that $T$ is finite almost surely.

	The remainder of the proof is devoted to showing that~\eqref{phiPart1} and \eqref{phiPart2} hold. The following events are introduced for the sake of clarity in the analysis. \\
	$\mathcal{D}_f:=\ $\{The iteration is $f$-Dominating\},\ \ \quad $\mathcal{D}_h:=\ $\{The iteration is $h$-Dominating\},\\
	$\mathcal{I}\ \ \, :=\{\text{The iteration is Improving}\}$,\quad\quad\quad $\ \, \mathcal{U}:=\{\text{The iteration is Unsuccessful}\}$.\\
	$ $\\
	\textbf{Part~1 ($\boldsymbol{T=+\infty}$ almost surely)}. 
	The random function $\Phi_k$ defined in~\eqref{phik} will be shown to satisfy~\eqref{phiPart1} with $\eta=\frac{1}{2}\alpha\beta(1-\nu)(1-\tau^2)$, no matter the change led in the objective function $f$ by the $\efPB$-infeasible iterates encountered by Algorithm~\ref{algoPB}. Moreover, since $T$ is infinite almost surely, then no iteration of Algorithm~\ref{algoPB} can be  $f$-Dominating. Two separate cases are distinguished and all that follows is conditioned on the almost sure event $\{T=+\infty\}$.\\
	\textbf{Case~1 (Good bounds, $\boldsymbol{\iik=1}$)}. No matter the type of iteration which occurs, the random function $\Phi_k$ is shown to decrease and the smallest decrease is shown to happen on unsuccessful iterations, thus yielding the following conclusion
	\begin{equation}\label{onze}
	\Esp\left[\iik(\Phi_{k+1}-\Phi_k)|\fcf\right] \leq -\alpha(1-\nu)(1-\tau^2)(\Dp)^2.
	\end{equation}
	\begin{itemize}
		\item[(i)] { The iteration is} $h$-Dominating ($\idh=1$). The iteration is $h$-Dominating and the bounds are good, so a decrease occurs in $h$ according to~\eqref{five} as follows 
		\begin{equation}\label{un}
		\iik\idh\frac{\nu}{m\efPB} (h(\Xkiun)-h(\Xki))\leq -\iik\idh\nu(\gamma-2)(\Dp)^2
		\end{equation}
		The frame size parameter is updated according to $\Delta^{k+1}_p=\min\{\tau^{-1}\Dp,\delta_{\max}\}$, which implies that
		\begin{equation}\label{unprime}
		\iik\idh(1-\nu)[(\Delta^{k+1}_p)^2-(\Dp)^2]\leq \iik\idh(1-\nu)(\tau^{-2}-1)(\Dp)^2.
		\end{equation}
		Then, by choosing $\nu$ according to~\eqref{nuChoice}, the right-hand side term of~\eqref{un} dominates that of~\eqref{unprime}. Specifically, the following holds
		\begin{equation}\label{deux}
		-\nu(\gamma-2)(\Dp)^2 + (1-\nu)(\tau^{-2}-1)(\Dp)^2 \leq -\frac{1}{2}\nu(\gamma-2)(\Dp)^2.
		\end{equation}
		Then combining~\eqref{un}, \eqref{unprime} and~\eqref{deux} leads to 
		\begin{equation}\label{quatre}
		\iik\idh(\Phi_{k+1}-\Phi_k)\leq-\iik\idh\frac{1}{2}\nu(\gamma-2)(\Dp)^2.
		\end{equation}
		\item[(ii)] { The iteration is} Improving ($\iimp=1$). The iteration is Improving and the bounds are good, so again, a decrease occurs in $h$ according to~\eqref{five}. Moreover, $\Dp$ is updated as at $h$-Dominating iterations. Thus, the change in $\Phi_k$ follows from~\eqref{quatre} by replacing $\idh$ by $\iimp$. Specifically, 
		\begin{equation}\label{cinq}
		\iik\iimp(\Phi_{k+1}-\Phi_k)\leq-\iik\iimp\frac{1}{2}\nu(\gamma-2)(\Dp)^2.
		\end{equation}
		\item[(iii)] { The iteration is} Unsuccessful ($\iuns=1$). There is a change of zero in $h$ function values while the frame size parameter is decreased. Consequently,
		\begin{equation}\label{six}
		\iik\iuns(\Phi_{k+1}-\Phi_k)=-\iik\iuns(1-\nu)(1-\tau^{2})(\Dp)^2
		\end{equation}
		Then, the choice of $\nu$ according to~\eqref{nuChoice} and the fact that $1-\tau^2<\tau^{-2}-1$ ensures that unsuccessful iterations, more precisely~\eqref{six}, provide the worst case decrease when compared to~\eqref{quatre} and~\eqref{cinq}. Specifically, the following holds
		\begin{equation}\label{sept}
		-\frac{1}{2}\nu(\gamma-2)(\Dp)^2\leq-(1-\nu)(1-\tau^{2})(\Dp)^2.
		\end{equation}
		Thus, it follows from~\eqref{quatre},~\eqref{cinq},~\eqref{six} and~\eqref{sept} that the change in $\Phi_k$ is bounded as follows
		\begin{equation}\label{dix}
		\iik(\Phi_{k+1}-\Phi_k)=\iik(\idh+\iimp+\iuns)(\Phi_{k+1}-\Phi_k)\leq -\iik(1-\nu)(1-\tau^{2})(\Dp)^2.
		\end{equation}
	\end{itemize}
	Since Assumption~\ref{randomEstim} holds, then taking conditional expectations with respect to $\fcf$ on both sides of the inequality in~\eqref{dix} leads to~\eqref{onze}. \\
	\textbf{Case~2 (Bad bounds, $\boldsymbol{\iikc=1}$)}. Since the bounds are bad, Algorithm~\ref{algoPB} can accept an iterate which leads to an increase in $h$ and $\Dp$, and hence in $\Phi_k$. Such an increase in $\Phi_k$ is controlled making use of~\eqref{Hcond1}. Then, the probability of outcome (Part~1, Case~2) is adjusted to be sufficiently small so that $\Phi_k$ can be reduced sufficiently in expectation. More precisely, the following will be proved
	\begin{equation}\label{seize}
	\Esp\left[\iikc(\Phi_{k+1}-\Phi_k)|\fcf\right] \leq 2\nu(1-\alpha)^{1/2}(\Dp)^2.
	\end{equation} 
	\begin{itemize}
		\item[(i)] { The iteration is} $h$-Dominating ($\idh=1$). The change in $h$ is bounded as follows
		\begin{eqnarray}
		\iikc\idh\frac{\nu}{m\efPB} (& &\hspace*{-1cm}h(\Xkiun)-h(\Xki)) \nonumber \\
		&\leq&\iikc\idh\frac{\nu}{m\efPB} \left[(\Hsk-\Hok)+\abs{h(\Xkiun)-\Hsk}+\abs{h(\Xki)-\Hok}\right]  \nonumber\\
		&\leq&\iikc\idh\nu\left[-\gamma(\Dp)^2+\frac{1}{m\efPB}\left(\abs{h(\Xkiun)-\Hsk}+\abs{h(\Xki)-\Hok}\right)\right],\label{douze}
		\end{eqnarray}
		where~\eqref{douze} follows from $\Hsk-\Hok\leq-\gamma m\efPB(\Dp)^2 $ which is satisfied for every $h$-Dominating iteration. Moreover, the change in $\Dp$ can be obtained simply by replacing in~\eqref{unprime} $\iik$ by $\iikc$ as follows
		\begin{equation}\label{unsecond}
		\iikc\idh(1-\nu)[(\Delta^{k+1}_p)^2-(\Dp)^2]\leq \iikc\idh(1-\nu)(\tau^{-2}-1)(\Dp)^2.
		\end{equation}
		Since choosing $\nu$ according to~\eqref{nuChoice} ensures that $-\nu\gamma(\Dp)^2+(1-\nu)(\tau^{-2}-1)(\Dp)^2\leq 0$,
		%\begin{equation}\label{douzeprime}
		%-\nu\gamma(\Dp)^2+(1-\nu)(\tau^{-2}-1)(\Dp)^2\leq 0.
		%\end{equation}
		then combining~\eqref{douze} and \eqref{unsecond}, yields
		\begin{equation}\label{treize}
		\iikc\idh(\Phi_{k+1}-\Phi_k)\leq\iikc\idh\frac{\nu}{m\efPB}\left(\abs{h(\Xkiun)-\Hsk}+\abs{h(\Xki)-\Hok}\right).
		\end{equation}
		\item[(ii)] { The iteration is} Improving ($\iimp=1$). $\Dp$ is updated as at $h$-Dominating iterations and because of bad bounds, the increase in $h$ is bounded following~\eqref{douze}. Thus, the bound on the change in $\Phi_k$ can be obtained by replacing $\idh$ by $\iimp$ in~\eqref{treize} as follows
		\begin{equation}\label{quatorze}
		\iikc\iimp(\Phi_{k+1}-\Phi_k)\leq\iikc\iimp\frac{\nu}{m\efPB}\left(\abs{h(\Xkiun)-\Hsk}+\abs{h(\Xki)-\Hok}\right).
		\end{equation}
		\item[(iii)] { The iteration is} Unsuccessful ($\iuns=1$). The change in $h$ is zero and $\Dp$ is decreased. Thus, the change in $\Phi_k$ follows from~\eqref{six} by replacing $\iik$ by $\iikc$ and is trivially bounded as follows
		\begin{equation}\label{sixprime}
		\iikc\iuns(\Phi_{k+1}-\Phi_k)
		\leq\iikc\iuns\frac{\nu}{m\efPB}\left(\abs{h(\Xkiun)-\Hsk}+\abs{h(\Xki)-\Hok}\right).
		\end{equation}
		Finally, it follows from~\eqref{treize}, \eqref{quatorze}, \eqref{sixprime} and the inequality $\iikc\leq 1$,  that
		\begin{equation}
		\iikc(\Phi_{k+1}-\Phi_k)
		\leq \frac{\nu}{m\efPB}\left(\abs{h(\Xkiun)-\Hsk}+\abs{h(\Xki)-\Hok}\right),\label{quinze}
		\end{equation}
		Then, taking conditional expectations with respect to $\fcf$ on both sides of~\eqref{quinze} and using the inequalities~\eqref{Hcond1} of Assumption~\ref{randomEstim}, lead to~\eqref{seize}.
	\end{itemize}
	Now, combining~\eqref{onze} and~\eqref{seize} yields,
	\begin{eqnarray}
	\E{\Phi_{k+1}-\Phi_k|\fcf}&=&\Esp\left[(\iik+\iikc)(\Phi_{k+1}-\Phi_k)|\fcf\right]\nonumber \\ 
	&\leq& \left[-\alpha(1-\nu)(1-\tau^2)+ 2\nu(1-\alpha)^{1/2}\right](\Dp)^2.\label{seizeprime}
	\end{eqnarray}
	Then, choosing $\alpha$ according to~\eqref{betaChoice} implies that $\ds{\alpha\geq\frac{4\nu (1-\alpha)^{1/2}}{(1-\nu)(1-\tau^2)}}$, which ensures 
	%\begin{equation}
	%\alpha\geq\frac{4\nu (1-\alpha)^{1/2}}{(1-\nu)(1-\tau^2)}, \nonumber
	%\end{equation}
	%which ensures that 
	\begin{equation}\label{seizesecond}
	-\alpha(1-\nu)(1-\tau^2)+ 2\nu(1-\alpha)^{1/2}\leq-\frac{1}{2}\alpha(1-\nu)(1-\tau^2)\leq -\frac{1}{2}\alpha\beta(1-\nu)(1-\tau^2).
	\end{equation}
	Thus, \eqref{phiPart1} follows from~\eqref{seizeprime} and~\eqref{seizesecond} with $\eta=\frac{1}{2}\alpha\beta(1-\nu)(1-\tau^2)$.\\
	$ $\\
	\textbf{Part~2 ($\boldsymbol{T<+\infty}$ almost surely)}. In order to show that the random function $\Phi_k^T$ defined by
	\begin{equation}
	\Phi_k^T=\frac{\nu}{\efPB}(f(\Xkft)-\kappa^f_{\min})+\frac{\nu}{m\efPB}h(\Xki)+(1-\nu)(\Dp)^2 \nonumber
	\end{equation} 
	satisfies~\eqref{phiPart2} with the same constant $\eta$ derived in Part~1, notice that whenever the event $\{T>k\}$ occurs, then $f(\Xkfunt)-f(\Xkft)=0$ since $\max\{k,T\}:=k\vee T=(k+1)\vee T=T$. Thus, on the event $\{T>k\}$, the random function $\Phi_k$ used in Part~1 has the same increments as~$\Phi_k^T$. Specifically, 
	\begin{equation}\label{equalIncrements}
	\itFini\itksup(\Phi_{k+1}^T-\Phi_k^T)=\itFini\itksup(\Phi_{k+1}-\Phi_k).\nonumber
	\end{equation}
	Moreover, it follows from the definition of the stopping time $T$ that no iteration can be $f$-Dominating as in Part~1 when the event $\{T>k\}$ occurs. Consequently, it easily follows from the analysis in Part~1 and the fact that the random variable $\itksup$ is $\fcf$-measurable that,
	\begin{equation}\label{phiPart1Tsupk}
	\itksup\E{\Phi_{k+1}^T-\Phi_k^T|\fcf}\leq -\eta(\Dp)^2\itksup.
	\end{equation}
	The remainder of the proof is devoted to showing that the following holds
	\begin{equation}\label{phiPart1Tlessk}
	\itkless\E{\Phi_{k+1}^T-\Phi_k^T|\fcf}\leq -\eta(\Dp)^2\itkless,
	\end{equation}
	since combining~\eqref{phiPart1Tsupk} and~\eqref{phiPart1Tlessk} leads to~\eqref{phiPart2}, which is the remaining overall goal. In all that follows, it is assumed that the event $\{T\leq k\}$ occurs.\\
	\textbf{Case~1 (Good estimates and good bounds, $\boldsymbol{\iik\ijk=1}$)}. Regardless of the iteration type, the smallest decrease in $\Phi_k^T$ is shown to happen on unsuccessful iterations, thus implying that
	\begin{equation}\label{vingtsix}
	\itkless\Esp\left[\iik\ijk(\Phi_{k+1}^T-\Phi_k^T)|\fcf\right] \leq -\alpha\beta(1-\nu)(1-\tau^2)(\Dp)^2\itkless.
	\end{equation}
	\begin{itemize}
		\item[(i)] { The iteration is} $f$-Dominating ($\idf=1$). The iteration is $f$-Dominating and the estimates are good, so a decrease occurs in $f$ according to~\eqref{successf} as follows 
		\begin{eqnarray}
		\itkless\iik\ijk\idf\frac{\nu}{\efPB}(f(& &\hspace*{-1cm}\Xkfunt)-f(\Xkft))\nonumber\\
		&\leq&-\itkless\iik\ijk\idf\nu(\gamma-2)(\Dp)^2.\label{vingtc}
		\end{eqnarray}
		Since the $\efPB$-infeasible iterate is not updated, then there is a change of zero in $h$. The frame size parameter is updated according to $\Delta^{k+1}_p=\min\{\tau^{-1}\Dp,\delta_{\max}\}$, thus implying that
		\begin{equation}\label{vingtb}
		\itkless\iik\ijk\idf(1-\nu)[(\Delta^{k+1}_p)^2-(\Dp)^2]\leq \itkless\iik\ijk\idf(1-\nu)(\tau^{-2}-1)(\Dp)^2.
		\end{equation}
		Then, choosing $\nu$ according to~\eqref{nuChoice} ensures that~\eqref{deux} holds, which implies that the right-hand side term of~\eqref{vingtc} dominates that of~\eqref{vingtb}, thus leading to the inequality below
		\begin{equation}\label{vingtun}
		\itkless\iik\ijk\idf(\Phi_{k+1}^T-\Phi_k^T)\leq-\itkless\iik\ijk\idf\frac{1}{2} \nu(\gamma-2)(\Dp)^2.
		\end{equation}
		\item[(ii)] { The iteration is} $h$-Dominating ($\idh=1$). There is a change of zero in $f$ since $\Xkf$ is not updated. Thus, the bound on the change in $\Phi_k^T$ follows from multiplying both sides of~\eqref{quatre} by $\itkless\ijk$, and replacing $\Phi_k$ by $\Phi_k^T$ as follows
		\begin{equation}\label{vingtdeux}
		\itkless\iik\ijk\idh(\Phi_{k+1}^T-\Phi_k^T)\leq-\itkless\iik\ijk\idh\frac{1}{2}\nu(\gamma-2)(\Dp)^2.
		\end{equation}
		\item[(iii)] { The iteration is} Improving ($\iimp=1$). Again, there is a change of zero in $f$. Thus, the bound on the change in $\Phi_k^T$ easily follows from multiplying both sides of~\eqref{cinq} by $\itkless\ijk$, and replacing $\Phi_k$ by $\Phi_k^T$ as follows
		\begin{equation}\label{vingttrois}
		\itkless\iik\ijk\iimp(\Phi_{k+1}^T-\Phi_k^T)\leq-\itkless\iik\ijk\iimp\frac{1}{2}\nu(\gamma-2)(\Dp)^2. 
		\end{equation}
		\item[(iv)] { The iteration is} Unsuccessful ($\iuns=1$). There is a change of zero in $f$ and in $h$ since no iterate is updated, while $\Dp$ is decreased. Consequently, the bound on the change in $\Phi_k^T$ follows from multiplying both sides of~\eqref{six} by $\itkless\ijk$, and replacing $\Phi_k$ by $\Phi_k^T$ as follows
		\begin{equation}\label{vingtquatre}
		\itkless\iik\ijk\iuns(\Phi_{k+1}^T-\Phi_k^T)=-\itkless\iik\ijk\iuns(1-\nu)(1-\tau^{2})(\Dp)^2.
		\end{equation} 
		Then combining~\eqref{vingtun}, \eqref{vingtdeux}, \eqref{vingttrois}, \eqref{vingtquatre} and using~\eqref{sept}, yields
		\begin{equation}\label{vingcinq}
		\itkless\iik\ijk(\Phi_{k+1}^T-\Phi_k^T)
		\leq-\itkless\iik\ijk(1-\nu)(1-\tau^2)(\Dp)^2.
		\end{equation}
	\end{itemize}
	Now, notice that under Assumption~\ref{randomEstim}, simple calculations lead to $\E{\iik\ijk|\fcf}\geq\alpha\beta$. Then, taking expectations with respect to $\fcf$ on both sides of~\eqref{vingcinq} and using the $\fcf$-measurability of the random variables $\itkless$ and $\Dp$, lead to~\eqref{vingtsix}.\\
	\textbf{Case~2 (Bad estimates and good bounds, $\boldsymbol{\iik\ijkc=1}$)}.
	An increase in the difference of~$\Phi_k^T$ may occurs since good bounds might not provide enough decrease to cancel the increase which occurs in $f$ whenever Algorithm~\ref{algoPB} wrongly accepts an iterate because of bad estimates. Specifically, the $f$-Dominating case dominates the worst-case increase in the change of $\Phi_k^T$, thus leading to
	\begin{equation}\label{trenteun}
	\itkless\Esp\left[\iik\ijkc(\Phi_{k+1}^T-\Phi_k^T)|\fcf\right] \leq 2\nu(1-\beta)^{1/2}(\Dp)^2\itkless.
	\end{equation}
	\begin{itemize}
		\item[(i)] { The iteration is} $f$-Dominating ($\idf=1$). Whenever bad estimates occur and the iteration is $f$-Dominating, the change in $f$ is bounded as follows
		\begin{equation}\label{vingtsept}
		\begin{split}
		\mathds{1}_{\{T\leq k\}}& \iik\ijkc\idf\frac{\nu}{\efPB}(f(\Xkfunt)-f(\Xkft))\\
		&\leq\itkless\iik\ijkc\idf\frac{\nu}{\efPB} \left[(\Fsk-\Fok)+\abs{f(\Xkfun)-\Fsk}+\abs{f(\Xkf)-\Fok}\right]  \\
		&\leq\itkless\iik\ijkc\idf\nu\left[-\gamma(\Dp)^2+\frac{1}{\efPB}\left(\abs{f(\Xkfun)-\Fsk}+\abs{f(\Xkf)-\Fok}\right)\right]
		\end{split}
		\end{equation}
		where the last inequality in~\eqref{vingtsept} follows from  $\Fsk-\Fok\leq-\gamma\efPB(\Dp)^2 $ which is satisfied for every $f$-Dominating iteration. While the change in $h$ is zero since $\Xki$ is not updated, that in $\Dp$ follows~\eqref{vingtb} by replacing~$\ijk$ by $\ijkc$ as follows
		\begin{equation}\label{vingtseptPrime}
		\itkless\iik\ijkc\idf(1-\nu)[(\Delta^{k+1}_p)^2-(\Dp)^2]\leq \itkless\iik\ijkc\idf(1-\nu)(\tau^{-2}-1)(\Dp)^2.
		\end{equation}
		%Since choosing $\nu$ according to~\eqref{nuChoice} ensures that $-\nu\gamma(\Dp)^2+(1-\nu)(\tau^{-2}-1)(\Dp)^2\leq 0$,
		Then, \eqref{vingtsept}, \eqref{vingtseptPrime} and the inequality $-\nu\gamma(\Dp)^2+(1-\nu)(\tau^{-2}-1)(\Dp)^2\leq 0$ due to~\eqref{nuChoice} yield
		\begin{equation}\label{vingthuit}
		\begin{split}
		\itkless\iik\ijkc\idf(& \Phi_{k+1}^T-\Phi_k^T)\\
		&\leq\itkless\iik\ijkc\idf\frac{\nu}{\efPB}\left(\abs{f(\Xkfun)-\Fsk}+\abs{f(\Xkf)-\Fok}\right).
		\end{split}
		\end{equation}
		\item[(ii)] { The iteration is} $h$-Dominating ($\idh=1$). The bound on the change in $\Phi_k^T$ which can be obtained by replacing~$\ijk$ by $\ijkc$ in~\eqref{vingtdeux} is trivially bounded as follows
		\begin{eqnarray}\label{vingtneuf}
		\itkless\iik\ijkc\idh& &\hspace*{-1cm}(\Phi_{k+1}^T-\Phi_k^T)\nonumber\\
		&\leq&\itkless\iik\ijkc\idh\frac{\nu}{\efPB}\left(\abs{f(\Xkfun)-\Fsk}+\abs{f(\Xkf)-\Fok}\right).
		\end{eqnarray}
		\item[(iii)] { The iteration is} Improving ($\iimp=1$). Again, the change in $\Phi_k^T$ which can be obtained by replacing~$\ijk$ by $\ijkc$ in~\eqref{vingttrois} is trivially bounded as follows
		%\begin{equation}
		%\itkless\iik\ijkc\iimp(\Phi_{k+1}^T-\Phi_k^T)\leq-\itkless\iikc\ijk\iimp\frac{1}{2}\nu(\gamma-2)(\Dp)^2. 
		%\end{equation}
		\begin{eqnarray}\label{trente}
		\itkless\iik\ijkc\iimp& &\hspace*{-1cm}(\Phi_{k+1}^T-\Phi_k^T)\nonumber\\
		&\leq&\itkless\iik\ijkc\iimp\frac{\nu}{\efPB}\left(\abs{f(\Xkfun)-\Fsk}+\abs{f(\Xkf)-\Fok}\right).
		\end{eqnarray}
		\item[(iv)] { The iteration is} Unsuccessful ($\iuns=1$). Because of the decrease of the frame size parameter and hence that in $\Phi_k^T$, the bound on the change in $\Phi_k^T$ is obviously as follows
		\begin{equation}\label{trenteprime}
		\begin{split}
		\itkless\iik\ijkc\iuns(& \Phi_{k+1}^T-\Phi_k^T)\\
		&\leq\itkless\iik\ijkc\iuns\frac{\nu}{\efPB}\left(\abs{f(\Xkfun)-\Fsk}+\abs{f(\Xkf)-\Fok}\right).
		\end{split}
		\end{equation}
		Then, combining~\eqref{vingthuit}, \eqref{vingtneuf}, \eqref{trente} and $\iik\ijkc\leq 1$, yields
		\begin{eqnarray}
		\itkless\iik\ijkc(& &\hspace*{-1cm}\Phi_{k+1}^T-\Phi_k^T)\nonumber \\
		%&=&\itkless\iik\ijkc(\idf+\idh+\iimp+\iuns)(\Phi_{k+1}^T-\Phi_k^T)\nonumber\\
		%&\leq&\itkless\iik\ijkc\frac{\nu}{\efPB}\left(\abs{f(\Xkfun)-\Fsk}+\abs{f(\Xkf)-\Fok}\right)\nonumber \\
		&\leq&\itkless\frac{\nu}{\efPB}\left(\abs{f(\Xkfun)-\Fsk}+\abs{f(\Xkf)-\Fok}\right).\label{trentesecond}
		\end{eqnarray}
		Since Assumption~\ref{randomEstim} holds, it follows from the conditional Cauchy-Schwarz inequality~\cite{bhattacharya2007basic} that
		\begin{eqnarray}
		\E{\abs{f(\Xkf)-\Fok}|\fcf}&\leq&\E{1|\fcf}^{1/2} \left[\E{\abs{f(\Xkf)-\Fok}^2|\fcf}\right]^{1/2}\nonumber\\
		&\leq&\efPB(1-\beta)^{1/2}(\Dp)^2,\label{condFok}
		\end{eqnarray}
		where~\eqref{condFok} follows from~\eqref{varcond1PB} and the fact that $\E{1|\fcf}=1$. Similarly, the following holds
		\begin{equation}
		\E{\abs{f(\Xkfun)-\Fsk}|\fcf}\leq\efPB(1-\beta)^{1/2}(\Dp)^2.\label{condFsk}
		\end{equation}
	\end{itemize}
	Thus, taking expectations with respect to $\fcf$ on both sides of~\eqref{trentesecond} and then using~\eqref{condFok},~\eqref{condFsk} and the $\fcf$-measurability of the random variables $\itkless$ and $\Dp$, lead to~\eqref{trenteun}.\\
	\textbf{Case~3 (Bad bounds, $\boldsymbol{\iikc=1}$)}. The difference in~$\Phi_k^T$ may increase since even though good estimates of $f$ values occur, they might not provide enough decrease to cancel the increase in $h$ whenever Algorithm~\ref{algoPB} wrongly accepts an iterate because of bad bounds. The following will be shown
	%Specifically, in addition to the $h$-Dominating case, the worst-case increase in the change of $\Phi_k^T$ is dominated by the possible increase in $f$ which can be caused by the possible occurrence of bad estimates, thus leading to
	\begin{equation}\label{trentesix}
	\itkless\Esp\left[\iikc(\Phi_{k+1}^T-\Phi_k^T)|\fcf\right] \leq 2\nu\left[(1-\alpha)^{1/2}+(1-\beta)^{1/2}\right](\Dp)^2\itkless.
	\end{equation}
	\begin{itemize}
		\item[(i)] { The iteration is} $f$-Dominating ($\idf=1$). The change in $\Phi_k^T$ is bounded, taking into account the possible aforementioned increase in $f$. Since the change in $h$ is zero, then it is easy to notice that the bound on the change in $\Phi_k^T$ can be derived from~\eqref{vingthuit} by replacing $\iik\ijkc$ by $\iikc$ as follows
		\begin{equation}\label{trentedeux}
		\begin{split}
		\itkless\iikc\idf(& \Phi_{k+1}^T-\Phi_k^T)\\
		&\leq\itkless\iikc\idf\frac{\nu}{\efPB}\left(\abs{f(\Xkfun)-\Fsk}+\abs{f(\Xkf)-\Fok}\right).
		\end{split}
		\end{equation}
		\item[(ii)] { The iteration is} $h$-Dominating ($\idh=1$). Since the change in $f$ is zero, the bound on the change in $\Phi_k^T$ is obtained by multiplying both sides of~\eqref{treize} by $\itkless$ and replacing $\Phi_k$ by $\Phi_k^T$ 
		\begin{equation}\label{trentetrois}
		\itkless\iikc\idh(\Phi_{k+1}-\Phi_k)\leq\itkless\iikc\idh\frac{\nu}{m\efPB}\left(\abs{h(\Xkiun)-\Hsk}+\abs{h(\Xki)-\Hok}\right). 
		\end{equation}
		\item[(iii)] { The iteration is} Improving ($\iimp=1$). The frame size parameter is updated as at $h$-Dominating iterations and the change in $f$ is zero. Thus, the bound on the change in $\Phi_k^T$ follows from~\eqref{trentetrois} by replacing~$\idh$ by~$\iimp$ as follows
		\begin{equation}\label{trentequatre}
		\itkless\iikc\iimp(\Phi_{k+1}-\Phi_k)\leq\itkless\iikc\iimp\frac{\nu}{m\efPB}\left(\abs{h(\Xkiun)-\Hsk}+\abs{h(\Xki)-\Hok}\right). 
		\end{equation}
		\item[(iv)] { The iteration is} Unsuccessful ($\iuns=1$). Because of the decrease of the frame size parameter and hence that in $\Phi_k^T$, the bound on the change in $\Phi_k^T$ is obviously as follows
		\begin{equation}\label{trentecinq}
		\begin{split}
		\itkless\iikc\iuns(& \Phi_{k+1}^T-\Phi_k^T)\\
		&\leq\itkless\iikc\iuns\nu\left[\frac{1}{\efPB}\left(\abs{f(\Xkfun)-\Fsk}+\abs{f(\Xkf)-\Fok}\right)\right.\\
		&\left.+\frac{1}{m\efPB}\left(\abs{h(\Xkiun)-\Hsk}+\abs{h(\Xki)-\Hok}\right)\right]
		\end{split}
		\end{equation}
	\end{itemize}
	Since~\eqref{trentecinq} dominates~\eqref{trentedeux}, \eqref{trentetrois} and~\eqref{trentequatre}, then combining all four cases lead to
	\begin{equation}\label{trentecinqprime}
	\begin{split}
	\itkless\iikc(& \Phi_{k+1}^T-\Phi_k^T)\\
	&\leq\itkless\iikc\nu\left[\frac{1}{\efPB}\left(\abs{f(\Xkfun)-\Fsk}+\abs{f(\Xkf)-\Fok}\right)\right.\\
	&\left.+\frac{1}{m\efPB}\left(\abs{h(\Xkiun)-\Hsk}+\abs{h(\Xki)-\Hok}\right)\right]
	\end{split}
	\end{equation}
	Now, taking expectations with respect to $\fcf$ on both sides of~\eqref{trentecinqprime} and using~\eqref{Hcond1}, \eqref{condFok} and~\eqref{condFsk} lead to~\eqref{trentesix}. Then, by combining the main results of Case~1, Case~2 and Case~3 of Part~2, specifically~\eqref{vingtsix},~\eqref{trenteun} and~\eqref{trentesix}, the following holds
	\begin{equation}\label{trentesept}
	\begin{split}
	\itkless\Esp\left[\Phi_{k+1}^T-\Phi_k^T|\fcf\right] &\leq\left[ -\alpha\beta(1-\nu)(1-\tau^2)+2\nu(1-\alpha)^{1/2} \right.\\
	& \left.+4\nu(1-\beta)^{1/2}\right](\Dp)^2\itkless.
	\end{split}
	\end{equation}
	Finally, choosing~$\alpha$ and~$\beta$ according to~\eqref{betaChoice} ensures that 
	\begin{equation}\label{trenteseptprime}
	-\alpha\beta(1-\nu)(1-\tau^2)+2\nu(1-\alpha)^{1/2}+4\nu(1-\beta)^{1/2}\leq -\frac{1}{2}\alpha\beta(1-\nu)(1-\tau^2),
	\end{equation}
	and~\eqref{phiPart1Tlessk} obviously follows from~\eqref{trentesept} and~\eqref{trenteseptprime} with the same constant $\eta=\frac{1}{2}\alpha\beta(1-\nu)(1-\tau^2)$ as Part~1, which achieves the proof.
\end{proof}

\subsubsection*{Proof of Corollary~\ref{corHFok}}
\begin{proof}
	Only~\eqref{hfzero} is proved but the proof also applies for $\abs{\Hsk-h(X^k+S^k)}$ and $\abs{\Fsk-f(X^k+S^k)}$. According to Assumption~\ref{randomEstim}{\it (vi)}, $\E{\abs{\Hok-h(X^k)} |\ \fcf}\leq m\efPB(1-\alpha)^{1/2}(\Dp)^2$, which implies that
	%\begin{equation}
	%\nonumber
	%\end{equation}
	
	\begin{equation}\label{assup6}
	\E{\abs{\Hok-h(X^k)}}\leq m\efPB(1-\alpha)^{1/2}\Esp\left[(\Dp)^2\right].
	\end{equation}
	By summing each side of~\eqref{assup6} over $k$ from $0$ to $N$, and observing that 
	\begin{equation}
	0\leq S_{N}^h:=\sum_{k=0}^N\abs{\Hok-h(X^k)}\nearrow \sum_{k=0}^{+\infty}\abs{\Hok-h(X^k)},\ \ \text{and}\ \ 0\leq S_{N}^\Delta:=\sum_{k=0}^N(\Dp)^2 \nearrow\sum_{k=0}^{+\infty}(\Dp)^2,\nonumber
	\end{equation}
	then, it follows from the monotone convergence theorem~\cite{durrett2010probability} that
	\begin{eqnarray}
	\E{\sum_{k=0}^{+\infty}\abs{\Hok-h(X^k)}} &=&\E{\lim_{N\to+\infty} S_{N}^h}=\lim_{N\to+\infty}\E{S_{N}^h}=\sum_{k=0}^{+\infty}\E{\abs{\Hok-h(X^k)}}\nonumber \\
	&\leq& m\efPB(1-\alpha)^{1/2}\sum_{k=0}^{+\infty}\Esp\left[(\Dp)^2\right]= m\efPB(1-\alpha)^{1/2}\lim_{N\to+\infty}\E{S_{N}^\Delta} \nonumber\\
	&=&m\efPB(1-\alpha)^{1/2}\E{\lim_{N\to+\infty} S_{N}^\Delta}=m\efPB(1-\alpha)^{1/2}\Esp\left[\sum_{k=0}^{+\infty}(\Dp)^2\right]\nonumber\\
	&\leq& \mu\times m\efPB(1-\alpha)^{1/2}<+\infty,\nonumber
	\end{eqnarray}
	where $\mu$ is the constant of~\eqref{deltaPartT}. This means that $\ds{\sum_{k=0}^{+\infty}\abs{\Hok-h(X^k)}<+\infty}$ almost surely, 
	which implies the first result of~\eqref{hfzero}. The proof for $\abs{\Fok-f(X^k)}$ is similar by observing that (see~\eqref{condFok})
	\begin{equation}
	\E{\abs{\Fok-f(X^k)}|\fcf}\leq\efPB(1-\beta)^{1/2}(\Dp)^2.\nonumber
	\end{equation}
\end{proof}

\subsubsection*{Proof of Lemma~\ref{Psikh}}

\begin{proof}
	The proof uses ideas derived in~\cite{audet2019stomads,chen2018stochastic}. The result is proved by contradiction conditioned on the almost sure event $E_1=\{\Dp\to 0\}$. All that follows is conditioned on the event $E_1$. Assume that with nonzero probability, there exists a random variable $\mathcal{E}'>0$ such that  
	\begin{equation}\label{Psim}
	\Psi_k^h\geq \mathcal{E}', \quad\text{for all}\ k\in\N.
	\end{equation}
	Let $\{\xki\}_{k\in\N}$, $\{s^k\}_{k\in\N}$, $\{\dpl\}_{k\in\N}$ and $\epsilon'>0$  be realizations of $\{\Xki\}_{k\in\N}$, $\{S^k\}_{k\in\N}$,  $\{\Dp\}_{k\in\N}$ and $\mathcal{E}'$, respectively for which~\eqref{Psim} holds. Let $\hat{z}$ be the same parameter of Algorithm~\ref{algoPB} satisfying $\dpl\leq\dmax$ for all $k\geq 0$. Since $\dpl\to 0$ because of the conditioning on $E_1$, there exists $k_0\in\N$ such that
	\begin{equation}\label{deltalambda}
	\dpl<\lambda:=\min\left\lbrace \frac{\epsilon'}{m\efPB(\gamma+2)},\tau^{1-\hat{z}} \right\rbrace, \quad\text{for all}\ k\geq k_0.
	\end{equation}
	Consequently and since $\tau<1$, the random variable $R_k$ with realizations $r_k:=-\log_{\tau}\left(\frac{\dpl}{\lambda}\right)$ satisfies $r_k<0$ for all $k\geq k_0$. The main idea of the proof is to show that such realizations occur only with probability zero, thus leading to a contradiction. Let first show that $\{R_k\}_{k\in\N}$ is a submartingale. Let $k\geq k_0$ be an iteration for which the events $I_k$ and $J_k$ both occur, which happens with probability of at least $\alpha\beta>1/2$. Then, it follows from the definition of the event $I_k$ (see Definition~\ref{ikDef}) that
	\begin{eqnarray}
	h(\xki)&\leq&\uok(\xki)
	\leq  \sum_{j=1}^{m}\max\left\lbrace \cok(\xki),0 \right\rbrace +m\efPB(\dpl)^2=\hok(\xki)+ m\efPB(\dpl)^2,\quad\quad\quad \label{diff1}\\
	\text{and}\quad h(\xki+s^k)&\geq& \lsk(\xki+s^k)
	\geq\hsk(\xki+s^k)- m\efPB(\dpl)^2.\label{diff2}
	\end{eqnarray}
	
	\begin{equation}\label{zkid}
	\begin{split}
	\! \! \! \! \! \! \text{Hence,}\quad \hsk(\xki+s^k)&-\hok(\xki)=[h(\xki+s^k)-h(\xki)]+[h(\xki)-\hok(\xki)]\quad \\& +[\hsk(\xki+s^k)-h(\xki+s^k)]  \\
	&\leq 2m\efPB(\dpl)^2-\epsilon'\dpl\leq 2m\efPB(\dpl)^2-m\efPB(\gamma+2)(\dpl)^2=-\gamma m\efPB(\dpl)^2\quad\ \ \ 
	\end{split}
	\end{equation}
	where the first inequality in~\eqref{zkid} follows from~\eqref{Psim}, \eqref{diff1} and~\eqref{diff2} while the last one follows from~\eqref{deltalambda}. Consequently, the iteration $k$ of Algorithm~\ref{algoPB} can not be unsuccessful. Thus, the frame size parameter is updated according to $\delta_p^{k+1}=\tau^{-1}\dpl$ since $\dpl<\tau^{1-\hat{z}}$. Hence, $r_{k+1}=r_k+1$.
	
	Let $\fij=\sigma(I_0,I_1,\dots,I_{k-1})\cap\sigma(J_0,J_1,\dots,J_{k-1})$. For all other outcomes of $I_k$ and $J_k$, which will occur with a total probability of at most $1-\alpha\beta$, the inequality $\delta_p^{k+1}\geq\tau\dpl$ always holds, thus implying that $r_{k+1}\geq r_k-1$. Hence, 
	\begin{eqnarray}
	\E{\mathds{1}_{I_k\cap J_k}(R_{k+1}-R_k)|\fij}&=&\pr{I_k\cap J_k|\fij}\geq\alpha\beta \nonumber\\
	\text{and}\quad \E{\mathds{1}_{\overline{I_k\cap J_k}}(R_{k+1}-R_k)|\fij}&\geq&-\pr{\overline{I_k\cap J_k}|\fij}\geq\alpha\beta-1.\nonumber
	\end{eqnarray}
	Thus, $\E{R_{k+1}-R_k|\fij}\geq 2\alpha\beta-1>0$, implying that $\{R_k\}$ is a submartingale. The remainder of the proof is almost identical to that of the proof of the $\liminf$-type first-order result in~\cite{chen2018stochastic}.
	
	Now, let construct a random walk $W_k$ with realizations $w_k$ on the same probability space as $R_k$, which will serve as a lower bound on $R_k$. Define $W_k$ as in~\eqref{Wk} by
	%and which satisfies $\accolade{\underset{k\to+\infty}{\limsup}\ W_k=+\infty}$ almost surely. 
	\begin{equation}%\label{Wk}
	W_k= \sum_{i=0}^{k}(2\cdot\iii\iji-1),
	\end{equation}
	where the indicator random variables $\iii$ and $\iji$ are such that $\iii=1$ if $I_i$ occurs, $\iii=0$ otherwise, and similarly, $\iji=1$ if $J_i$ occurs while $\iji=0$ otherwise. Then following the proof of Theorem~\ref{trueIterations}, it is easy to notice that $\{W_k\}$ is a $\fij$-submartingale (see also~\cite{chen2018stochastic} for the same result), thus leading to the conclusion that
	%Notice that $\{W_k\}$ is a submartingale. Indeed, %(see also~\cite{chen2018stochastic} for the followings), 
	%
	%\begin{eqnarray*}
	%\E{W_k|\fij} &=& \E{W_{k-1}|\fij} + \E{2 \cdot \iik\ijk - 1|\fij}\\
	%&=& W_{k-1}+2 \E{\iik\ijk |\fij}- 1\\
	%&=& W_{k-1}+2 \pr{I_k\cap J_k|\fij}- 1\\
	%&\geq& W_{k-1},
	%\end{eqnarray*}
	%where the last inequality follows from the fact that $\alpha\beta>1/2$. Moreover, it is easy to notice that $\{W_k\}$ has no limit since $W_k$ has only $\pm 1$ increments. Thus, it follows from Theorem~\ref{durett} that
	$\accolade{\underset{k\to+\infty}{\limsup}\ W_k=+\infty}$ almost surely. Since by construction
	\begin{equation}
	r_k-r_{k_0} = -{\log}_\tau \left(\frac{\dpl}{\delta^{k_0}_p}\right) = k-k_0 \geq w_k-w_{k_0},\nonumber 
	\end{equation}
	then with probability one, $R_k$ has to be positive infinitely often. Thus, the sequence of realizations $r_k$ such that $r_k<0$ for all $k\geq k_0$ occurs with probability zero. Consequently, the assumption that $\Psi_k^h\geq \mathcal{E}'$ holds for all $k\in\N$ with a positive probability is false, which implies that~\eqref{liminfPsikh} holds.
\end{proof}

\subsubsection*{Proof of Theorem~\ref{resulth}}

\begin{proof}
	The theorem is proved using ideas derived in~\cite{AuDe09a,audet2019stomads}. Define the events $E_1$ and $E_2$ by 
	\begin{equation}
	E_1=\accolade{\omega\in\Omega:\Dp(\omega)\to 0 }\quad\text{and}\quad  E_2=\accolade{\omega\in\Omega:\exists K'(\omega)\subset\N\ \text{such that}\ {\lim}_{K'(\omega)}\Psi_k^h(\omega)\leq 0}.\nonumber
	\end{equation}
	Then $E_1$ and $E_2$ are almost sure due to Corollary~\ref{meshframeConv} and~\eqref{liminfPsikh} respectively. Let $\omega\in E_1\cap E_2$ be an arbitrary outcome and note that the event $E_1\cap E_2$ is also almost sure as countable intersection of almost sure events. Then $\lim_{K'(\omega)}\Dp(\omega)=0$. It follows from the compactness hypothesis of Assumption~\ref{lipschitzAssumption} that there exists $K(\omega)\subseteq K'(\omega)$ for which the subsequence $\{\Xki(\omega)\}_{k\in K(\omega)}$ converges to a limit $\hat{X}_{\inf}(\omega)$. Specifically, $\hat{X}_{\inf}(\omega)$ is a refined point for the refining subsequence $\{\Xki(\omega)\}_{k\in K(\omega)}$. Let $v\in T^H_{\mathcal{X}}(\hat{X}_{\inf}(\omega))$ be a refining direction for $\hat{X}_{\inf}(\omega)$. Denote by $V$ the random vector with realizations $v$, i.e., $v=V(\omega)$, and let $\hat{x}_{\inf}=\hat{X}_{\inf}(\omega)$, $\xki=\Xki(\omega)$, $\dpl=\Dp(\omega)$, $\dm=\Dm(\omega)$, $\psi_k^h=\Psi_k^h(\omega)$ and $\mathcal{K}=K(\omega)$. Since $v$ is a refining direction, then there exists $\mathcal{L}\subseteq \mathcal{K}$ and polling directions $d^k\in\mathbb{D}^k_p(\xki)$ such that $v=\underset{k\in \mathcal{L}}{\lim}\frac{d^k}{\norminf{d^k}}$. For each $k\in\mathcal{L}$,  define 
	\begin{equation}
	\begin{split}
	t_k&=\dm\norminf{d^k}\to 0,\quad\quad\quad\quad\quad\ \ \ y^k=\xki+t_k\left(\frac{d^k}{\norminf{d^k}}-v\right)\to \hat{x}_{\inf},\nonumber\\
	a_k&=\frac{h(y^k+t_k v)-h(\xki)}{t_k} \quad\text{and}\quad b_k= \frac{h(\xki)-h(y^k)}{t_k}, \nonumber
	\end{split}
	\end{equation}
	where the fact that $t_k\to 0$ follows from Definition~\ref{meshpollset}, specifically the inequality $\dm\norminf{d^k}\leq \dpl b$. Since $h$ is $\lambda^h$--locally Lipschitz, then  
	\begin{equation}
	\abs{a_k}\leq\frac{\lambda^h}{t_k}\norminf{(y^k+t_k v)-\xki}=\lambda^h \quad\text{and}\quad \abs{b_k}\leq \frac{\lambda^h}{t_k}\norminf{\xki-y^k} =\lambda^h\norminf{\frac{d^k}{\norminf{d^k}}-v}\to 0, \nonumber
	\end{equation}
	which shows that Lemma~\ref{akbk} applies for both subsequences $\{a_k\}_{k\in\mathcal{L}}$ and $\{b_k\}_{k\in\mathcal{L}}$. Moreover, combining the inequality $\lim_{\mathcal{L}}\psi_k^h\leq 0$ and  Assumption~\ref{dmin} (the fact that $\dpl\norminf{d^k}\geq d_{\min}>0$), yields
	\begin{equation}\label{psikpositive}
	\lim_{k\in\mathcal{L}}\left(\frac{-\psi_k^h}{\dpl\norminf{d^k}}\right)=\lim_{k\in\mathcal{L}} \frac{h(\xki+\dm d^k)-h(\xki)}{t_k}\geq -d_{\min}^{-1}\lim_{k\in\mathcal{L}}\psi_k^h \geq 0. 
	%=\lim_{k\in\mathcal{L}}a_k
	\end{equation}
	Thus, by adding and subtracting $h(\xki)$ to the numerator of the definition of the Clarke derivative, and using the fact that $\xki+\dm d^k\in\X$ for sufficiently large $k\in\mathcal{L}$ since $v$ is a hypertangent direction,
	\begin{eqnarray}
	h^{\circ}(\hat{x}_{\inf};v)&\geq& \limsup_{k\in\mathcal{L}}\frac{h(y^k+t_k v)-h(\xki)+h(\xki)-h(y^k)}{t_k}=\limsup_{k\in\mathcal{L}}(a_k+b_k) \nonumber\\
	&=& \limsup_{k\in\mathcal{L}}a_k+ \lim_{k\in\mathcal{L}}b_k= \limsup_{k\in\mathcal{L}} \frac{h(\xki+\dm d^k)-h(\xki)}{t_k}\geq 0, \nonumber
	\end{eqnarray} 
	where the last inequality follows from~\eqref{psikpositive}. Now, notice that it has been showed that every outcome~$\omega$ arbitrarily chosen in $E_1\cap E_2$, belongs to the event
	\begin{equation}
	\begin{split}
	E_3:= \left\lbrace \omega\in\Omega:\exists K(\omega)\subseteq\N\ \text{and}\ \exists\hat{X}_{\inf}(\omega)\right.&= \lim_{k\in K(\omega)}\Xki(\omega), \hat{X}_{\inf}(\omega)\in\X,\ \text{such that} \\ \nonumber
	\forall V(\omega)\in&\left. T^H_{\mathcal{X}}(\hat{X}_{\inf}(\omega)),\   h^{\circ}(\hat{X}_{\inf}(\omega);V(\omega))\geq 0\right\rbrace,  
	\end{split} 
	\end{equation}
	thus implying that $E_1\cap E_2\subseteq E_3$. Then the proof is complete by noticing that $\pr{E_1\cap E_2}=1$.
\end{proof}

\subsubsection*{Proof of Lemma~\ref{Psikf}}

\begin{proof}
	The proof is almost identical to those of Lemma~\ref{Psikh} and a similar result in~\cite{audet2019stomads}. Hence, full details are not provided here again. Unless otherwise stated, all the sequences, events and constants considered are defined as in the proof of Lemma~\ref{Psikh}. The result is proved by contradiction and all that follows is conditioned on the almost sure event $E_1\cap\{T<+\infty\}$. 
	Assume that with nonzero probability there exists a random variable $\mathcal{E}''>0$ such that
	\begin{equation}\label{Psiz}
	\Psi_k^{f,T}\geq \mathcal{E}'', \quad\text{for all}\ k\geq 0.
	\end{equation}
	Let $\{\xkft\}_{k\in\N}$, $\{s^k\}_{k\in\N}$, $\{\dpl\}_{k\in\N}$ and $\epsilon''>0$ be realizations of $\{\Xkft\}_{k\in\N}$, $\{S^k\}_{k\in\N}$, $\{\Dp\}_{k\in\N}$ and~$\mathcal{E}''$, respectively for which~\eqref{Psiz} holds. Let $\bar{k}_0\in\N^*$ be such that 
	\begin{equation}\label{deltalambda2}
	\dpl<\lambda:=\min\left\lbrace \frac{\epsilon''}{\efPB(\gamma+2)},\tau^{1-\hat{z}} \right\rbrace \quad\text{for all}\ k\geq \bar{k}_0.
	\end{equation}
	The key element of the proof is to show that an iteration $k\geq k_0:=\max\{\bar{k}_0,t\}$ for which the events $I_k$ and $J_k$ both occur can not be unsuccessful, thus leading to the fact that $\{R_k\}$ is a submartingale.
	
	It follows from~\eqref{Psiz} and~\eqref{deltalambda2} that 
	\begin{equation}
	f(\xkf+s^k)-f(\xkf)\leq -\epsilon''\dpl\leq-(\gamma+2)\efPB(\dpl)^2, \quad \text{for all}\ k\geq k_0. \nonumber
	\end{equation}
	
	\begin{equation}
	\begin{split}
	\!\!\!\!\!\!\!\!\!\! \text{Since $J_k$ occurs,}\quad \fsk(\xkf+s^k)-\fok(\xkf)&=[f(\xkf+s^k)-f(\xkf)]+[f(\xkf)-\fok(\xkf)]	\nonumber\\
	&  +[\fsk(\xkf+s^k)-f(\xkf+s^k)]\nonumber \\
	&\leq-(\gamma+2)\efPB(\dpl)^2+2\efPB(\dpl)^2=-\gamma\efPB(\dpl)^2,\nonumber
	\end{split}
	\end{equation}
	which implies that the iteration $k\geq k_0$ of Algorithm~\ref{algoPB} can not be unsuccessful.
\end{proof}

\subsubsection*{Proof of Theorem~\ref{xhatD}}

\begin{proof}
	The proof results from Corollary~\ref{corHFok} by observing that for all outcome $\omega$ in the almost sure event 
	\begin{equation}
	E_4:=\left\lbrace \omega\in\Omega:\forall K(\omega)\subseteq\N,\lim_{k\in K(\omega)}\abs{\Hok(\Xkft)(\omega)-h(\Xkft(\omega))}=0 \right\rbrace\cap\{T<+\infty\} ,\nonumber
	\end{equation}
	
	\begin{equation}
	\lim_{k\in K(\omega)}\abs{\Hok(\Xkft)(\omega)-h(\Xkft(\omega))}=\lim_{k\in K(\omega)}h(\Xkft(\omega))=h(\Xhatf(\omega))=0,\nonumber
	\end{equation}
	where the penultimate equality follows from the continuity of $h$ in~$\X$. This means that \[\pr{h(\Xhatf)=0}=\pr{\Xhatf\in\D}=1.\]
\end{proof}

\subsubsection*{Proof of Theorem~\ref{resultf}}

\begin{proof}
	First, notice that the fact that $\pr{\Xhatf\in\D}=1$ follows from Theorem~\ref{xhatD}. Then the proof easily follows from that of Theorem~\ref{resulth}, by replacing $h$ by $f$,  $\hat{x}_{\inf}=\hat{X}_{\inf}(\omega)$ by $\xhatf=\Xhatf(\omega)$, $\xki=\Xki(\omega)$ by $\xkft=\Xkft(\omega)$, $\psi_k^h=\Psi_k^h(\omega)$ by $\psi_k^{f,t}=\Psi_k^{f,T}(\omega)$ with $t=T(\omega)$ and  $T^H_{\X}(\cdot)$ by $T^H_{\D}(\cdot)$, for $\omega$ fixed and arbitrarily chosen in the almost sure event $E_1\cap E_5\cap\{T<+\infty \}$, where 
	\begin{equation}
	\begin{split}
	E_5= \left\lbrace \omega\in\Omega:\exists K(\omega)\subseteq\N\ \text{such that}\ \Xhatf(\omega)\right.&= \lim_{k\in K(\omega)}\Xkft(\omega),\  \Xhatf(\omega)\in\D, \\
	\lim_{k\in K(\omega)}\Psi_k^{f,T}(\omega)\leq 0 & \left. \ \text{and}\ \lim_{k\in K(\omega)}\Hok(\Xkft)(\omega) = 0 \right\rbrace.
	\end{split} 
	\end{equation}
\end{proof}

%------------------------------------------------%
\bibliographystyle{plain}
\bibliography{bibliography}
%\pdfbookmark[1]{References}{sec-refs}
%------------------------------------------------%

%------------------------------------------------%
\end{document}